\documentclass[12pt,letterpaper]{article}

\pdfoutput=1 

\usepackage[utf8]{inputenc}
\usepackage{amsmath, amsfonts, amssymb, amsthm}  
\usepackage[top=2.5cm, bottom=2.5cm, left=2.5cm, right=2.5cm]{geometry}
\usepackage{mathdots, mathrsfs}
\usepackage{titlesec}
\usepackage{enumitem}
\usepackage{hyperref}
\usepackage[nameinlink, capitalise]{cleveref} 
\crefname{equation}{}{}  
\crefname{figure}{Figure}{Figures}

\usepackage[bottom]{footmisc}
\usepackage{graphicx, svg}
\usepackage{wrapfig}
\usepackage{verbatim, textcomp}
\usepackage[font=small,labelfont=bf]{caption}
\usepackage{placeins} 
\usepackage{csquotes}
\usepackage{microtype}

\usepackage{subcaption}

\usepackage[
backend=biber,
date=year, 
url=true, 
eprint=false,
doi=false,
isbn=false,
style=numeric, 
citestyle=numeric 
]{biblatex}
\usepackage{xurl} 

\AtEveryBibitem{ 
    \clearfield{urlyear}
    \clearfield{urlmonth}
}
\DefineBibliographyStrings{english}{%
	mathesis = {Master's thesis},
}
\newbibmacro{string+doi}[1]{
    \iffieldundef{doi}{
    	\iffieldundef{url}{
    		#1}{
        		\href{\thefield{url}}{#1}
        	}
        } {
        	\href{http://dx.doi.org/\thefield{doi}}{#1}
        }
}

\DeclareFieldFormat{url}{ 
	\iffieldundef{eprint:arxiv}{
		\iffieldundef{eprint}{
			\iffieldundef{doi}{ 
				\mkbibacro{URL}\addcolon\space{ \url{#1}}
			}{} 
		}{ 
			arXiv \addcolon\space\href{https://arxiv.org/abs/\thefield{eprint}}{\thefield{eprint}}
		}
	}{ 
		arXiv \addcolon\space\url{\thefield{eprint:arxiv}}
	}
}

\DeclareFieldFormat{eprint}{%
	\iffieldundef{journaltitle}{
		\mkbibacro{URL}\addcolon\space\url{#1}%
	}{%
	}%
}

\DeclareFieldFormat[book]{title}{\usebibmacro{string+doi}{\mkbibemph{#1}}}
\DeclareFieldFormat[inproceedings]{title}{\usebibmacro{string+doi}{\mkbibquote{#1}}}
\DeclareFieldFormat[incollection]{title}{\usebibmacro{string+doi}{\mkbibquote{#1}}}
\DeclareFieldFormat[article]{title}{\usebibmacro{string+doi}{\mkbibquote{#1}}}

\usepackage{tikz}
\usetikzlibrary{calc}
\tikzset{  
	vertex/.style={circle, draw=black, thick,  fill=black, minimum size = 2mm, inner sep=0mm},
    smallVertex/.style={circle, draw=black, thick,  fill=black, minimum size = 1.4mm, inner sep=0mm},
    edge/.style={black,line width = .7mm, line cap=round},
    tube1/.style={blue,line width=.6mm, double distance=3mm,rounded corners=.1mm,line cap=round},
    tube2/.style={blue,line width=.6mm, double distance=5mm,rounded corners=.5mm,line cap=round},
    tube3/.style={blue,line width=.6mm, double distance=7mm,rounded corners=.1mm,line cap=round},
    chord/.style={line width = .5mm, black, bend angle=70,bend left,line cap=round},
	none/.style={}
}

\titleformat{\section}{\normalsize\scshape\center}{\thesection}{1em}{}
\titleformat{\subsection}{\normalsize\scshape\center}{\thesubsection}{1em}{}

\makeatletter
\newtheoremstyle{indented} 
  {3pt}
  {3pt}
  {\addtolength{\@totalleftmargin}{2em}
   \addtolength{\linewidth}{-2em}
   \parshape 1 2em \linewidth}
  {}
  {\bfseries}
  {.}
  {.5em}
  {}
\makeatother

\theoremstyle{definition}
\newtheorem{definition}{Definition}[section]

\theoremstyle{indented}
\newtheorem{example}[definition]{Example}
\newtheorem{remark}[definition]{Remark}

\theoremstyle{plain}
\newtheorem{theorem}[definition]{Theorem}
\newtheorem{proposition}[definition]{Proposition}
\newtheorem{lemma}[definition]{Lemma}

\usepackage{ifthen}   
\usepackage{xstring}  
\usepackage{tikz}
\usetikzlibrary{decorations.pathreplacing}

\newcommand*{\IsInteger}[3]{
	\IfStrEq{#1}{ }{%
		#3
	}{%
		\IfInteger{#1}{#2}{#3}%
	}%
}%

\newcommand{\rtLine}[1]{
	\IsInteger{#1}{
		\ifthenelse{ #1=1}
		{\bullet}
		{\ifthenelse {#1=2}
			{\tikz[baseline=0]  {\draw [thick](0,-.05) -- (0,.2 ); \node [] at (0,-.05) {\textbullet};\node [] at (0,.2) {\textbullet}; }}
			{\ifthenelse {#1=3}
				{\tikz[baseline=0]   {\draw [thick](0,-.15) -- (0,.3 ); \node [] at (0,-.15) {\textbullet};\node [] at (0,.3) {\textbullet};\node [] at (0,.07) {\textbullet}; }}
				{\ifthenelse {#1=4}
					{\tikz[baseline=0]   {\draw [thick](0,-.25) -- (0,.4 ); \node [] at (0,-.25) {\textbullet};\node at (0,-.05) {\textbullet}; \node [] at (0,.15) {\textbullet}; \node  at (0,.35) {\textbullet}; }}
					{ \tikz[baseline=0]   {\draw [thick](0,.2) -- (0,.4 ); \node [] at (0,-.25) {\textbullet};\node at (0,-.05) {$\cdot$};\node at (0,.05) {$\cdot$}; \node [] at (0,.2) {\textbullet}; \node  at (0,.4) {\textbullet}; \node at (.4, .1){$#1$};\draw [decorate,decoration={brace}] (.15,.45) -- (.15,-.25); } }
				}
			}
		}
	}{
	 \tikz[baseline=0]   {\draw [thick](0,.2) -- (0,.4 ); \node [] at (0,-.25) {\textbullet};\node at (0,-.05) {$\cdot$};\node at (0,.05) {$\cdot$}; \node [] at (0,.2) {\textbullet}; \node  at (0,.4) {\textbullet}; \node at (.4, .1){$#1$};\draw [decorate,decoration={brace}] (.15,.45) -- (.15,-.25); }
	}
}

\newcommand{\rtV}{{
	\tikz[baseline=0] {
		\draw [thick](-.15,-.05) -- (0,.2 )--(.15,-.05); 
		\node [] at (-.15,-.05) {\textbullet}; 
		\node [] at (0,.2) {\textbullet};  
		\node [] at (.15,-.05) {\textbullet}; 
	}
}}
\newcommand{\rtW}{{
	\tikz[baseline=0]{
		\draw [thick](-.15,-.1) -- (0,.2 )--(.15,-.1); 
		\draw[thick] (0,.2)--(0,-.1); 
		\node [] at (-.2,-.1) {\textbullet}; 
		\node [] at (0,.2) {\textbullet};  
		\node [] at (.2,-.1) {\textbullet}; 
		\node [] at (0,-.1) {\textbullet};  
	}
}}

\newcommand{\Aut}{\operatorname{Aut}}

\newcommand{\CK}{\mathcal H}
\newcommand{\Tub}{\mathrm{Tub}}
\newcommand{\od}[2][]{\operatorname{od}_{#1}(#2)}
\newcommand{\id}{\mathrm{id}}
\newcommand{\ZZ}{\mathbb Z}
\newcommand{\rootv}{\operatorname{rt}}

\usepackage{tocloft}
\usepackage{xparse} 
\usepackage{xstring}
\ExplSyntaxOn
\newcounter{todo}
\newcommand{\listtodoname}{ {\large \color{red} List ~of~ TODOs}}
\newlistof {todos}{tod}{\listtodoname}
\crefname{todo}{{\color{red}TODO}}{{\color{red}TODOs}}
\NewDocumentCommand{\todo}{o o m}{
    \refstepcounter{todo}  
    \addcontentsline{tod}{todos}{ \protect\numberline{\thetodo}   {  \protect\StrLeft{#3}{40}\protect\ldots} } 
    \noindent{ \color{red} { [ \textbf{TODO\; \thetodo} \IfNoValueTF{#1}{}{\tl_if_blank:nTF{#1}{}{ {\footnotesize {~by~ #1} }}}
    \IfNoValueTF{#2}{}{~for~ \textbf{#2}}\textbf{:} ~
    #3 ] } }   
}
\ExplSyntaxOff

\begin{document}
\title{\uppercase{\textbf{\normalsize Tubings, chord diagrams, and Dyson--Schwinger equations}}}
\author{\small{\textsc{Paul-Hermann Balduf, Amelia Cantwell, Kurusch Ebrahimi-Fard,}} \\ \small{\textsc{Lukas Nabergall, Nicholas Olson-Harris, and Karen Yeats}}\thanks{We thank Guillaume Laplante-Anfossi for instructive comments regarding the polytope connection, and  the referee for their detailed reading. PHB thanks KY for a stay at U Waterloo in early 2022, during which many of the results of this work were first discovered. KY is supported by an NSERC Discovery grant and the Canada Research Chairs program.  KY thanks the Perimeter Institute for its support and NTNU Trondheim for its hospitality on a 2018 visit which ultimately led to this project. Research at Perimeter Institute is supported in part by the Government of Canada through the Department of Innovation, Science and Economic Development Canada and by the province of Ontario through the Ministry of Economic Development, Job Creation and Trade. This research was also supported in part by the Simons Foundation through the Simons Foundation Emmy Noether Fellows Program at Perimeter Institute.\indent Addresses: PHB, AC, LN, NOH, KY: C\&O Dept., University of Waterloo, Waterloo, ON, Canada.\\
\indent KEF: Dept.~of Mathematical Sciences, Norwegian University of Science and Technology, Trondheim, Norway.
}}

\date{\small{\textsc{\today}}}

\maketitle

\begin{abstract}
 We give series solutions to single insertion place propagator-type systems of Dyson--Schwinger equations using binary tubings of rooted trees.  These solutions are combinatorially transparent in the sense that each tubing has a straightforward contribution.  The Dyson--Schwinger equations solved here are more general than those previously solved by chord diagram techniques, including systems and non-integer values of the insertion parameter $s$.  We remark on interesting combinatorial connections and properties.
\end{abstract}

\allowdisplaybreaks


\section{Introduction}

Dyson--Schwinger equations \cite{dyson_matrix_1949,schwinger_green_1951a,schwinger_green_1951} 
play a fundamental role in quantum field theory, where they are understood as quantum equations of motion. In mathematical terms, they form a set of coupled integral equations describing relations between Green's functions, i.e.~correlation functions of quantum fields. See \cref{sec physical set up,sec mellin transform} for the physical background. Finding exact solutions to Dyson--Schwinger equations is extremely challenging. Therefore various perturbative expansions are employed, based on diagrammatical tools such as Feynman graphs \cite{feynman_spacetime_1949,itzykson_quantum_2005,broadhurst_renormalization_1999}, or chord diagrams \cite{nabergall_combinatorics_2023, marie_chord_2013, hihn_generalized_2019, courtiel_terminal_2017, courtiel_connected_2019, courtiel_nexttok_2020}. 

In the present paper, we introduce a new series expansion of Dyson-Schwinger equations, based on a certain notion of tubings on rooted trees. We will discuss how this notion of tubings on the one hand lets us give series solutions to a wide class of Dyson--Schwinger equations in quantum field theory, and on the other hand relates to rooted connected chord diagrams.  
The tubing approach improves on the Feynman graph expansion in that each tubing has a simple and combinatorially controlled contribution to the sum. Furthermore, it improves upon the chord diagram expansion by being generalizable, with transparent proofs.



Our tubings are \emph{binary tubings}, as will be defined in \cref{def tubing}, because each tube of size at least 2 will be partitioned into exactly two smaller tubes.  This is a special case of Galashin's notion of \textit{tubings} or \textit{pipings} of posets \cite{galashin_associahedra_2023} but we came to it independently.  Additionally, our notion of tubing is different from, but related to, the notion of tubings of trees and graphs that are used to define graph associahedra, see \cite{carr_coxeter_2006}.  We will discuss the connection between these different notions of tubings further in   \cref{sec associahedra}.  Outside of those discussions, we will only be working with binary tubings, so we will simply refer to tubings meaning binary tubings.

Using these tubings we will be able to both clarify and surpass the previous theory of chord diagram expansions of Dyson--Schwinger equations. However, the connection between tubings and rooted connected chord diagrams remains interesting as we will discuss in \cref{sec chord}.  This connection is purely combinatorial.

The most general single Dyson--Schwinger equation (DSE) that we give series solutions for has the form 
\begin{align}\label{dse_differential_general}
    G(x,L) = 1 + \sum_{k\geq 1}x^kG\left(x, \frac{\partial}{\partial\rho}\right)^{1+sk}(e^{L\rho}-1)F_k(\rho)\bigg|_{\rho=0}.
\end{align}
We also solve certain systems of Dyson--Schwinger equations of a similar form; see \cref{sec systems}.  From a combinatorial perspective, if the $F_k(\rho)$ are taken as given formal Laurent series with first order poles, then this equation recursively defines $G(x,L)$ as a bivariate formal power series, and so from this perspective \cref{dse_differential_general} can be seen in the spirit of algebraic enumeration with $G(x,L)$ ultimately as a generating series for tubings.

Note that conventions with regard to signs vary in the literature. The physical motivation for this equation along with the meaning of the arguments and parameters is explained in \cref{sec mellin transform}.  Unlike in previous work on chord diagram expansions, here $s\in \mathbb{R}$ is no longer restricted to certain integers (the sign of $s$ is also reversed compared to this previous work).  The bulk of this paper is devoted to proving that this equation has a series solution given by tubings of trees, as stated in \cref{thm main thm}.  

In \cref{sec definitions},  we review the necessary background, both physical and algebraic.

We look at a few special cases of tubings in \cref{sec examples}.  
Then \cref{sec tubing feynman rules} proves how the structure of Feynman rules implies what the contribution of each tubing must be, and brings this together to prove the main theorem solving the Dyson--Schwinger equation by a tubing expansion.

An extremely nice feature of this solution, and one of the original motivations for this project, is that it is not just an alternate expansion for the Green functions which contrasts with the Feynman graph expansion in the simple contributions of each object, but, moreover, each Feynman graph maps to a particular set of tubings which gives the same contribution, as explained in \cref{rem graph by graph}. 

In \cref{sec chord} we return to the chord diagrams and prove a bijection between tubed rooted plane trees and rooted connected chord diagrams and verify that under this bijection our tubing expansions of solutions to Dyson--Schwinger equations include the previously known chord diagram expansions of \cite{marie_chord_2013, hihn_generalized_2019}.  We also look at a few combinatorially interesting special cases. 

While our interest in these objects is in this use to solve Dyson--Schwinger equations in quantum field theory, they and related objects are also interesting as pure combinatorics.  We overview this in \cref{sec associahedra} notably including the relation to the other notion of tubings that exists in the literature and to associahedra.  The results of \cref{sec associahedra} are not new, but are of interest both to physicists and mathematicians and are included in order to link our results to some other areas of contemporary interest.

We conclude in \cref{sec conclusion}.

\section{Definitions and notation}\label{sec definitions}

\subsection{Rooted trees and binary tubings}\label{sec rooted trees and binary tubings}

We will have cause to work with both rooted trees and plane rooted trees.  

\begin{definition}\label{def rooted tree}
    Recursively, a \emph{rooted tree} is a vertex $r$, called the root, and a possibly empty multiset of rooted trees whose roots are the children of $r$.  A \emph{plane rooted tree} is a vertex $r$, also called the root, and a possibly empty ordered list of plane rooted trees whose roots are the children of $r$. 
\end{definition}

 By forgetting the order structure at each vertex, every plane rooted tree has an underlying rooted tree structure.  For the tubing expansions of Dyson--Schwinger equations it suffices to work without a plane structure, though that part can all be done at the plane level, but for the connection to chord diagrams it will be more convenient to work with plane rooted trees.

Write $\rootv(t)$ for the root vertex of a tree $t$. A vertex with no children in a rooted tree or a plane rooted tree is known as a \emph{leaf}.  We will write $\od v$ for the number of children of the vertex $v$.  An $n$-ary rooted tree is a tree where all vertices have at most $n$ children. 

We can view a rooted tree or a plane rooted tree as a graph with the extra information that the vertex set of the graph is the same as the vertex set of the tree and the edge relation is given by the parent-child relation.  We will use both graph language and parent-child language on rooted trees and plane rooted trees without further definition; for instance, we will speak of paths in trees and connectivity meaning the corresponding graph notions, and we will also speak of the descendants of a vertex as the set containing the vertex's children, its children's children, etc.  An \emph{automorphism} of a rooted tree is an automorphism of it seen as a graph which also fixes the root; denote the group of automorphisms of the rooted tree $t$ by $\operatorname{Aut}(t)$.  We will draw rooted trees with the root vertex at the top and the leaves on the bottom, see \cref{fig:tree_introduction}.

The \emph{rooted subtree} at the vertex $v$ of a rooted tree or a rooted plane tree is the subtree consisting of $v$ and all its descendants.  The rooted subtree inherits a rooted tree or rooted plane tree structure from the outer tree.  Furthermore, if $t$ is a rooted tree or a rooted plane tree and $t'$ a rooted subtree, then $t\setminus t'$ is also connected and hence is itself a rooted tree or rooted plane tree, but not a rooted subtree in the sense above.

\medskip

We will also have cause to work with trees with decorated or weighted vertices.  First, for a rooted
tree or plane rooted tree, the \emph{size} of the tree is the number of vertices.  A \emph{decorated
rooted tree} or \emph{decorated plane rooted tree}, $t$, decorated from the set $D$ is a rooted tree
or a plane rooted tree respectively with a function $d\colon V(t) \rightarrow D$. An
\textit{automorphism} of a decorated rooted tree is an automorphism of the underlying rooted tree
that preserves the decorations of all vertices; again the automorphism group of the rooted tree $t$ is denoted $\Aut(t)$.

A \emph{weighted rooted tree} or \emph{weighted plane rooted tree} $t$ is a decorated rooted tree or
a decorated plane rooted tree, respectively, where the set of decorations is $\mathbb{Z}_{\geq 1}$. In
this case we use slightly different language, writing $w\colon V(t)\rightarrow \mathbb{Z}_{\geq 1}$
for the function and calling $w(v)$ the \emph{weight} of the vertex $v$. (Distinguishing decoration
from weight will become essential when considering systems in \cref{sec systems} as the weight will
then be only part of the decoration.) Whenever useful we will identify the unweighted rooted tree or
plane rooted tree with the weighting of the same tree that gives weight 1 to each vertex.  The
\emph{weight} of a weighted rooted tree or weighted plane rooted tree is the sum of the weights of
the vertices.  For an unweighted tree the weight and size agree.

\medskip

\begin{definition}\label{def tube}
Let $t$ be a rooted tree or a plane rooted tree.  A \emph{tube} of $t$ is a set of vertices of $t$ that induces a connected subgraph of $t$. 
\end{definition}
A connected subgraph of a tree $t$ is a tree, hence is a subtree in an unrooted sense, but we are reserving the term subtree for rooted subtree in the sense given above.

We will draw tubes by drawing a curve around the relevant vertices, as in \cref{fig:tree_introduction}.

\begin{definition}\label{def tubing}
Let $t$ be a rooted tree or a plane rooted tree.  A \emph{binary tubing} $\tau$ of $t$ is a set of tubes such that 
\begin{enumerate}
    \item $\tau$ contains the tube of all vertices of $t$, and
    \item each tube $b$ of $\tau$ either is a single vertex or there are two other tubes in $\tau$ which partition $b$.
\end{enumerate}
\end{definition}

\begin{remark}
   More generally we can consider tubings of arbitrary partially ordered sets. In this context a tube is a connected \textit{convex} subset of a poset, which is automatic in the case of trees. This notion appears in Galashin's work on the $P$-associahedron \cite{galashin_associahedra_2023}, where binary tubings are known as maximal improper tubings and are in bijection with the maximal \textit{proper} tubings that index the vertices of the polytope. The combinatorics of tubings of posets in general and the special case of trees have been studied by Nguyen and Sack \cite{nguyen_poset_2023a, nguyen_poset_2023}.
\end{remark}

To improve readability, we will generally not draw the outermost tube, which encircles the full tree, nor the innermost tubes around each individual vertex, see \cref{fig:tubings}.

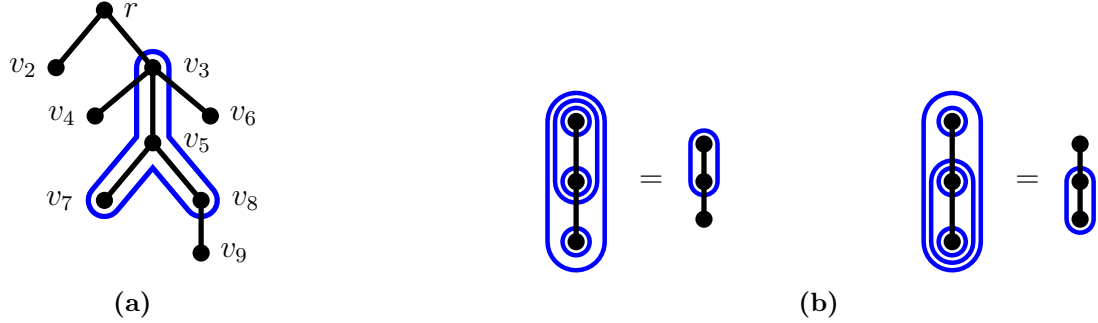
\begin{figure}[htb]
    \centering
    \begin{subfigure}[b]{.3\linewidth}
        \centering
        \begin{tikzpicture}

\coordinate (v1) at (0,.3){};
\coordinate (v2) at ($(v1) + (230:1)$){};
\coordinate (v3) at ($(v1) + (310:1)$){};
\coordinate (v4) at ($(v3) + (220:1)$){};
\coordinate (v5) at ($(v3) + (270:1)$){};
\coordinate (v6) at ($(v3) + (320:1)$){};
\coordinate (v7) at ($(v5) + (230:1)$){};
\coordinate (v8) at ($(v5) + (310:1)$){};
\coordinate (v9) at ($(v8) + (270:.7)$){};

\draw[blue,line width=5mm, line cap=round] (v3) -- (v5) -- (v7);
\draw[blue,line width=5mm, line cap=round] (v5) -- (v8);
\draw[white,line width=3.6mm, line cap=round] (v3) -- (v5) -- (v7);
\draw[white,line width=3.6mm, line cap=round] (v5) -- (v8);

\draw[edge] (v1)--(v2);
\draw[edge] (v1)--(v3)--(v4);
\draw[edge] (v3) -- (v6);
\draw[edge] (v3) -- (v5) -- (v7);
\draw[edge] (v5) -- (v8) -- (v9);
    
\node [vertex,label=right:$r$] at (v1){};
\node [vertex,label=left:$v_2$] at (v2){};
\node [vertex,label={[label distance=1.5mm]right:$v_3$}] at (v3){};
\node [vertex,label=left:$v_4$] at (v4){};
\node [vertex,label={[label distance=1.5mm]right:$v_5$}] at (v5){};
\node [vertex,label=right:$v_6$] at (v6){};
\node [vertex,label={[label distance=1.5mm]left:$v_7$}] at (v7){};
\node [vertex,label={[label distance=1.5mm]right:$v_8$}] at (v8){};
\node [vertex,label=right:$v_9$] at (v9){};
        \end{tikzpicture}
        \caption{}
        \label{fig:tree_introduction}
    \end{subfigure}
    \hfill 
    \begin{subfigure}[b]{.6\linewidth}
        \centering
        \begin{tikzpicture}

\coordinate (v3) at (3,-.3){};
\coordinate (v2) at (3,.5){};
\coordinate (v1) at (3,1.3){};
\draw [tube3] (v1) -- (v3);
\draw [tube2] (v1)--(v2);
\draw [blue, line width=.6mm] (v1) circle (1.8mm);
\draw [blue, line width=.6mm] (v2) circle (1.8mm);
\draw [blue, line width=.6mm] (v3) circle (1.8mm);
\draw [edge] (v1) to (v2) to (v3);
\node [vertex] at (v1){};
\node [vertex] at (v2){};
\node [vertex] at (v3){};

\node at (4,.5){$=$};

\coordinate (v3) at (4.7,0){};
\coordinate (v2) at (4.7,.5){};
\coordinate (v1) at (4.7,1){};
\draw[tube1] (v1)--(v2);
\draw [edge] (v1) to (v2) to (v3);
\node [vertex] at (v1){};
\node [vertex] at (v2){};
\node [vertex] at (v3){};

\coordinate (v3) at (8,-.3){};
\coordinate (v2) at (8,.5){};
\coordinate (v1) at (8,1.3){};
\draw [tube3] (v1) -- (v3);
\draw [tube2] (v2) -- (v3);
\draw [blue, line width=.6mm] (v1) circle (1.8mm);
\draw [blue, line width=.6mm] (v2) circle (1.8mm);
\draw [blue, line width=.6mm] (v3) circle (1.8mm);
\draw [edge] (v1) to (v2) to (v3);
\node [vertex] at (v1){};
\node [vertex] at (v2){};
\node [vertex] at (v3){};

\node at (9,.5){$=$};

\coordinate (v3) at (9.7, 0){};
\coordinate (v2) at (9.7, .5){};
\coordinate (v1) at (9.7, 1){};
\draw[tube1] (v2)--(v3);
\draw [edge] (v1) to (v2) to (v3);
\node [vertex] at (v1){};
\node [vertex] at (v2){};
\node [vertex] at (v3){};
        \end{tikzpicture}
        \caption{}
        \label{fig:tubings}
    \end{subfigure}
    
    \caption{ \textbf{(a)}  An unweighted rooted tree  of size 9. Vertex $r$ is the root, vertices $v_2,v_4,v_6,v_7$, and $v_9$ are leaves. The out degrees are  $\od{r}=2, \od{v_2}=0, \od{v_3}=3$ etc.  One of the possible tubes is indicated in blue, consisting of $\{v_3,v_5,v_7,v_8\}$. \textbf{(b)} The two possible binary tubings of a ladder with three vertices according to \cref{def tubing}. Later, we only draw the non-trivial part and leave out the innermost and outermost tubes, as shown to the right in the equations.}
\end{figure}

Note that any two tubes in a binary tubing are either disjoint or one is contained within the other, so binary tubings have a nested structure.

Our notion of binary tubing also appears under the name of maximal nesting in \cite{ward_massey_2022}.  Note however, that the definition of nesting is based on subgraphs defined by their edge sets, so while this agrees with our definition when restricted to rooted trees, the obvious generalization of our definition to graphs (which also comes up in physics, see \cref{rem tubings in BCFW}) is different from the notion of maximal nesting on a graph as there the vertex-based and edge-based definitions differ.  See \cref{sec associahedra} for further discussion.

More standard in the literature is a notion of tubings used in defining graph associahedra \cite{carr_coxeter_2006}.  These tubings are also made up of nested families of tubes but with different conditions on them.  Other than to compare our notions with theirs and further summarize some nesting results, see \cref{sec associahedra}, we will only work with binary tubings in this paper, and so we will just use the term \emph{tubing} for binary tubings.

We can also treat our tubings recursively.  Let $t$ be a rooted tree or a plane rooted tree.  A (binary) tubing $\tau$ of $t$ can be obtained recursively as follows:
\begin{enumerate}
    \item $\tau$ contains the outer tube consisting of all vertices of $t$.
    \item Let $e=uv \in E(t)$.  Partition the vertices $V(t) = \{A|B\}$ such that $B$ is the vertices of the subtree rooted at $v$.  Said another way, $B$ induces  a connected rooted subtree with $v$ as the root and contains all of $v$'s descendants and $A$ induces a connected subgraph that contains the root of $t$.  This subgraph is itself a rooted tree and has $u$ as a leaf.  Set $A$ and $B$ to be tubes of $\tau$.
    \item Proceed likewise to construct tubings of the trees induced by $A$ and $B$.
\end{enumerate}
Every tubing of $t$ can be constructed as above.
Expanding on this in another way, to find all the tubings of a tree, for every edge in $t$, we ``break'', or cut, the edge and put the lower vertex and all of its descendants into one tube and the rest of the tree in another tube.  Then for each tube and each edge in the tube, do the same  until all of the tubes contain only one vertex.  In other words, a tubing defines a binary tree structure on the edges of the original tree based on which stage in the hierarchy of the tubes we cut the edge.  This relates to the notion of the linearized coproduct in the Connes--Kreimer Hopf algebra, compare \cref{rem tubings in BCFW,rem linearized coprod,rem graph by graph}.

\begin{remark}\label{rem tubings binary tree}
    The containment structure of any binary tubing $\tau$ of a tree $t$ gives a binary rooted tree $B$ with left and right children distinguished.   To turn a tubing $\tau$ into such a binary tree, associate a vertex to each tube of $\tau$ and let the children of a vertex corresponding to a tube of size $>1$ be the two smaller tubes which partition it, the one containing the root of the outer tube being the right child and the other the left child.  Furthermore, $B$ is a \emph{full} binary rooted tree, that is, every vertex has either 2 or 0 children since each tube is either the tube of a single vertex or is partitioned into two nonempty subtubes.  The number of leaves of $B$ is the number of tubes of $\tau$ consisting of a single vertex which is the number of vertices of $t$.   The size of a tube is the number of vertices in the subtree of $B$ rooted at the vertex corresponding to the tube. A particular example will be discussed in \cref{rem ladders trees}.
\end{remark}

\begin{lemma}\label{lem recursive tubing}
	\begin{enumerate}
		\item Let $t$ be a rooted tree or a plane rooted tree and $\tau$ a tubing of $t$.  Then if $b$ is any tube of $\tau$, the tree induced by the vertices of $b$ has a tubing given by $b$ and those tubes of $\tau$ which are properly contained in $b$.
		\item If $t$ has $n$ vertices, then any tubing of $t$ consists of exactly $2n-1$ tubes.
	\end{enumerate}
\end{lemma}

\begin{proof}
1. is immediate from \cref{def tubing}. \\
2. A tubing of a single vertex has $2\cdot 1-1=1$ tube. Use induction on the number of vertices. Let $t$ be a rooted tree or plane rooted tree with $n>1$ vertices. Then by \cref{def tubing} its outermost tube $b$ contains all of $t$. Inside of $b$ there are two other tubes, containing $n_1$ and $n_2$ vertices such that $n_1+n_2=n$. By point 1 of the present lemma, these tubes induce tubings of subtrees, which, by the induction hypothesis, have $2n_1-1$ and $2n_2-1$ tubes. Together with the outermost tube $b$, the tree $t$ has $1+2n_1-1 + 2n_2-1  = 2n-1$ tubes as claimed.
\end{proof}

\begin{remark}\label{rem tau prime and tau double prime}
\begin{enumerate}

\item In view of the last lemma, a tubing is either 
\begin{enumerate}
    \item
        the unique tubing of $t = \bullet$, or 
    \item
        determined by $(\tau', \tau'')$ where $\tau'$ is a tubing of some
        proper rooted subtree $t'$ and $\tau''$ is a tubing of $t''= t \setminus t'$.  Note that $t''$ contains the root of $t$.
\end{enumerate}

\item Tubings of weighted rooted trees and weighted plane rooted trees are simply tubings of the underlying unweighted trees.  The weights are carried along harmlessly.
\end{enumerate}
\end{remark}

\medskip

Generally given a tubing $\tau$ of the tree $t\neq \bullet$ we will write $t'$ and $t''$ for the two rooted trees induced by the tubes that partition the outermost tube of $\tau$ with the root of $t$ being in $t''$. The sub-tubings $\tau'$ and $\tau''$ denote the tubings of the subtrees $t'$ and $t''$, respectively, as in \cref{rem tau prime and tau double prime}.

\begin{remark}\label{rem tubings in BCFW}
    By \cref{rem tau prime and tau double prime}, a tubing is a graphical representation of the sequence in which all edges of a given tree $t$ are cut, until only disconnected vertices remain. In our case, the tree will be the insertion tree of a Feynman graph, and cutting this tree edge by egde can be interpreted as the linearization of the (Connes--Kreimer) coproduct to be discussed in \cref{sec algebraic set up}. A tubing represents an iteration of the linearized coproduct, see \cref{rem graph by graph}.

    A similar procedure of iteratively splitting graphs into two components is familiar from the Britto--Cachazo--Feng--Witten (BCFW) recursion relation \cite{britto_new_2005,britto_direct_2005} of scattering amplitudes. Indeed, following \cite{arkani-hamed_cosmological_2017}, a graphical notation for splitting Feynman graphs has been used which seemingly coincides with the binary tubings in the present work. 

    Despite the similarity in graphical appearance, we would like to stress the striking difference between the two concepts: In the  setting of \cite{britto_new_2005,britto_direct_2005,arkani-hamed_cosmological_2017}, the recursive splitting of a Feynman graph reflects particular analytic properties of the corresponding amplitude. Conversely, our binary tubings operate on the insertion trees, not on the Feynman graphs themselves, and ultimately arise from the combinatorics of renormalization, considering the scale dependence of the resulting amplitude. The physical interpretation will be outlined in \cref{sec physical set up}.
\end{remark}

There are a few statistics on tubings that we will need.  Write $\Tub(t)$ for the set of tubings of $t$, and let $N(t) = |\Tub(t)|$ be the number of tubings of $t$.  Following \cref{lem recursive tubing} and \cref{rem tau prime and tau double prime}, this number is determined recursively:
\begin{lemma}\label{lem recursive tubing count}
    Let $t$ be a rooted tree with edge set $E(t)$. Removing any edge $e\in E(t)$ splits $t$ into two rooted trees $t' \cup t''$. The number   of tubings of $t$ is 
\begin{align}
    N(t) = \sum_{e\in E(t)} N(t') N(t'').
\end{align}
Moreover, $N(t_1)=N(t_2)$ whenever $t_1$ and $t_2$ are two rooted trees whose underlying graphs are isomorphic and hence only differ in which of the vertices is the root.
\end{lemma}

\begin{proof}
 By \cref{lem recursive tubing}, the remainder of $\tau$, after removing the edge $e$ which defines the bipartition of the outermost tube, is exactly a tubing of two trees $t',t''$. Summing over all possible ways to split the outermost tube, that is, over all $e\in E(t)$, reproduces the total number of tubings. Independence of the position of the root follows inductively, or can also be seen from the fact that the original definition of binary tubings does not make reference to the root.
\end{proof}

\medskip

\begin{definition}\label{def b statistic}
Let $\tau$ be a tubing of a rooted tree or plane rooted tree $t$.  Given a vertex $v$ of $t$, the \emph{b-statistic} $b(v, \tau)$ denotes the number of tubes of $\tau$ for which $v$ is the root of the tube.  Let $b(\tau) = b(r,\tau)$ where $r$ is the root of $t$.
\end{definition}

For example, in \cref{fig:tubings} we have $b(\tau)=3$ for the ladder on the left and $b(\tau)=2$ for the same ladder on the right.  For $v$ the middle vertex of the ladder we have $b(v,\tau)=1$ on the left but $b(v,\tau)=2$ on the right.  For both trees the $b$-statistic of the leaf is 1.  As we use the more condensed drawings with the outer and innermost tubes suppressed going forward, keep in mind that those tubes do still count for the $b$-statistic.

\medskip

Using the $b$-statistic we define the Mellin monomial of a tubing.

\begin{definition}\label{def Mellin monomial}
Given a tubing $\tau$ of a rooted tree or a plane rooted tree $t$ and a sequence $c_0, c_1,c_2, \ldots$ define the \emph{Mellin monomial} to be 
\[
    c(\tau) = \prod_{\genfrac{}{}{0pt}{2}{v\in V(t)}{ v \ne \rootv(t)}} c_{b(v,\tau)-1}.
\]
Given a tubing $\tau$ of a $D$-decorated rooted tree or a $D$-decorated plane rooted tree $t$ and a doubly
index sequence $(c_{i,d})_{\genfrac{}{}{0pt}{2}{0\leq i}{d \in D}}$, define the \emph{Mellin monomial} to be
\[
    c(\tau) = \prod_{\genfrac{}{}{0pt}{2}{v\in V(t)}{ v \ne \rootv(t)}} c_{b(v,\tau)-1,d(v)} 
\]
\end{definition}

In our application the sequence $c_0, c_1, \ldots$ will be the coefficient sequence of the series expansion of a Mellin transform of a regularized Feynman integral of a primitive Feynman graph, see \cref{mellin_transform}.  In the weighted case, for each fixed $j$ the sequence $c_{i,j}$ will be the coefficient sequence of the series expansion of the Mellin transform of the $j$ loop primitive Feynman graph while the decoration will index different external leg structures or other relevant properties of the primitive graphs.  From a mathematical perspective, we can view these series as simply given to us by physics; we will not use any properties of them and the $c_{i}$ and $c_{i,j}$ can be treated as indeterminates for the algebraic results.

The reader might wonder why $v=\rootv(t)$ is excluded in the products above.  This is one of those points where the requirements of the application force the situation, as to use tubings to solve Dyson--Schwinger equations we need to treat the root differently, see the first point of \cref{thm main thm}.  Note, however, that for the coefficient of $L^1$ in the first point of \cref{thm main thm} the contribution of the root is exactly as it would be if the root were not excluded from the definition of the Mellin monomial.  The coefficient of $L^1$ gives the \emph{anomalous dimension} (see \cref{G_log_expansion}), and this illustrates how the anomalous dimension is particularly elegant mathematically as well as being meaningful physically.

We are now ready to state our main theorem.    Write
\begin{align} \label{def falling factorial}
    (x)_k = x(x-1)(x-2)\cdots (x-k+1)
\end{align}
for the falling factorial. (As per the usual convention for the empty product, $(x)_0 = 1$.)
We will use the square bracket notation to specify a particular coefficient in a generating series.  For example, in a series \(A(x) = \sum_{n=0}^{\infty} a_n x^n\), \([x^n] A(x)\) denotes the coefficient \(a_n\) of \(x^n\).

\begin{theorem}
\label{thm main thm}
    The Dyson--Schwinger equation \cref{dse_differential_general} with $F_k(\rho) = \sum_{i\geq 0} c_{i,k}\rho^{i-1}$ can be solved by a sum over tubings of weighted rooted trees, where $w(v)$ denotes the weight of vertex $v$, as follows.
    \begin{enumerate}
        \item The contribution of a tubing $\tau \in \Tub(t)$ to the solution is 
        \[
            \phi_L(\tau) = c(\tau)\sum_{i=1}^{b(\tau)}c_{b(\tau)-i, w(\rootv(t))} \frac{L^i}{i!}.
        \]
        
        \item The solution to \cref{dse_differential_general} is
        \[
            G(x,L) = \sum_{t} \frac{x^{w(t)}}{|\operatorname{Aut}(t)|}\left(\prod_{v \in V(t)} (1+sw(v))_{\od{v}}\right) \sum_{\tau \in \Tub(t)}\phi_L(\tau),
        \]
        where the outer sum is over all weighted rooted trees. 
        \item In particular, the anomalous dimension is 
        \begin{align*}
               \gamma(x) =[L^1] G(x,L) = \sum_{t} \frac{x^{w(t)}}{|\operatorname{Aut}(t)|}\left(\prod_{v \in V(t)} (1+sw(v))_{\od{v}}\right) \sum_{\tau \in \Tub(t)}\prod_{ v\in V(t)} c_{b(v,\tau)-1, w(v)}.
        \end{align*} 
    \end{enumerate}
\end{theorem}
Note that in the case of only one Mellin transform, which we sometimes call the \emph{unweighted} case, the theorem also applies, but we suppress the second index of the $c_{i,j}$ as it is redundant and so as to keep the notation in accord with \cref{def Mellin monomial}. 
The proof of \cref{thm main thm} will be given in \cref{sec single equation}.
Systems can be solved similarly but with more notation and bookkeeping required, see \cref{sec systems}.

Furthermore, as a consequence of \cref{thm multi primitive comb dse} the solution to \cref{dse_differential_general} can also be written as a sum over weighted plane rooted trees.  This removes the need for dividing by the size of the automorphism group and replaces the falling factorial with binomials, specifically,
\[
   G(x,L) = \sum_{\genfrac{}{}{0pt}{2}{t \text{ weighted plane }}{\text{rooted tree}}} x^{w(t)}\left(\prod_{v \in V(t)} \binom{1+sw(v)}{\od{v}}\right) \sum_{\tau \in \Tub(t)}\phi_L(\tau).
\]
This form will be more convenient for relating to chord diagram expansions in \cref{sec chord}.

Furthermore, we have the additional property that the tubings not only give one way of indexing a solution to the Dyson--Schwinger equation, but they refine the solution as a sum over Feynman graphs, see \cref{rem graph by graph}.

\begin{example} 
Since the core of our main theorem is the contribution of each tree, let us consider the contributions of some explicit trees. Here, we restrict ourselves to unweighted trees, in \cref{thm main thm} this means that there is only one $F_k(\rho)$, namely $k=1$, and the second index in $c_{i,k}$ is suppressed,  $c_{i,1}=c_i$.    We will write 
$$
\phi_L(t) := \sum_{\tau \in \Tub(t)}\phi_L(\tau)
$$

The single vertex tree has only a single tubing consisting of a single tube, so the $b$-statistic (\cref{def b statistic}) is 1 and the Mellin monomial (\cref{def Mellin monomial}) is $1$, since it is an empty product. By \cref{thm main thm},  
\[
    \phi_L( \rtLine 1 ) = c_0 L.
\]
The two vertex tree (ladder) also has only a single tubing, see \cref{fig:small_example} b), with Mellin monomial $c(\tau) = c_0$. The $b$-statistic for the root is $2$ and for the other vertex (the leaf) it is $1$, therefore, by  \cref{thm main thm},
\[
\phi_L\left( \rtLine 2  \right) = c_0\left(\frac 12 c_0 L^2 + c_1 L\right).
\]
The two trees with three vertices have two tubings each.  For $\rtV$ both tubings give the same contribution, see \cref{fig:small_example} c), leading to an overall factor 2 in the amplitude
\[
\phi_L \left( \rtV  \right) = 2\cdot c_0^2\left(\frac 16 c_0 L^3 + \frac 12 c_1 L^2 + c_2 L\right) =  \frac 13 c_0^3 L^3 + c_0^2 c_1 L^2 +2 c_2 c_0^2 L.
\]
For the second tree on three vertices, the two tubings have different contributions since the $b$-statistics differ, see \cref{fig:small_example} d), 
\begin{align*}
\phi_L \left( \rtLine 3  \right) &= c_1c_0\left(\frac 12 c_0 L^2 + c_1 L\right) +c_0^2\left(\frac 16 c_0 L^3 + \frac 12 c_1 L^2 + c_2 L\right)  \\
& = \frac 16 c_0^3 L^3  +     c_1 c_0^2  L^2  +\left( c_1^2 c_0 + c_2 c_0^2\right) L.
\end{align*}
All these examples reproduce the known tree Feynman rules for unweighted trees discussed in \cite{panzer_hopf_2012} or \cite{dugan_sequences_2023}.  

\end{example}

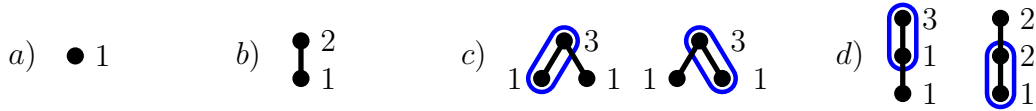
\begin{figure}[htb]
	\centering
	\begin{tikzpicture}

        \node at (-6,.5){$a)$};
        \coordinate (v1) at (-5.3,.5){};
        \node [vertex, label={right:1}] at (v1){};

        \node at (-3,.5){$b)$};
        \coordinate (v1) at (-2.3,.7){};
        \coordinate (v2) at (-2.3,.2){};
        \draw [edge] (v1) to (v2) ;
        \node [vertex, label={right:2}] at (v1){};
        \node [vertex, label={right:1}] at (v2){};

        \node at (0,.5){$c)$};
        \coordinate (v1) at (1.2,.7){};
        \coordinate (v2) at ($(v1)+ (-.3,-.5)$){};
        \coordinate (v3) at ($(v1)+ (.3,-.5)$){};
        \draw[tube1] (v1)--(v2);
        \draw [edge] (v2) to (v1) to (v3);
        \node [vertex,  label={right:3}] at (v1){};
        \node [vertex, label={left:1}] at (v2){};
        \node [vertex, label={right:1}] at (v3){};

        \coordinate (v1) at (3,. 7){};
        \coordinate (v2) at ($(v1)+ (-.3,-.5)$) {};
        \coordinate (v3) at ($(v1)+ (.3,-.5)$){};
        \draw[tube1] (v1)--(v3);
        \draw [edge] (v2) to (v1) to (v3);
        \node [vertex, label={[label distance=1.5mm]right:3}] at (v1){};
        \node [vertex, label={left:1}] at (v2){};
        \node [vertex, label={[label distance=1.5mm]right:1}] at (v3){};

        \node at (5,.5){$d)$};
        \coordinate (v3) at (5.7,0){};
        \coordinate (v2) at (5.7,.5){};
        \coordinate (v1) at (5.7,1){};
        \draw[tube1] (v1)--(v2);
        \draw [edge] (v1) to (v2) to (v3);
        \node [vertex, label={right:3}] at (v1){};
        \node [vertex, label={right:1}] at (v2){};
        \node [vertex, label={right:1}] at (v3){};

        \coordinate (v3) at (7, 0){};
        \coordinate (v2) at (7, .5){};
        \coordinate (v1) at (7, 1){};
        \draw[tube1] (v2)--(v3);
        \draw [edge] (v1) to (v2) to (v3);
        \node [vertex, label={right:2}] at (v1){};
        \node [vertex, label={right:2}] at (v2){};
        \node [vertex, label={right:1}] at (v3){};
	\end{tikzpicture}
	\caption{Some small trees, and the $b$-statistic (\cref{def b statistic}) of the vertices for the indicated tubing. Recall that we do not draw the innermost tubes around the individual vertices, nor the outermost tube which contains the complete tree.}
	\label{fig:small_example}
\end{figure}

\subsection{Physical set up}
\label{sec physical set up}

The results of this paper can be viewed purely from a combinatorial perspective, but the Dyson--Schwinger equation in \cref{dse_differential_general} originates from quantum field theory. In the remainder of \cref{sec definitions}, we review the physical set-up, introduce the necessary algebraic background, and finally bring the two together to understand the Dyson--Schwinger equations which underpin our approach.  A reader unfamiliar with the physics may read lightly: the precise algebraic notions will be introduced in \cref{sec algebraic set up}.

The goal of perturbative quantum field theory is to compute Feynman amplitudes, which can be understood as $n$-point correlation functions of quantum fields. A $n$-point Feynman amplitude $G^{(n)}$ is a function of the masses of all particles involved and of $n-1$ momenta, each of which is a vector in Minkowski space. The last momentum is fixed by overall momentum conservation. 

We introduce an arbitrary, but fixed, reference momentum $\mu$. This momentum defines a scale.  Namely, let $q$ be a non-vanishing linear combination of the arguments of $G^{(n)}$, then we define the logarithmic energy scale of the $G^{(n)}$ to be
\begin{align}\label{def L}
    L= \ln \left(\frac{q^2}{\mu^2} \right).
\end{align}
All remaining arguments of $G^{(n)}$ can be expressed as scale-free \enquote{angles} by rescaling them to $q^2$; see \cite{brown_angles_2013,brown_decomposing_2012} for details. Note that $L$ is not necessarily the   \enquote{energy-dependence} of the amplitude: a change in $L$, with all angles fixed, amounts to a scaling of all momenta and all masses by the same factor, whereas \enquote{energy-dependence} typically means to scale the momenta but leave masses fixed. 

For the remainder of this paper, we will restrict ourselves to one-particle irreducible propagator-type  Green functions, denoted $G(x,L)$, where $x$ is the renormalized coupling and $L$ is the logarithmic scale (see \cref{def L}). The restriction to one-particle irreducible is standard and loses no generality by the invertiblity of log and the Legendre transform \cite{jackson_robust_2017}.  The restriction to propagator-type insertions and one scale $L$ does decrease generality. In this setting, all angles come from masses. We restrict ourselves further by assuming that all angles are kept fixed and $L$ is the only kinematic variable.  The order-$x^n$ contribution, $[x^n]G(x,L)$,  is given by Feynman graphs with $n$ loops (recall that in physics terminology, the \emph{loop order} amounts to the dimension of cycle space of the Feynman graph) and exactly two external edges.

We use kinematic renormalization conditions with renormalization point $\mu^2$. Each renormalized $n$-loop Feynman graph evaluates to a polynomial of degree at most $n$ in $L$ with no constant term. We project the Green function onto its tree-level tensor such that the leading term is scaled to unity. Therefore, the Green functions considered in this paper have an expansion in the scale variable $L$ \cref{def L} of the form 
\begin{align}\label{G_log_expansion}
    G(x,L) = 1 + \sum_{k=1}^\infty \gamma_k(x) L^k ,
\end{align}
where the function $\gamma_1(x)$, often simply denoted $\gamma(x)$, is called \emph{anomalous dimension}.
The Green function $G(x,L)$ is determined by a Dyson--Schwinger equation. Before we discuss the precise form of this equation in our physical application, we will first review some algebraic background and terminology which we will need later.

For the results of this paper, we will be working in the context of rooted trees, not Feynman graphs, so it is worth sketching the relation between them.  The key structure here is that of divergent subgraphs inside one-particle irreducible Feynman graphs.  These are the subgraphs whose Feynman integral is divergent before renormalization.  If the divergent subgraphs sit inside their parent graph in a tree structure,  this tree structure is the \emph{insertion tree} of the Feynman graph.

\begin{definition}[\cite{kreimer_overlapping_1999}]\label{def insertion tree}
    An \emph{insertion tree} $t$ is a rooted tree with each vertex labelled with a primitive divergent Feynman graph and edges labelled to mark the insertion place.  This is a decorated rooted tree where the decoration at each vertex is given by the graph labelling the vertex along with the insertion place labelling the edge above the vertex.
\end{definition}

For a particular example see \cref{eg yukawa insertion trees}. In the case of overlapping subdivergences, a single Feynman graph gives rise to multiple insertion trees, each of which gives a choice of how to build the parent graph by insertion.
The labelled insertion trees are in bijection with the original Feynman graphs. Without labelling, the map is generally not bijective as different Feynman graphs can give rise to the same shape of insertion tree.

\subsection{Algebraic set up} \label{sec algebraic set up}

The algebraic underpinnings of our results are Hopf algebraic.
We work over some commutative $\mathbb Q$-algebra $K$, though $\mathbb{R}$ is sufficient for the physical application. The \textit{Connes--Kreimer Hopf algebra} $\CK$
is the free commutative $K$-algebra generated by rooted trees\footnote{For now, these are unlabelled and undecorated trees. We introduce decorations later in this section.}. We identify a product of rooted trees with the forest having these trees as its components; thus forests give a basis of $\CK$. (In particular, the multiplicative unit 1 corresponds to the empty forest.)  We equip this with a
coproduct \cite{kreimer_overlapping_1999,connes_hopf_1998,connes_hopf_1999} defined on trees as
\begin{equation} \label{eqn ck coproduct}
    \Delta t = \sum_f f \otimes (t \setminus f)
\end{equation}
where $f$ ranges over rooted subforests (disjoint unions of rooted subtrees) of $t$, and $t\setminus f$ is the remainder of $t$ where $f$ has been removed.

\begin{remark}\label{rem linearized coprod}
Observe that in the case where $f$
has a single component, \cref{eqn ck coproduct} is the same kind of decomposition as in the definition of binary tubings (\cref{def tubing}).  If we take the coproduct and then project onto the vector space inside $\CK$ given by the span of single rooted trees, this is known as the \emph{linearized coproduct}.  From the recursive characterization of tubings, we see that the tubings of a tree $t$ correspond to  the ways to iterate the linearized coproduct   until $\bullet\otimes \bullet\otimes \cdots \otimes \bullet$ is obtained.
\end{remark}

Let $B_+$ be the linear operator on $\CK$ which sends a forest to the tree obtained by adding a root and attaching this new root to the root of each component, as in our original recursive definition of rooted trees (\cref{def rooted tree}). This operator satisfies
\begin{equation}\label{coprod B+}
    \Delta B_+ = B_+ \otimes 1 + (\id \otimes B_+) \Delta.
\end{equation}
Any tree can be uniquely written as $B_+$ applied to a forest, so  \cref{coprod B+} could also be taken as a recursive definition of the coproduct \cref{eqn ck coproduct}. In general, an operator on any bialgebra which satisfies this identity is a \textit{Hochschild 1-cocycle}.

\begin{theorem}[Universal property {\cite[Sec.~3, Thm.~2]{connes_hopf_1999}}] \label{ck universal property}
    Let $A$ be a commutative algebra and $\Lambda$ be an operator on $A$. There exists a unique
    algebra morphism $\phi\colon \CK \to A$ such that $\phi B_+ = \Lambda \phi$. Moreover, if $A$ is
    a bialgebra and $\Lambda$ is a Hochschild 1-cocycle then $\phi$ is a bialgebra morphism.
\end{theorem}

For any $s \in K$, we consider the \textit{combinatorial Dyson--Schwinger equation} (DSE)
\begin{equation} \label{dse algebraic}
    T(x) = 1 + xB_+\big(T(x)^{1+s} \big).
\end{equation}
We will see in \cref{sec mellin transform} what these Dyson--Schwinger equations mean in quantum field theory.  For now, they are simply defining equations for certain special classes of simple trees.

\begin{lemma}\label{lem cdse sols}
The combinatorial Dyson--Schwinger equation $T(x) = 1 + xB_+(T(x)^{1+s})$ has a unique power series solution $T(x) \in \CK[[x]]$. Explicitly, the solution is given as a sum over all rooted trees 
\begin{align}
    T(x) &= 1 + \sum_t \left(\prod_{v \in V(t)} (1 + s)_{\od v}\right) \frac{tx^{|t|}}{|\operatorname{Aut}(t)|} \label{rt solution 1}
    \\
    &= 1 + \sum_t \operatorname{pe}(t)\left(\prod_{v \in V(t)} \binom{1+s}{\od v}\right) tx^{|t|}, \label{rt solution 2}
\end{align}
where $\operatorname{pe}(t)$ is the number of plane embeddings and $(1 + s)_{\od v}$ is the falling factorial (see \cref{def falling factorial}). 
\end{lemma}

 This lemma is a consequence of Proposition 2 of \cite{foissy_faa_2008} or Lemma 4 of \cite{bergbauer_hopf_2006} along with elementary group theory relating the plane and non-plane cases. It can also be derived from the more general version we will later prove as \cref{thm cdse system sols}.

Let us observe that while the sums are over all rooted trees, for some values of $s$ not all rooted trees have nonzero coefficients.  If $s\geq -1$ is an integer, then for $\od v > s+1$ the falling factorial is 0. In that case, the solution to the combinatorial Dyson--Schwinger equation is a sum over $(s+1)$-ary rooted trees. For noninteger $s$, and for integer $s<-1$, the falling factorials will all be nonzero and all trees will appear in the sum.

Note also that \cref{rt solution 2} can be interpreted as a sum over plane rooted trees upon removing $\operatorname{pe}(t)$.

\begin{example} \label{ex dse trees trivial}
   A trivial example of \cref{dse algebraic} is the choice $s=-1$. In this case, the DSE is not recursive. The equation -- and at the same time its solution -- read
    \begin{align*}
        T(x) = 1+x B_+(1),
    \end{align*}
    where $B_+(1)=\rtLine 1$ is the unique rooted tree consisting of only a single vertex. 
\end{example}

\begin{example} \label{ex dse trees ladder}
    The simplest non-trivial DSE is the choice $s=0$ in \cref{dse algebraic}, called \emph{linear DSE} because the argument of $B_+$ is a linear function of $T(x)$. In this case, the solution \cref{rt solution 1} is the sum over \emph{ladder} trees -- trees where each vertex has at most one child.  Let $\ell_n$ denote the ladder on $n$ vertices.
    \[
    \ell_0 = 1, \qquad \ell_1 = \rtLine 1, \qquad \ell_2 = \rtLine 2, \qquad \ell_3 = \rtLine 3, \;\;\ldots
    \]
    In this case $(1+s)_{\od v} =1$ and  $|\operatorname{Aut}(t)| =1$ so we find
    \begin{align*}
        T(x) &= 1 + x \rtLine 1 + x^2 \rtLine 2 + x^3 \rtLine 3 + \cdots  = \sum_{j\geq 0} x^j \ell_j.
    \end{align*}
    We will return to the ladders in \cref{sec ladders}.
\end{example}

\begin{example}\label{ex trees negative s}
A particularly important example of \cref{dse algebraic} is the case when $s=-2$.
The equation generates the class of plane rooted trees (or non-plane rooted trees with multiplicities giving the number of plane embeddings).  The equation becomes 
\[
T(x) = 1+xB_+\left(\frac{1}{T(x)}\right).
\]
Observe that the series $T(x)$ starts with a constant 1, so the fraction is actually well-defined and can be expanded in a geometric series  
\[
T(x) =1+ xB_+\left(\sum_{k\geq 0} (-1)^k \left( T(x)-1\right) ^k\right).
\]
The latter equation has the intuitive interpretation that the subtrees at each vertex are a list of nonempty trees; $T(x)$ generates plane rooted trees.  The sign is a minor irritant, keeping track of the parity of the number of children at each vertex.  Avoiding this sign is one reason for different sign conventions appearing in the set-up in some sources, but it causes nuisance signs in other places.

Let us compare this to \cref{lem cdse sols}.  Using \cref{rt solution 2} we see that the coefficient for a given plane tree when $s=-2$ should be $\prod_{v\in V(t)} \binom{-1}{\od v}$, but $\binom{-1}{k} = (-1)^k$, so the coefficient of the tree is $\prod_{v\in V(t)}(-1)^{\od v}$ as discussed above.
A physical realization of the combinatorial DSE for $s=-2$ is the Yukawa propagator, introduced in \cref{eg yukawa 1}.

\end{example}

\begin{remark}
Similarly to the above commutative Connes--Kreimer Hopf algebra $\CK$, one can use plane rooted trees (\cref{def rooted tree}) to construct a noncommutative Connes--Kreimer Hopf algebra. To this end, one changes the commutative product of trees into a noncommutative product of plane trees and otherwise carries the plane information along in an obvious way \cite{foissy_algebres_2002, han_new_2010}.  In the present paper, we only use the commutative Hopf algebra $\CK$, but we could reformulate to work on the noncommutative Connes--Kreimer Hopf algebra and so directly use the plane rooted trees from the beginning that we only pass to for the connection to chord diagrams in \cref{sec bijection tubings chord diagrams}.
\end{remark}

\medskip

The Hopf algebra $\CK$ introduced so far is not yet sufficiently general to capture a Dyson--Schwinger equation of type \cref{dse_differential_general}, because it contains only a single cocycle $B_+$. We generalize it to a family of Hopf algebras $\CK(D)$, where the elements of $D$ are called \emph{decorations}. Concretely, let $\CK(D)$ be the free commutative $K$-algebra generated by trees with vertices decorated by elements of $D$. We will chiefly be interested in the case $D = \ZZ_{\ge 1}$, which can be interpreted as weighted rooted trees. We equip $\CK(D)$ with the same coproduct as that of $\CK$ in the sense that we sum over the same subtrees and simply carry along the vertex decorations. Then for each $d \in D$ there is a corresponding 1-cocycle $B_+^{(d)}$ that adds a root with decoration $d$. This Hopf algebra satisfies a universal property which generalizes of \cref{ck universal property}.

\begin{theorem}[{\cite[Theorem 31]{foissy_algebres_2002}}] \label{thm ck universal property dec}
    Let $A$ be a commutative algebra and $\{\Lambda^{(d)}\}_{d \in D}$ be a family of operators on
    $A$.  There exists a unique algebra morphism $\phi\colon \CK(D) \to A$ such that $\phi B_+^{(d)}
    = \Lambda^{(d)} \phi$. Moreover, if $A$ is a bialgebra and $\Lambda^{(d)}$ is a Hochschild
    1-cocycle for each $d$ then $\phi$ is a bialgebra morphism.
\end{theorem}

The generalization of \cref{dse algebraic} to the decorated Hopf algebra $\CK(\ZZ_{\geq 1})$ is the \emph{multi-primitive combinatorial DSE}
\begin{equation} \label{dse algebraic general}
    T(x) = 1 + \sum_{k \ge 1} x^k B_+^{(k)} \big(T(x)^{1 + sk} \big).
\end{equation}

\begin{lemma} \label{thm multi primitive comb dse}
The multi-primitive combinatorial DSE \cref{dse algebraic general} has the unique solution
\begin{equation*}
    T(x) = 1 + \sum_t \left(\prod_{v \in V(t)} (1 + sw(v))_{\od v}\right)
    \frac{tx^{w(t)}}{|\operatorname{Aut}(t)|},
\end{equation*}
where $t$ now ranges over weighted rooted trees and the vertex weight $w(v)$  is defined before \cref{def tube}.
\end{lemma}

For $s = 1$ this is precisely \cite[Lemma 4]{bergbauer_hopf_2006}. The argument there can be extended to positive integer $s$, and the general case follows by uniqueness of interpolating polynomials. It can also be derived from our result \cref{thm cdse system sols}.

\begin{remark}
    While restricting to $D = \ZZ_{\ge 1}$ is sufficient for solving the single equation
    \cref{dse_differential_general}, we will make use of the more general setup when working with
    \textit{systems} of Dyson--Schwinger equations in \cref{sec systems}.
\end{remark}

So far, we have worked in the Hopf algebra of rooted trees. 
To solve equations of a similar form to \cref{dse algebraic} or \cref{dse algebraic general} in
other algebras, we can apply the map $\phi$ from \cref{thm ck universal property dec} to $T(x)$.
Most important for us is the algebra $K[L]$ of polynomials in the single variable $L$. (We will treat the $G(x, L)$ in \cref{dse_differential_general} as a series in $x$ with coefficients from $K[L]$.) This algebra is a Hopf algebra with coproduct given by
\begin{equation} \label{eq polynomial coproduct}
    \Delta L = 1 \otimes L + L \otimes 1
\end{equation}
extended uniquely as an algebra homomorphism.

\begin{theorem}[{\cite[Theorem 2.6.4]{panzer_hopf_2012}}] \label{thm polynomial cocycles}
    Consider a power series $A(u) \in K[[u]]$ and let $\partial_u = \frac\partial {\partial u}$. The operator on $K[L]$ given by
    \begin{equation} \label{polynomial cocycle 1}
        f(L) \mapsto \int_0^L \textnormal{d}u\;A( \partial_u)f(u)
    \end{equation}
     is a Hochschild 1-cocycle. Moreover, all 1-cocycles on $K[L]$ are of this form.
\end{theorem}

\begin{remark} \label{rem polynomial cocycles}
    Let $\Lambda$ be the operator in \cref{thm polynomial cocycles} and write $a_k$ for the
    coefficient of $u^k$ in $A(u)$. We can explicitly calculate the action of $\Lambda$ on a basis
    of $K[L]$:
    \begin{align*}
        \Lambda\left(\frac{L^n}{n!}\right)
        &= \int_0^L \mathrm{d}u\; \sum_{k=0}^\infty a_k \partial_u^k \frac{u^n}{n!} \\
        &= \int_0^L \mathrm{d}u\; \sum_{k=0}^n a_k \frac{u^{n-k}}{(n-k)!} \\
        &= \sum_{k=0}^n a_k \frac{L^{n-k+1}}{(n-k+1)!}.
    \end{align*}
    This is equal to the coefficient of $\rho^n$ in the power series $(e^{L\rho} - 1) A(\rho)/\rho$.
    It follows by linearity that $\Lambda$ can also be written as
    \begin{equation}\label{polynomial cocycle 2}
        \Lambda f(L) = f(\partial_\rho) \left(e^{L\rho} - 1\right)
        \frac{A(\rho)}{\rho}\bigg|_{\rho=0}
    \end{equation}
    since in the case of $f(L) = L^n/n!$ this extracts the same coefficient. The form
    \cref{polynomial cocycle 2} is more physically natural (see \cref{sec mellin transform}) since
    it appears in the Dyson--Schwinger equation \cref{dse_differential_general}. Nonetheless, we
    will make use of the form \cref{polynomial cocycle 1} in the proof of our main theorem.
\end{remark}

One of our main goals will be to understand explicitly what the bialgebra morphism from \cref{ck
universal property} looks like when $\Lambda$ is a 1-cocycle on $K[L]$. It will be useful to better
understand what morphisms into $K[L]$ look like in general. Recall that a \textit{character} of a
Hopf algebra $H$ is simply an algebra morphism $H \to K$; these form a group under the convolution
product $*$ defined as
\begin{equation}\label{convolution product}
    \alpha * \beta = (\alpha \otimes \beta) \Delta.
\end{equation}
An algebra morphism $H \to K[L]$ then defines a one-parameter family of characters; in the case of a
bialgebra morphism this is a one-parameter subgroup.  An \textit{infinitesimal character} is a map
$\sigma\colon H \to K$ such that
\begin{equation}\label{infinitesimal character}
    \sigma(ab) = \varepsilon(a)\sigma(b) + \sigma(a)\varepsilon(b)
\end{equation}
for all $a, b \in H$. Infinitesimal characters form a Lie algebra under convolution commutator.  More details can be found in \cite{bogfjellmo_character_2016,bogfjellmo_geometry_2016} and references therein. An
analogue in this setting of the exponential relationship between Lie algebras and groups is the
following theorem. 

\begin{theorem}  \label{thm character exponential}
    Suppose $H$ is a bialgebra and $\phi\colon H \to K[L]$ is an algebra morphism. Further, suppose
    that the constant term of $\phi$ is given by the counit. Let $\sigma$ be the linear term of
    $\phi$.  Then $\sigma$ is an infinitesimal character. Moreover, $\phi$ is a bialgebra morphism
    if and only if
    \begin{equation*}
        \phi = \exp_*(L\sigma) = \varepsilon + L\sigma + \frac 12 L^2 (\sigma * \sigma) +
        \cdots.
    \end{equation*}
\end{theorem}

\begin{proof}
    By assumption we have
    \[
        \phi(a) = \varepsilon(a) + \sigma(a)L + \text{(higher-order terms)}
    \]
    for all $a \in H$. Then
    \[
        \phi(ab) = \phi(a)\phi(b) = \varepsilon(a)\varepsilon(b) + (\varepsilon(a)\sigma(b) +
        \sigma(a)\varepsilon(b))L + \text{(higher-order terms)}
    \]
    and hence
    \[
        \sigma(ab) = \varepsilon(a)\sigma(b) + \sigma(a)\varepsilon(b)
    \]
    so $\sigma$ is an infinitesimal character.

    For any element $u$ of a $K$-algebra $A$, let us write $\phi_u$ for the map $H \to A$ given by
    evaluating $\phi$ at $L = u$ (so $\phi_L = \phi$). Then by the definition \cref{eq polynomial
    coproduct} of the coproduct we have $\Delta \phi = \phi_{L \otimes 1 + 1 \otimes L}$. Since
    $\phi$ preserves the counit by hypothesis, $\phi$ is a bialgebra morphism if and only if
    \begin{equation} \label{eqn:phi preserves coproduct}
        \Delta\phi = (\phi \otimes \phi)\Delta.
    \end{equation}
    Now, we may identify $K[L] \otimes K[L]$ with the polynomial
    algebra $K[L_1, L_2]$ in two variables via the map $f(L) \otimes g(L) \mapsto f(L_1)g(L_2)$. Applying this to both sides of \cref{eqn:phi preserves coproduct} we see that $\phi$ is a bialgebra morphism if and only if
    \begin{equation} \label{eq bialgebra morphism equivalent}
        \phi_{L_1 + L_2} = \phi_{L_1} * \phi_{L_2}.
    \end{equation}
    If $\phi = \exp_*(L\sigma)$ it clearly satisfies \cref{eq bialgebra morphism equivalent};
    this is the defining property of the exponential function. Conversely suppose $\phi$ satisfies
    \cref{eq bialgebra morphism equivalent}. Then
    \begin{align*}
        \partial_L \phi
        &= [h^1] \phi_{L + h} \\
        &= \phi_L * [h^1] \phi_h \\
        &= \phi * \sigma
    \end{align*}
    and hence
    \begin{align*}
        [L^n] \phi
        &= \frac{1}{n!} \partial_L^n \phi\bigg|_{L = 0} \\
        &= \frac{1}{n!} \phi * \sigma^{*n}\bigg|_{L = 0} \\
        &= \frac{\sigma^{*n}}{n!}
    \end{align*}
    since $\phi_0 = \varepsilon$ by assumption. Of course, this coefficient is the same as in
    $\exp_*(L\sigma)$, so the two are equal.
\end{proof}

\subsection{Physical Dyson--Schwinger equations and Mellin transforms}
\label{sec mellin transform}

Returning to the physical motivation, we will now make the connection between the Dyson--Schwinger equations as set up so far and their more physical form and meaning. A more detailed exposition can be found in  \cite{kreimer_etude_2006,kreimer_etude_2008,kreimer_renormalization_2013,balduf_dyson_2024}.

As introduced in \cref{sec physical set up}, the Green function in question, $G(x,L)$, represents the 2-point function, that is, the propagator including quantum corrections. Schematically,  a Dyson--Schwinger equation in terms of Feynman integrals is a fixed point integral equation of the form
\begin{align}\label{dse_integral}
    G(x,p^2) = 1 +(1-\mathcal R) \sum_{k \geq 1} x^k \int \textnormal{d}^{kD} q\;  \Big( K_k (p^2,q^2) G(x,q_e^2)^{1+sk} \Big).
\end{align}
Here, $K_k$ denotes the sum of all $k$-loop kernel graphs. These are the  Feynman graphs with two external edges and without subdivergences. The integration $\textnormal{d}^{kD}$ is over the $k$ independent loop momenta of the kernel, each of which is a $D$-dimensional vector in Minkowski space. The variable $q_e$ denotes the momentum of the edge $e\in K_k$, it is a linear combination of the external momentum $p$ and the loop momenta $q$. Finally, $\mathcal R$ denotes the renormalization operator, which in our case is the subtraction at the fixed external momentum scale $\mu^2$. In general, the monomial $G(x,q_e^2)^{1+sk}$ is to be understood as a product of $G(x,q_e)$ where the individual $q_e$ belong to the different edges. The parameter $s\in \mathbb R$ indicates how many new edges arise per order $k$ of the coupling. For a theory with only  $n$-valent interaction vertices, the physical choice would be $s=-n$, but we leave $s$ as a free parameter. The choice $s\neq -n$ amounts to a sum of Feynman graphs with non-standard combinatorial prefactors.

We simplify the DSE \cref{dse_integral} by the assumption that the propagator corrections are inserted into only a single edge of the kernel graphs.  This is a loss of generality, but it implies that all factors in $G(x,q_e)^{1+sk}$ have the same argument $q_e$ -- no longer belonging to different edges $e$ -- and we can choose $q_e$ to be one of the integration momenta $q$. 

\begin{example}
\label{eg yukawa 1}
    The most basic example is the fermion propagator in 4-dimensional Yukawa theory, restricted to a single kernel graph ($k=1$), the 1-loop multi-edge with one fermion and one meson edge shown in \cref{fig yukawa kernel}. The case $s=0$ in the DSE \cref{dse_integral} for the Yukawa fermion,  
    \begin{align*}
    G(x,p^2) = 1 +(1-\mathcal R)\;  x \int \textnormal{d}^{D} q\;  \Big( K_1 (p^2,q^2) G(x,q^2) \Big),
    \end{align*}
    amounts to the \emph{rainbow approximation} shown in \cref{fig yukawa rainbows}. 
    
    Inserting the kernel into itself iteratively in all possible ways corresponds to $s=-2$ in the DSEs \cref{dse_integral,dse algebraic}, discussed in \cref{ex trees negative s}. We obtain graphs like \cref{fig yukawa all}.  These are the class of graphs for which Broadhurst and Kreimer solved the Dyson--Schwinger equation in \cite{broadhurst_exact_2001}, and in terms of which the original chord diagram expansion of one of us with Marie \cite{marie_chord_2013} was formulated.
\end{example}

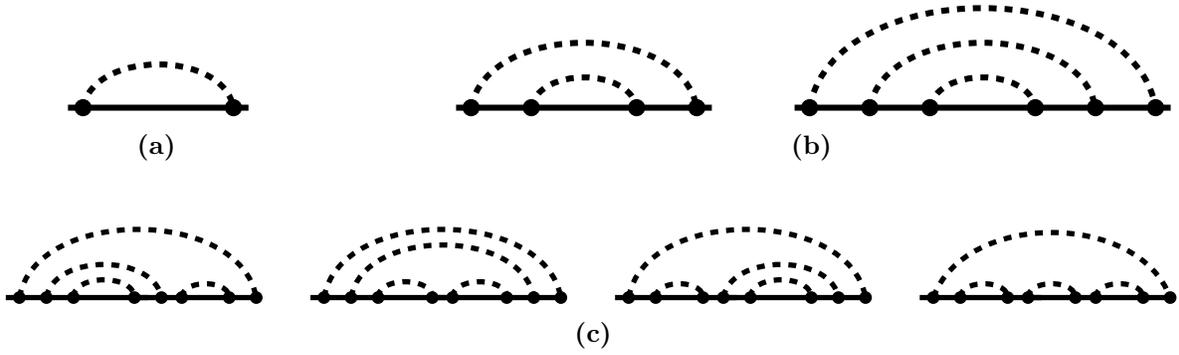
\begin{figure}[htb]
    \centering
    \begin{subfigure}[b]{.3\linewidth}
        \centering
        \begin{tikzpicture}

\coordinate (v1) at (0,0){};
            \coordinate (v2) at (2,0){};

            \node[vertex] at (v1){};
            \node [vertex] at (v2){};

            \draw[line width=.8mm, dashed, bend angle=80,bend left] (v1) to (v2);
            \draw[line width = .8mm] (v1) to (v2);
            \draw[line width = .8mm] (v1) to +(-.2,0);
            \draw[line width = .8mm] (v2) to +(.2,0);
        \end{tikzpicture}
        \caption{}
        \label{fig yukawa kernel}
    \end{subfigure}
    \hfill
    \begin{subfigure}[b]{.65\linewidth}
        \centering
        \begin{tikzpicture}

\coordinate (v1) at (.5,0){};
            \coordinate (v2) at (1.3,0){};
            \coordinate (v3) at (2.7,0){};
            \coordinate (v4) at (3.5,0){};
            
            \node [vertex] at (v1){};
            \node [vertex] at (v2){};
            \node [vertex] at (v3){};
            \node [vertex] at (v4){};

            \draw[line width=.8mm, dashed, bend angle=80,bend left] (v1) to (v4);
            \draw[line width=.8mm, dashed, bend angle=80,bend left] (v2) to (v3);
            \draw[line width = .8mm] (v1) -- (v2) -- (v3) -- (v4);
            \draw[line width = .8mm] (v1) to +(-.2,0);
            \draw[line width = .8mm] (v4) to +(.2,0);

            \coordinate (v1) at (5,0){};
            \coordinate (v2) at (5.8,0){};
            \coordinate (v3) at (6.6,0){};
            \coordinate (v4) at (8,0){};
            \coordinate (v5) at (8.8,0){};
            \coordinate (v6) at (9.6,0){};
            
            \node [vertex] at (v1){};
            \node [vertex] at (v2){};
            \node [vertex] at (v3){};
            \node [vertex] at (v4){};
            \node [vertex] at (v5){};
            \node [vertex] at (v6){};

            \draw[line width=.8mm, dashed, bend angle=80,bend left] (v1) to (v6);
            \draw[line width=.8mm, dashed, bend angle=80,bend left] (v2) to (v5);
            \draw[line width=.8mm, dashed, bend angle=80,bend left] (v3) to (v4);
            \draw[line width = .8mm] (v1) -- (v2) -- (v3) -- (v4) -- (v5) -- (v6);
            \draw[line width = .8mm] (v1) to +(-.2,0);
            \draw[line width = .8mm] (v6) to +(.2,0);
        \end{tikzpicture}
        \caption{}
        \label{fig yukawa rainbows}
    \end{subfigure}

    \vspace{.5cm}

    \begin{subfigure}[b]{\linewidth}
        \centering

\begin{tikzpicture}[scale=.9]

    \coordinate (v1) at (0,0){};
    \coordinate (v2) at ($(v1)+ (.4,0)$){};
    \coordinate (v3) at ($(v1)+ (.8,0)$){};
    \coordinate (v4) at ($(v1)+ (1.7,0)$){};
    \coordinate (v5) at ($(v1)+ (2.1,0)$){};
    \coordinate (v6) at ($(v1)+ (2.4,0)$){};
    \coordinate (v7) at ($(v1)+ (3.1,0)$){};
    \coordinate (v8) at ($(v1)+ (3.5,0)$){};
            
    \node [smallVertex] at (v1){};
    \node [smallVertex] at (v2){};
    \node [smallVertex] at (v3){};
    \node [smallVertex] at (v4){};
    \node [smallVertex] at (v5){};
    \node [smallVertex] at (v6){};
    \node [smallVertex] at (v7){};
    \node [smallVertex] at (v8){};

    \draw[line width=.7mm, dashed, bend angle=80,bend left] (v1) to (v8);
    \draw[line width=.7mm, dashed, bend angle=80,bend left] (v2) to (v5);
    \draw[line width=.7mm, dashed, bend angle=90,bend left] (v3) to (v4);
    \draw[line width=.7mm, dashed, bend angle=90,bend left] (v6) to (v7);
    \draw[line width = .7mm] (v1) -- (v2) -- (v3) -- (v4) -- (v5) -- (v6) -- (v7) -- (v8) ;
    \draw[line width = .7mm] (v1) to +(-.2,0);
    \draw[line width = .7mm] (v4) to +(.2,0);

    \coordinate (v1) at (4.5,0){};
    \coordinate (v2) at ($(v1)+ (.4,0)$){};
    \coordinate (v3) at ($(v1)+ (.8,0)$){};
    \coordinate (v4) at ($(v1)+ (1.6,0)$){};
    \coordinate (v5) at ($(v1)+ (1.9,0)$){};
    \coordinate (v6) at ($(v1)+ (2.7,0)$){};
    \coordinate (v7) at ($(v1)+ (3.1,0)$){};
    \coordinate (v8) at ($(v1)+ (3.5,0)$){};
            
    \node [smallVertex] at (v1){};
    \node [smallVertex] at (v2){};
    \node [smallVertex] at (v3){};
    \node [smallVertex] at (v4){};
    \node [smallVertex] at (v5){};
    \node [smallVertex] at (v6){};
    \node [smallVertex] at (v7){};
    \node [smallVertex] at (v8){};

    \draw[line width=.7mm, dashed, bend angle=80,bend left] (v1) to (v8);
    \draw[line width=.7mm, dashed, bend angle=80,bend left] (v2) to (v7);
    \draw[line width=.7mm, dashed, bend angle=90,bend left] (v3) to (v4);
    \draw[line width=.7mm, dashed, bend angle=90,bend left] (v5) to (v6);
    \draw[line width = .7mm] (v1) -- (v2) -- (v3) -- (v4) -- (v5) -- (v6) -- (v7) -- (v8) ;
    \draw[line width = .7mm] (v1) to +(-.2,0);
    \draw[line width = .7mm] (v4) to +(.2,0);

    \coordinate (v1) at (9,0){};
    \coordinate (v2) at ($(v1)+ (.4,0)$){};
    \coordinate (v3) at ($(v1)+ (1.1,0)$){};
    \coordinate (v4) at ($(v1)+ (1.4,0)$){};
    \coordinate (v5) at ($(v1)+ (1.8,0)$){};
    \coordinate (v6) at ($(v1)+ (2.7,0)$){};
    \coordinate (v7) at ($(v1)+ (3.1,0)$){};
    \coordinate (v8) at ($(v1)+ (3.5,0)$){};
            
    \node [smallVertex] at (v1){};
    \node [smallVertex] at (v2){};
    \node [smallVertex] at (v3){};
    \node [smallVertex] at (v4){};
    \node [smallVertex] at (v5){};
    \node [smallVertex] at (v6){};
    \node [smallVertex] at (v7){};
    \node [smallVertex] at (v8){};

    \draw[line width=.7mm, dashed, bend angle=80,bend left] (v1) to (v8);
    \draw[line width=.7mm, dashed, bend angle=90,bend left] (v2) to (v3);
    \draw[line width=.7mm, dashed, bend angle=80,bend left] (v4) to (v7);
    \draw[line width=.7mm, dashed, bend angle=90,bend left] (v5) to (v6);
    \draw[line width = .7mm] (v1) -- (v2) -- (v3) -- (v4) -- (v5) -- (v6) -- (v7) -- (v8) ;
    \draw[line width = .7mm] (v1) to +(-.2,0);
    \draw[line width = .7mm] (v4) to +(.2,0);

    \coordinate (v1) at (13.5,0){};
    \coordinate (v2) at ($(v1)+ (.4,0)$){};
    \coordinate (v3) at ($(v1)+ (1.1,0)$){};
    \coordinate (v4) at ($(v1)+ (1.4,0)$){};
    \coordinate (v5) at ($(v1)+ (2.1,0)$){};
    \coordinate (v6) at ($(v1)+ (2.4,0)$){};
    \coordinate (v7) at ($(v1)+ (3.1,0)$){};
    \coordinate (v8) at ($(v1)+ (3.5,0)$){};
            
    \node [smallVertex] at (v1){};
    \node [smallVertex] at (v2){};
    \node [smallVertex] at (v3){};
    \node [smallVertex] at (v4){};
    \node [smallVertex] at (v5){};
    \node [smallVertex] at (v6){};
    \node [smallVertex] at (v7){};
    \node [smallVertex] at (v8){};

    \draw[line width=.7mm, dashed, bend angle=70,bend left] (v1) to (v8);
    \draw[line width=.7mm, dashed, bend angle=90,bend left] (v2) to (v3);
    \draw[line width=.7mm, dashed, bend angle=90,bend left] (v4) to (v5);
    \draw[line width=.7mm, dashed, bend angle=90,bend left] (v6) to (v7);
    \draw[line width = .7mm] (v1) -- (v2) -- (v3) -- (v4) -- (v5) -- (v6) -- (v7) -- (v8) ;
    \draw[line width = .7mm] (v1) to +(-.2,0);
    \draw[line width = .7mm] (v4) to +(.2,0);

\end{tikzpicture}
        \caption{}
        \label{fig yukawa all}
    \end{subfigure}
    
    \caption{ In all three figures, bold edges represent fermions and dashed edges represent mesons. \textbf{(a)}  1-loop kernel graph $K_1 $ of the fermion propagator in Yukawa theory.  Corrections are to be inserted recursively into the fermion edge.  \textbf{(b)} The first two graphs of the \emph{rainbow approximation} to the fermion propagator. Their insertion trees are the ladder trees $\ell_2$ and $\ell_3$, respectively. \textbf{(c)}  4-loop graphs for the choice $s=-2$. The 4-loop rainbow is not shown.}
    \label{fig yukawa}
\end{figure}

The integral DSE \cref{dse_integral} is the counterpart of the combinatorial DSE \cref{dse algebraic general}; the Hochschild 1-cocycle $B^{(k)}_+$ corresponds to insertion into the $k$-th kernel graph $K_k$. To see this more explicitly, we rewrite the series expansion \cref{G_log_expansion} in terms of derivatives, using \cref{def L} and the shorthand $\partial_\rho = \frac  \partial {\partial \rho}$:
\begin{align}\label{G_log_expansion2}
    G(x,L) = 1 + \sum_{k=1}^\infty \gamma_k(x) L^k = 1+\sum_{k=1}^\infty \gamma_k(x) \partial_\rho^k \left( \frac{q^2}{\mu^2}\right)^{\rho} \Big|_{\rho=0} = G(x,\partial_\rho) \left( \frac{q^2}{\mu^2}\right)^{\rho} \Big|_{\rho=0}.
\end{align}
Inserting \cref{G_log_expansion2} into the Dyson--Schwinger equation \cref{dse_integral}, we obtain
\begin{align}\label{dse_integral_mellin}
    G(x,p^2) = 1 +(1-\mathcal R) \sum_{k \geq 1} x^k G(x,\partial_\rho)^{1+sk} (\mu^{-2})^\rho  \int \textnormal{d}^{kD} q\;  \Big( K_k (p^2,q^2) (q^2)^\rho  \Big)\Big|_{\rho=0}.
\end{align}
For a kernel $K_k$, the integral on the righthand side is the definition of the \emph{Mellin transform} $F_k(\rho)$, which in turn is expressed by its power series coefficients $c_{j,k}$:
\begin{align}\label{mellin_transform}
    \int \text{d}^{kD} q \; \Big(K_k (p^2,q^2)  \cdot \left(  q^2 \right)^{\rho} \Big)  &=  (p^2)^{\rho} F_k(\rho)=  e^{\rho \ln p^2}\sum_{j=0}^\infty c_{j,k} \rho^{j-1}.
\end{align}
In this definition, all powers of $4\pi$ and $\gamma_E$, which might arise in the loop integral, are absorbed by a redefinition of the coupling $x^k$. The simple pole $c_{0,k} \rho^{-1}$ in $F_k(\rho)$ indicates that $K_k$ is a primitively divergent Feynman graph.

The kinematic renormalization operator $\mathcal R$ in \cref{dse_integral} amounts to subtraction at the renormalization point $q^2 =\mu^2$. This produces a prefactor $((p^2)^{\rho} - (\mu^2)^\rho)$ to the integral in \cref{dse_integral}. Rescaling all momenta by $\mu^2$, and using \cref{def L}, we finally arrive at \cref{dse_differential_general}, 
\begin{align}\label{analytic dse}
    G(x,L) &= 1 + \sum_{k \geq 1} x^k G(x,\partial_\rho)^{1+sk} \Big( (e^{L \rho}-1) F_k(\rho) \Big)\Big|_{\rho=0}.
\end{align}
In \cref{analytic dse}, the insertion into kernel $K_k$ takes exactly the form \cref{polynomial cocycle 2}, where the series $A(\rho)/\rho = F_k(\rho)$ is the Mellin transform \cref{mellin_transform}. By this identification, the integral DSE \cref{dse_integral} -- restricted to insertions into a single edge -- indeed corresponds to the combinatorial DSE \cref{dse algebraic general}. 

\begin{example}\label{eg yukawa mellin transform}
    For the Yukawa fermion propagator introduced in \cref{eg yukawa 1}, in $D=4$ dimensions, the Mellin transform \cref{mellin_transform} of the kernel is (remember that we rescaled $x$ to absorb  $4\pi$)
\begin{align*}
    \int \frac{\text{d}^4 q}{(2\pi)^4} \; \frac{1}{(p+q)^2}\frac{1}{q^2} (q^2)^\rho &\propto (p^2)^\rho \frac{\Gamma(-\rho) \Gamma(1+\rho)}{\Gamma(1-\rho) \Gamma(2+\rho)} = \frac{e^{\rho \ln p^2} }{(-\rho)(\rho+1)} \nonumber \\
    \Rightarrow  F(\rho) &= \frac 1 {(-\rho)(\rho+1)} = \sum_{j=0}^\infty - (-1)^j \cdot \rho^{j-1}, \qquad c_{j,1} =c_j = -(-1)^j.
\end{align*}
\end{example}

The morphism $\phi$ from \cref{thm ck universal property dec} operates not on Feynman graphs, but on the Connes--Kreimer Hopf algebra of their insertion trees (\cref{def insertion tree}). Consequently, $\phi$ is sometimes called \emph{tree Feynman rules}: it assigns a polynomial in $L$ to each insertion tree, and turns the combinatorial Green function $T(x)$, which is a series in rooted trees, into the analytic Green function  $G(x,L) = \phi (T(x))$. By \cref{thm character exponential}, the linear term of the Feynman rules is given by the infinitesimal character \cref{infinitesimal character}, and the order-$k$ term is given by $\sigma^{*k}$. In the (log) expansion \cref{G_log_expansion}, the coefficient of $L^k$ is the function $\gamma_k(x)$, therefore the infinitesimal character extracts the anomalous dimension:
\begin{align}\label{anomalous dimension character}
    \gamma(x) &= \sigma(T(x)), \qquad \gamma_k(x) = \sigma^{\star k} (T(x)), \qquad G(x,L) = e^{\star \sigma L}(T(x)).
\end{align}

Conversely, the  Feynman rules discussed in the present subsection are a map from Feynman graphs directly to polynomials in $L$, without the intermediate step of insertion trees. This map is a composition of the tree Feynman rules $\phi$ and the map that sends a Feynman graph to its decorated insertion trees. Likewise, the combinatorial Dyson--Schwinger equations \cref{dse algebraic,dse_differential_general} can equivalently be seen as equations for Feynman graphs, where $T(x)$ is a series of Feynman graphs and $B^{(k)}_+$ denotes insertion into the corresponding kernel graph.

\begin{example}\label{eg yukawa insertion trees}

The Yukawa propagator from \cref{eg yukawa 1} involves only a single kernel graph, hence the resulting insertion trees are undecorated. It can be generalized to multiple kernel graphs, see e.g.~\cite{bierenbaum_nexttoladder_2007}, in which case the insertion trees obtain non-trivial labelling or decoration.

Let $Y$ be  a Yukawa graph with $n$ meson lines, where the interaction vertices are numbered $\{1,\ldots, 2n\}$.  We construct an explicit map $d :Y \mapsto t$   to its insertion tree $t$.
Each meson line is expressible by its two end vertices,  the outermost meson line is $c_0=(1,2n)$ and maps to the root $v_0 = d(c_0)$ of the plane rooted tree.  For each meson line $(a_j,b_j)$ of $Y$, consider the mesons which are outermost inside $(a_j,b_j)$, calling them $\{a_{i_1},b_{i_1}\}, \{a_{i_2},b_{i_2}\}$, $\dots, \{a_{i_k},b_{i_k}\}$ such that $a_j  <  a_{i_1}<b_{i_1} <  a_{i_2}<b_{i_2}< \dots < a_{i_k}<b_{i_k}  <  b_j$. Then the children of the vertex $v_j$ are $v_{i_1},v_{i_2},\dots, v_{i_k}$, where $v_j = d(a_j,b_j)$ is associated to the meson $(a_j,b_j)$ and $v_{i_l} = d(a_{i_l},b_{i_l})$ is associated with $(a_{i_l},b_{i_l})$.
In other words, if there exist components $c_1,c_2, \dots, c_k$ nested within each meson $c_i$ of the Yukawa graph, then there is a vertex $v_i=d(c_i)$ in the plane rooted tree with children $d(c_1),d(c_2), \dots, d(c_k)$.

The so-defined map $d$  is a bijection between the set of these Yukawa graphs and the set of plane rooted trees. Furthermore, the number of meson lines corresponds to the number of vertices under $d$.
To see this, proceed inductively. The result holds for one meson $c_0$ and one vertex $v_0$ immediately from the definition.  Given $Y$ with more than one meson, applying the definition to the outermost meson $c_0$, we see that the children of the root $v_0$ are given by the subgraphs inserted under $c_0$.  Inductively, $d$ is bijective on these subgraphs and subtrees, and so $d(Y)$ can be uniquely obtained from $Y$, and the number of mesons corresponds to the number of vertices. 

\cref{figure yukawa insertion trees} shows the insertion trees corresponding to the Yukawa Feynman graphs in \cref{fig yukawa}. The insertion tree of the $n$-loop rainbow graph is the ladder $\ell_n$ from \cref{ex dse trees ladder}.

So far, $d$ has been a bijection between individual graphs and rooted trees. If we set $s=-2$ in the Dyson--Schwinger equation, we obtain precisely \emph{all} possible Yukawa graphs, and $d$ is a bijection to the set of \emph{all} plane rooted trees. Indeed, for $s=-2$,  the structure of the combinatorial Dyson--Schwinger equation for the Feynman graphs is identical to that of the combinatorial Dyson--Schwinger equation for plane rooted trees as described in \cref{ex trees negative s}.

As will be discussed in \cref{eg Yukawa chord}, the Yukawa propagator Feynman graphs can be viewed as indecomposable non-crossing chord diagrams, where the meson lines are the chords. Then, the above bijection $d$ is between indecomposable non-crossing chord diagrams and plane rooted trees. This bijection $d$ is different from the later $\theta$ (\cref{map tubed RTs to connected CDs}), which is between chord diagrams and \emph{tubed} rooted trees.  Rather, $d$ gives us the tree of a Yukawa graph, and then that is the tree we tube. $\theta$ then translates each tubing to a chord diagram to give the contribution of the Yukawa graph to the original chord diagram expansions. Note further that this concrete form of $d$ is valid only for the Yukawa example, whereas the bijection $\theta$, and the tubings, operate on the insertion trees and are agnostic with respect to the physical theory in question.

\end{example}

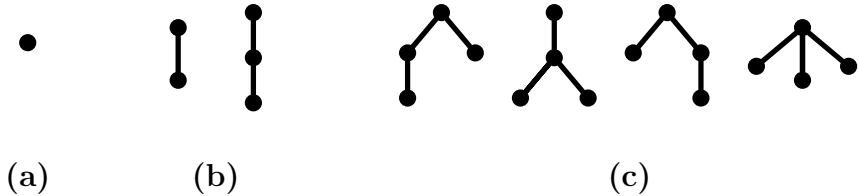
\begin{figure}[htb]
    \centering
    \begin{tikzpicture}
        \coordinate (dy) at (0,-.7);
    
        \node[vertex] at (-.5,.8){};
        \node at (-.5,-1){{\bfseries (a)}};

        \node[vertex](v1) at (1.5,1){};
        \node[vertex](v2) at ($(v1) + (dy)$){};
        \draw[edge] (v1)--(v2);
        \node[vertex](v1) at (2.5, 1.2){};
        \node[vertex](v2) at ($(v1) + (0,-.6)$){};
        \node[vertex](v3) at ($(v2) + (0,-.6)$){};
        \draw[edge] (v1) -- (v2) -- (v3);
        \node at (2,-1){{\bfseries (b)}};

        \node[vertex] (v1) at (5, 1.2){};
        \node[vertex] (v2) at ($(v1) + (230:.7)$){};
        \node[vertex] (v3) at ($(v1) + (310:.7)$){};
        \node[vertex] (v4) at ($(v2) + (270:.6)$){};
        \draw[edge] (v4) -- (v2) -- (v1) -- (v3);
        \node[vertex] (v1) at (6.5, 1.2){};
        \node[vertex] (v2) at ($(v1) + (270:.6)$){};
        \node[vertex] (v3) at ($(v2) + (230:.7)$){};
        \node[vertex] (v4) at ($(v2) + (310:.7)$){};
        \draw[edge] (v1) -- (v2) -- (v3) -- (v2) --(v4);
        \node[vertex] (v1) at (8, 1.2){};
        \node[vertex] (v2) at ($(v1) + (310:.7)$){};
        \node[vertex] (v3) at ($(v1) + (230:.7)$){};
        \node[vertex] (v4) at ($(v2) + (270:.6)$){};
        \draw[edge] (v4) -- (v2) -- (v1) -- (v3);
        \node[vertex] (v1) at (9.8, 1){};
        \node[vertex] (v2) at ($(v1) + (220:.8)$){};
        \node[vertex] (v3) at ($(v1) + (270:.7)$){};
        \node[vertex] (v4) at ($(v1) + (320:.8)$){};
        \draw[edge] (v2) -- (v1) -- (v3) -- (v1) -- (v4);

        \node at (7.5, -1){{\bfseries (c)}};

    \end{tikzpicture}
    \caption{Insertion trees corresponding to the Feynman graphs in \cref{fig yukawa}. {\bfseries (a)} the kernel itself is represented by a single vertex. {\bfseries (b)} Rainbow graphs correspond to ladder trees $\ell_n$ (\cref{sec ladders}). {\bfseries (c)} the choice $s=-2$ gives rise to all plane rooted trees, not just ladders. }
    \label{figure yukawa insertion trees}
\end{figure}

The Dyson--Schwinger equation implies the validity of the renormalization group equation, or Callan--Symanzik equation \cite{callan_broken_1970,symanzik_small_1970}.  For the higher coefficients $\gamma_{k>1}(x)$ in the series expansion \cref{G_log_expansion}, the latter reads \cite{kreimer_anatomy_2006,bergbauer_hopf_2006,kreimer_etude_2006,kreimer_etude_2008}
\begin{align}\label{gammak_rge}
   \gamma_1(x) = \gamma(x), \qquad  k \gamma_k(x) &= \gamma(x) \left( 1+sx \partial_x \right) \gamma_{k-1}(x) \qquad \forall k>1.
\end{align}
By virtue of \cref{gammak_rge}, the full solution \cref{G_log_expansion} of a DSE can be reconstructed entirely from the anomalous dimension $\gamma(x)$, that is, the linear coefficient $[L^1]G(x,L)$. 
The case of a linear Dyson--Schwinger equation, that is $s=0$, is particularly simple.  From \cref{gammak_rge} we obtain
\begin{align}
    \gamma_k(x) = \frac{1}{k!}\gamma(x)^k, \qquad G(x,L) = \exp \left( \gamma(x) \cdot L \right) = \left( \frac{q^2}{\mu^2}\right)^{\gamma(x)}.
\end{align}

In the case of only a single kernel graph and $s\neq 0$, one obtains the (pseudo) differential equation \cite{balduf_dyson_2024}
\begin{align}\label{anomalous_dimension_nonlinear}
    \frac{1}{\rho \cdot F(\rho)} \Big|_{\rho \rightarrow \gamma(x)(1+sx \partial_x)} \gamma(x) &= x
\end{align}
for the  anomalous dimension $\gamma(x)$, by inserting the expansion \cref{G_log_expansion} into the DSE \cref{analytic dse} and using \cref{gammak_rge}.

\begin{example}\label{eg yukawa ode}
For the Yukawa example, the Mellin transform was computed in \cref{eg yukawa mellin transform}, and \cref{anomalous_dimension_nonlinear} becomes
\begin{align*}
     -\left( 1+\gamma(x) (1+s x \partial_x) \right) \gamma(x) = x.
\end{align*}
This ODE, and its close cousin for $D=6$ dimensions, recently attracted interest for its non-perturbative solutions \cite{borinsky_nonperturbative_2020,borinsky_semiclassical_2021,borinsky_resonant_2022}.
\end{example}

For a linear DSE, where $s=0$, \cref{anomalous_dimension_nonlinear} is replaced by the algebraic equation \cite{kreimer_etude_2008}
  \begin{align}\label{anomalous_dimension_linear}
    \frac 1 {F(\gamma(x))} =  x \qquad \Leftrightarrow \qquad \gamma(x) = F^{-1} \left(  \frac 1 x \right).
\end{align}

For the bijection to tubings examined later, it will be useful to express \cref{anomalous_dimension_linear} in terms of the individual series coefficients. This is done by the following lemma:

\begin{lemma}\label{lem linear coefficients}
    Consider a linear DSE with a single kernel graph with Mellin transform \cref{mellin_transform} $F(\rho) = \sum_{j=0}^\infty c_j \rho^{j-1}$ (that is, we ware in the unweighted setting). Then, the coefficients of the anomalous dimension are $ [x^1]\gamma(x) = c_0$ and
    \begin{align*}
       [x^n] \gamma(x) &=  \sum_{j=1}^{n-1} \frac{c_0^{n-j}}{(n-j)!} B_{n-1,j} (1! c_1, 2! c_2, 3! c_3, \ldots),\qquad n>1,
    \end{align*}
    where $B_{n,k}(\ldots)$ is the partial Bell polynomial.
\end{lemma}

\begin{proof}
 
For the reader's convenience, we recall from Comtet's textbook \cite[p133]{comtet_advanced_1974} the well-known formula
\begin{align}\label{partial bell series}
    \frac{1}{j!}
    \Big(
    \sum_{m > 0} c_m t^m  
    \Big)^j 
    = \sum_{n \ge j} B_{n,j}(1!c_1, 2!c_2
    \ldots, (n-j+1)!c_{n-j+1}) \frac{t^n}{n!},
\end{align}
which defines the partial Bell polynomials 
$$
    B_{n,j}(1!c_1, 2!c_2
    \ldots, (n-j+1)!c_{n-j+1})
    :=\sum 
    \frac{n!}{i_1!i_2! \cdots i_{n-j+1}!}
    c_1^{i_1}c_2^{i_2} \cdots c_{n-j+1}^{i_{n-j+1}},
$$
where the sum on the righthand side runs over all sequences of non-negative integers $i_1,\ldots, i_{n-j+1}$ such that
$i_1+i_2+ \cdots + i_{n-j+1}= j$ and $i_1+ 2i_2+ \cdots + ({n-j+1})i_{n-j+1} = n$.

    The  claimed expression in \cref{lem linear coefficients} is a series expansion of \cref{anomalous_dimension_linear}. To derive it, we use the Lagrange inversion formula \cite{merlini_lagrange_2006}, starting from the following statement: If two power series $A(x)$ and $B(x)$ satisfy 
    \begin{align}\label{lif}
        A(x) =   x B(A(x)),
    \end{align}
    then the $n$-the coefficient of $A(x)$ can be computed from the series $B(x)$ according to
    \begin{align}\label{lif2}
        [x^n] A(x) =  \frac 1 n [x^{n-1}] B(x)^n
    \end{align}
    for $n > 0$.
    In our case, we let $A(x) =\gamma(x)$ be the anomalous dimension and we choose $B(x) =  x F(x)$. With this choice, \cref{lif} becomes $\gamma(x) =  x \cdot \gamma(x) \cdot F(\gamma(x))$, which is \cref{anomalous_dimension_linear}. 

    Now define $\bar F(\rho)$ by the equation
    \begin{align*}
        F(\rho) = \sum_{j=0}^\infty c_j \rho^{j-1} = \frac{c_0}\rho \left( 1+ \bar F(\rho)\right), \qquad \bar F (\rho) = \sum_{j=1}^\infty \frac{c_j}{c_0} \rho^j.
    \end{align*}
    From Faà di Bruno's formula and the binomial theorem, one has
    \begin{align*}
        \bar F(\rho)^j &=  \sum_{k=j}^\infty \frac{j!}{k!}c_0^{-j} B_{k,j} \left( 1! c_1, 2! c_2, 3! c_3, \ldots \right) \rho^k\\
		F(\rho)^n &= \frac{c_0^n}{\rho^n} \sum_{j=0}^n \binom{n}{j}\bar F(\rho)^j= \sum_{j=0}^n  \sum_{k=j}^\infty \frac{n!}{(n-j)!k!}c_0^{n-j} B_{k,j} \left( 1! c_1, 2! c_2, 3! c_3, \ldots \right) \rho^{k-n}.
    \end{align*}
    By \cref{lif2}, we need to extract the coefficient $[\rho^{n-1}]$ of $(\rho F(\rho))^n$, which amounts to extracting $[\rho^{-1}]$ from $(F(\rho))^n$. Using \cref{partial bell series}, we find
    \begin{align*}
        [x^n]\gamma(x) &= \frac{1}{n} [\rho^{-1}] \sum_{j=0}^n  \sum_{k=j}^\infty \frac{n!}{(n-j)!k!}c_0^{n-j} B_{k,j} \left( 1! c_1, 2! c_2, 3! c_3, \ldots \right) \rho^{k-n}.
    \end{align*}
    Taking the coefficient with respect to $k=n-1$ from this sum produces the claimed result. Observe that $j=0$ and $j=n$ do not contribute since the Bell polynomial vanishes.
\end{proof}

\begin{example}
\label{eg Yukawa ladders physically}
The rainbow approximation to the Yukawa fermion, introduced in \cref{eg yukawa 1}, is a linear DSE. Using the Mellin transform from \cref{eg yukawa mellin transform} in \cref{anomalous_dimension_linear}, we obtain the well-known  anomalous dimension  \cite{delbourgo_dimensional_1996,kreimer_etude_2008}
\begin{align*}
    \gamma(x) = \frac{\sqrt{1-4x}-1}{2}, \qquad \text{(for $s=0$ in Yukawa theory)}.
\end{align*}
Observe that our choice $G=1+L+\cdots$ in \cref{G_log_expansion} flips the sign of $x$ compared to some of the literature.  We'll return to this and the connection to Catalan numbers in \cref{sec ladders}.
\end{example}

\section{Special cases of tubings}
\label{sec examples} 

Having introduced the algebraic formalism and the physical motivation of the Dyson--Schwinger equations in question, we now return to the main topic of the present paper: the solution of the DSE in terms of tubings, claimed in \cref{thm main thm}. Before we prove this theorem, we illustrate its consequences in some more detailed examples.

\smallskip

Recall that a \emph{decreasing labelling} of a rooted tree $t$ with $n$ vertices is a labelling of the vertices by $1, \ldots, n$ such that every child has a smaller label than its parent.  This is a classical and well-studied notion, particularly in the equivalent form of \emph{increasing labellings} \cite{bergeron_varieties_1992}.

Our first example is a specific type of tubing called \emph{leaf tubing} which is defined to be a tubing where every bipartition of a tube divides the tube into a single leaf and the rest of the tube. The following lemma illustrates the relationship between decreasing labelling and leaf tubing.

\begin{lemma} \label{leaf_label_lemma}
Let $t$ be a rooted tree with $n$ vertices.  The number of leaf tubings of $t$ is the number of decreasing labellings of $t$.
\end{lemma}

Note that when a tube is divided into a single leaf and the rest, that leaf is a leaf in the induced subgraph of the vertices of the tube, and may or may not be a leaf in the full tree.

\begin{proof}

Label the root by the number of tubes it is contained in and label every other vertex by one less than the number of tubes it is contained in.  The claim is that this gives a decreasing labelling.

The leaf partitioned off at the first step is in exactly 2 tubes, the outermost tube and its own tube.  The single leaf partitioned of the remaining part at the next step (which may or may not be a leaf of the full tree), is in exactly 3 tubes, the outermost tube, the root tube from the first partitioning and its own tube.  Inductively, the leaf partitioned off at the $i$th step is in exactly $i+1$ tubes.  At the last step, which is the $n-1$st step, both the root and the leaf partitioned off are in $n$ tubes.   Therefore, the labelling given above does label the vertices by $1,\ldots, n$ with no repetition.

It remains to check that every child's label is smaller than its parent's label.  This is true of the root and the vertex partitioned off from the root at the last step since each is in $n$ tubes, but only the non-root has this number decreased to get its label.  For every other pair of parent and child, the two vertices must be in a different number of tubes so it suffices to check that the child is in no more tubes than the parent.  To see this, note that by the choice of kind of tubing we are considering, every tube of size $>1$ includes the root and this means that a child can be in as most as many tubes of size $>1$ as its parent, then that the fact that every vertex is in exactly one tube of size 1 gives the result.

\end{proof}

Under the bijection with rooted connected chord diagrams we will define in \cref{sec chord}, the special tubings considered in this example correspond to a particularly interesting class of chord diagrams; see \cref{prop interesting chord classes}.

\medskip

Another interesting pair of examples comes from looking at special trees, such as ladders and corollas, rather than special tubings.

\subsection{Ladders}
\label{sec ladders}

Recall the ladders, $\ell_n$ from \cref{ex dse trees ladder}, which arise as the solution of a linear DSE. For concreteness, we assume that there is only one cocycle; otherwise the vertices of the ladders would be decorated. The Yukawa rainbows discussed in \cref{eg Yukawa ladders physically} give one such example, but we make no assumption on $F(\rho)$ at this point.
Now we have $s=0$  (for being linear) and $k=1$ (for having only one cocycle) in the combinatorial DSE \cref{dse algebraic general}:
\begin{align}\label{eq lin comb DSE}
   T(x) = 1+ x B_+ \big( T(x) \big).
\end{align}
The generating function of ladders, that is, the series $T(x) = \sum_{n\geq 0}x^n\ell_n$, is the unique power series solution of \cref{eq lin comb DSE}, as can be seen iteratively, starting from $B_+(1)=\rtLine 1$.
In this way, we say that \cref{eq lin comb DSE} generates the family of ladders.  The physical DSE corresponding to \cref{eq lin comb DSE} is  
\[
G(x,L)=1+xG(x, \partial_\rho)(e^{L\rho}-1)F(\rho)\bigg|_{\rho=0}
\]
which are \cref{dse_differential_general,analytic dse} with $k=1$ and $s=0$.

\medskip

\begin{figure}[htb]
    \centering
    \begin{tikzpicture}

        \coordinate (v1) at (0,1.5){};
        \coordinate (v2) at (0,1){};
        \coordinate (v3) at (0,.5){};
        \coordinate (v4) at (0,0){};
        \draw [tube2] (v1) -- (v3);
        \draw [tube1] (v1) -- (v2);
        \draw [edge] (v1) to (v2) to (v3) to (v4);
        \node [vertex] at (v1){};
        \node [vertex] at (v2){};
        \node [vertex] at (v3){};
        \node [vertex] at (v4){};

        \coordinate (v1) at (1.5,1.5){};
        \coordinate (v2) at (1.5,1){};
        \coordinate (v3) at (1.5,.5){};
        \coordinate (v4) at (1.5,0){};
        \draw [tube2] (v1) -- (v3);
        \draw [tube1] (v2) -- (v3);
        \draw [edge] (v1) to (v2) to (v3) to (v4);
        \node [vertex] at (v1){};
        \node [vertex] at (v2){};
        \node [vertex] at (v3){};
        \node [vertex] at (v4){};

        \coordinate (v1) at (3,1.5){};
        \coordinate (v2) at (3,1){};
        \coordinate (v3) at (3,.5){};
        \coordinate (v4) at (3,0){};
        \draw [tube2] (v2) -- (v4);
        \draw [tube1] (v2) -- (v3);
        \draw [edge] (v1) to (v2) to (v3) to (v4);
        \node [vertex] at (v1){};
        \node [vertex] at (v2){};
        \node [vertex] at (v3){};
        \node [vertex] at (v4){};

        \coordinate (v1) at (4.5,1.5){};
        \coordinate (v2) at (4.5,1){};
        \coordinate (v3) at (4.5,.5){};
        \coordinate (v4) at (4.5,0){};
        \draw [tube2] (v2) -- (v4);
        \draw [tube1] (v3) -- (v4);
        \draw [edge] (v1) to (v2) to (v3) to (v4);
        \node [vertex] at (v1){};
        \node [vertex] at (v2){};
        \node [vertex] at (v3){};
        \node [vertex] at (v4){};

        \coordinate (v1) at (6,1.5){};
        \coordinate (v2) at (6,1){};
        \coordinate (v3) at (6,.5){};
        \coordinate (v4) at (6,0){};
        \draw [tube1] (v1) -- (v2);
        \draw [tube1] (v3) -- (v4);
        \draw [edge] (v1) to (v2) to (v3) to (v4);
        \node [vertex] at (v1){};
        \node [vertex] at (v2){};
        \node [vertex] at (v3){};
        \node [vertex] at (v4){};
    \end{tikzpicture}
    \caption{The five different tubings of $\ell_4$. Note that we skipped the inner- and outermost tubes.}
    \label{figure ladders}
\end{figure}
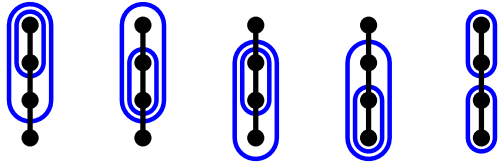

Let us consider the content of the main theorem (\cref{thm main thm}) in this case. 
The tubings of ladders can be described very nicely. We saw in \cref{fig:small_example,figure ladders} that $\ell_3$ has two and $\ell_4$ has five tubings.

\begin{lemma}\label{lem ladders tubings}
For $n\geq0$, $N(\ell_{n+1})=C_n$, where $C_n = \frac{1}{n+1}\binom{2n}{n} $ are the Catalan numbers.
\end{lemma}

\begin{proof}
We use induction. There is exactly one ladder tree $\ell_n$ for each $n\in \mathbb{N}$. Removing one edge $e$ amounts to splitting $\ell_n$ into two ladders, $\ell_n = \ell_{j} \cup e \cup \ell_{n-j}$, where $0< j <n$. 

As seen explicitly above, the claim of the lemma holds for $\ell_{n\leq 4}$. Assume, $N(\ell_n)= C_{n-1}$ holds up to some number $n$ of vertices. Then, by \cref{lem recursive tubing count}, 
\begin{align*}
    N(\ell_{n+1}) = \sum_{j=1}^{n} N(\ell_j)N(\ell_{n+1-j}) = \sum_{j=0}^{n-1} C_{j} C_{n-1-j} = C_n.
\end{align*}
The sum is the well-known recursive definition of the Catalan numbers \cite{stanley_catalan_2015}.
\end{proof}

\begin{example}
    In the special case of the Yukawa theory, \cref{eg Yukawa ladders physically}, the Mellin transform \cref{mellin_transform} assigns the value $c_i=-(-1)^i$ to all $i$. Any tubing of $\ell_n$ contains exactly $2n-1$ tubes, the Mellin monomial (\cref{def Mellin monomial}) along with the order $L^1$ factor for the root, $c(\tau) c_{b(\tau)-1}$, involves exactly $n$ factors. The sum of the indices of these factors is therefore $(2n-1)-n\cdot 1 = n-1$ and the product of the Mellin coefficients is $(-1)$ for every tubing of $\ell_n$.  By \cref{thm main thm}, the contribution of any tubing $\tau$ of a ladder to the anomalous dimension is $[L^1] \phi_L(\tau) =-1$.
Consequently, the contribution of $\ell_n$ to the $L^1$ coefficient of $G(x,L)$, that is to the anomalous dimension, is exactly the number of tubings of $\ell_n$, namely  $-1\cdot C_{n-1}$. The constant $\ell_0=1$ does not contribute to the anomalous dimension. Summing over all orders $n\in [1,\infty)$, we obtain the generating function of the Catalan numbers, which reproduces  \cref{eg Yukawa ladders physically}:
\begin{align*}
    [L^1] G(x,L) = \gamma(x) = \sum_{n=1}^\infty (-1)  C_{n-1}  x^n = \frac{\sqrt{1-4x}-1}{2}.
\end{align*}
\end{example}

\noindent
For the ladder trees, it is possible to find an explicit formula for the sum of their tubings:

\begin{proposition}\label{prop tubings ladder}
    Consider a single Mellin transform \cref{mellin_transform} $F(\rho)=\sum_{j\geq 0}c_j \rho^{j-1}$ (that is, unweighted trees) and let $\ell_n$ be the ladder tree on $n$ vertices.  Then, the contribution of  $\ell_n$ to the anomalous dimension is
        \begin{align*}
             \sum_{\tau \in \Tub(\ell_n)} [L^1] \phi_L(\tau)=   \sum_{j=1}^{n-1} \frac{c_0^{n-j}}{(n-j)!} B_{n-1,j} (1! c_1, 2! c_2, 3! c_3, \ldots).
        \end{align*}
\end{proposition}
\begin{proof}
	This statement is a special case of point 3 of \cref{thm main thm}. 
	Recall that we are still working in the case of trees without explicit weights. As defined above \cref{def tube}, the quantity $w(t)$ of a tree  in that case simply amounts to the number of vertices. As $\ell_n$ has $n$ vertices, we have $w(\ell_n)=n$. 
 There is exactly one ladder tree $\ell_n$ for each $n\in \mathbb{N}$. Consequently, a sum over all trees $t$ in \cref{thm main thm} amounts to a sum over $n\in \mathbb N$.
 
 As discussed in \cref{ex dse trees ladder}, ladders correspond to the choice $s=0$ in the   DSE \cref{dse_differential_general}, and hence in \cref{thm main thm}.  Except for the leaf, every vertex has  $\od v=1$ children, and the falling factorials (defined in \cref{def falling factorial}) become $(1)_1=1=(1)_0$ for all vertices.  Therefore, the Green function $G(x,L)$ as given by point 2 of \cref{thm main thm} reduces to a sum over $\phi_L(\tau)$, 
 \begin{align*}
 G(x,L) &= 1+\sum_{n=1}^\infty \frac{x^n}{1} \left( \prod_{v\in V(\ell_n)} 1 \right) \sum_{\tau \in \Tub(\ell_n)} \phi_L(\tau) = 1+\sum_{n=1}^\infty  x^n  \sum_{\tau \in \Tub(\ell_n)} \phi_L(\tau).
 \end{align*}
 Recall that here, and in all computations in this paper, the order-zero term of $G(x,L)$ is unity by definition since we are using kinematic renormalization conditions.
 By \cref{G_log_expansion}, the term of order $L^1$ in  $G(x,L)$ is the anomalous dimension. Using point 3 of \cref{thm main thm} (in the unweighted case), the term in question is 
 \begin{align*}
    \sum_{n=1}^\infty x^n  [L^1] \phi_L(\ell_n)=\gamma(x)  = \sum_{n=1}^{\infty} x^n \sum_{\tau \in \Tub(\ell_n)}\prod_{ v\in V(\ell_n)} c_{b(v,\tau)-1} .
  \end{align*}
 On the other hand, the coefficients of the solution of a linear DSE are given by \cref{lem linear coefficients}. Equating the two formulas for $\gamma(x)$ and comparing coefficients results in the claimed expression for the contribution of a particular ladder $\ell_n$.
\end{proof}

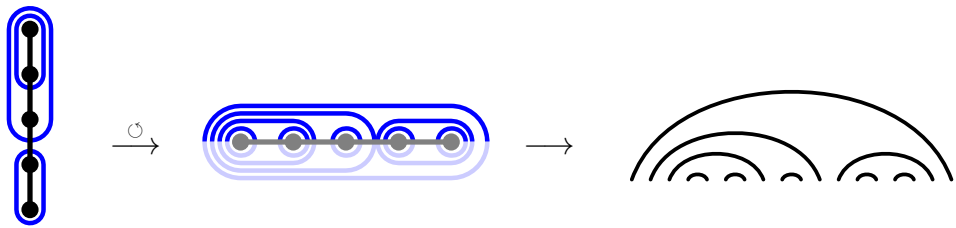
\begin{figure}[htb]
    \centering
    \begin{tikzpicture}

        \coordinate (v1) at (0,0){};
        \coordinate (v2) at (0,-.6){};
        \coordinate (v3) at (0,-1.2){};
        \coordinate (v4) at (0,-1.8){};
        \coordinate (v5) at (0,-2.4){};
        \draw [tube2] (v1) -- (v3);
        \draw [tube1] (v1) -- (v2);
        \draw [tube1] (v4) -- (v5);
        \draw [edge] (v1) to (v2) to (v3) to (v4) to (v5);
        \node [vertex] at (v1){};
        \node [vertex] at (v2){};
        \node [vertex] at (v3){};
        \node [vertex] at (v4){};
        \node [vertex] at (v5){};

        \node[anchor=south] at (1.4, -1.8) {$\overset \circlearrowleft \longrightarrow$};

        \coordinate (v1) at (2.8,-1.5){};
        \coordinate (v2) at (3.5,-1.5){};
        \coordinate (v3) at (4.2,-1.5){};
        \coordinate (v4) at (4.9,-1.5){};
        \coordinate (v5) at (5.6,-1.5){};
        \draw [blue, line width = 10.2mm, line cap=round] (v1) -- (v5);
        \draw [white, line width = 9mm,line cap=round] (v1) -- (v5);
        \draw [tube3] (v1) -- (v3);
        \draw [tube2] (v1) -- (v2);
        \draw [tube2] (v4) -- (v5);
        \draw [tube1] (v1) -- (v1);
        \draw [tube1] (v2) -- (v2);
        \draw [tube1] (v3) -- (v3);
        \draw [tube1] (v4) -- (v4);
        \draw [tube1] (v5) -- (v5);
    
        \draw [white, fill=white, opacity=.8] (2.2,-1.5) rectangle ( 6.5,-2.5);
        \draw [edge,gray] (v1) to (v2) to (v3) to (v4) to (v5);
        \node [vertex,gray] at (v1){};
        \node [vertex,gray] at (v2){};
        \node [vertex,gray] at (v3){};
        \node [vertex,gray] at (v4){};
        \node [vertex,gray] at (v5){};

        \node[anchor=south] at (6.9, -1.8) {$\longrightarrow$};

        \coordinate (v1) at (8,-2){};
        \coordinate (v2) at (8.25,-2){};
        \coordinate (v3) at (8.5,-2){};
        \coordinate (v4) at (8.75,-2){};
        \coordinate (v5) at (9,-2){};
        \coordinate (v6) at (9.25,-2){};
        \coordinate (v7) at (9.5,-2){};
        \coordinate (v8) at (9.75,-2){};
        \coordinate (v9) at (10,-2){};
        \coordinate (v10) at (10.25,-2){};
        \coordinate (v11) at (10.5,-2){};
        \coordinate (v12) at (10.75,-2){};
        \coordinate (v13) at (11,-2){};
        \coordinate (v14) at (11.25,-2){};
        \coordinate (v15) at (11.5,-2){};
        \coordinate (v16) at (11.75,-2){};
        \coordinate (v17) at (12,-2){};
        \coordinate (v18) at (12.25,-2){};
            
        \draw [chord] (v1) to (v18);
        \draw [chord] (v2) to (v11);
        \draw [chord] (v3) to (v8);
        \draw [chord] (v4) to (v5);
        \draw [chord] (v6) to (v7);
        \draw [chord] (v9) to (v10);
        \draw [chord] (v12) to (v17);
        \draw [chord] (v13) to (v14);
        \draw [chord] (v15) to (v16);
    \end{tikzpicture}
    \caption{Turning a tubing of $\ell_5$ into a non-crossing chord diagram $C$. Observe that, as usual, we skipped the innermost and outermost tubes in the figure on the left.}
    \label{fig:ladders_chords}
\end{figure}

\begin{remark}\label{rem ladders chord diagrams}

The tubings $\tau$ of a ladder $\ell_n$ are in bijection to binary non-crossing chord diagrams $C$, that is, chord diagrams (\cref{rooted chord diagram}) where the chords do not intersect and below every chord, there are either exactly two or zero smaller chords. This bijection is constructed graphically by turning $\ell_n$ horizontally and removing the lower halves of all tubes of $\tau$. The remaining arcs form the chords of the chord diagram $C$, see \cref{fig:ladders_chords}. 

Later, in \cref{sec chord}, we construct a bijection $\theta$ (\cref{map tubed RTs to connected CDs}) between binary tubings and chord diagrams which is different from the \enquote{naive} mapping to $C$ shown in \cref{fig:ladders_chords}, compare \cref{figure bijection example ladders}. Namely, $C$ contains one chord for every tube, while the latter map gives one chord for every vertex.  Restricting $\theta$ to only those chord diagrams obtained from tubings of ladders will define a different Catalan subclass of chord diagrams, see \cref{prop interesting chord classes}. 

\end{remark}
\begin{remark}\label{rem ladders trees}
 
    As seen in \cref{rem tubings binary tree}, a binary tubing $\tau$ has the structure of a binary rooted tree $B$, where the tubes correspond to vertices. For a generic rooted tree $t$, the underlying rooted tree structure can not be recovered from the corresponding $B$ alone, only the information contained in the sizes of tubes.  However, for $t=\ell_n$ this information is sufficient; the sizes of the tubes is the only information needed to reconstruct $t$. Given a tube containing $\ell_n$, there is one way to partition it into subtubes containing $\ell_i$ and $\ell_{n-i}$ as non-root and root part respectively.  Thus every full binary tree on $n$ vertices appears through this construction from a tubing of $\ell_n$ and every tubing of $\ell_n$ gives a distinct full binary tree, see \cref{fig:ladders_trees}.

The binary tree $B$  is related to the chord diagram $C$ constructed in \cref{rem ladders chord diagrams}: If $C$ were a Yukawa-type Feynman graph, then $B$ would be the corresponding insertion tree by the map $d$ constructed in \cref{eg yukawa insertion trees}. This correspondence is coincidence and comes from the fact that both $C$ and $B$ are equivalent ways of drawing a tree structure, it is \emph{not} related to Yukawa theory. In particular, the Feynman graphs corresponding to ladders $\ell_n$ are rainbows (\cref{eg yukawa 1}), which are generally not of the shape $C$.
\end{remark}

\begin{figure}[htb]
    \centering
    \begin{tikzpicture}

        \coordinate (v1) at (0,0){};
        \coordinate (v2) at (0,-.6){};
        \coordinate (v3) at (0,-1.2){};
        \coordinate (v4) at (0,-1.8){};
        \coordinate (v5) at (0,-2.4){};
        \draw [tube2] (v1) -- (v3);
        \draw [tube1] (v1) -- (v2);
        \draw [tube1] (v4) -- (v5);
        \draw [edge] (v1) to (v2) to (v3) to (v4) to (v5);
        \node [vertex] at (v1){};
        \node [vertex] at (v2){};
        \node [vertex] at (v3){};
        \node [vertex] at (v4){};
        \node [vertex] at (v5){};

        \node at (1.5, -1.5) {$\rightarrow$};

        \coordinate (v1) at (2.8,-2){};
        \coordinate (v2) at (3.5,-2){};
        \coordinate (v3) at (4.2,-2){};
        \coordinate (v4) at (4.9,-2){};
        \coordinate (v5) at (5.6,-2){};

        \draw [line width = .7mm, blue,rounded corners=3mm,line cap=round,opacity=.4] ($(v1) + (-.25,.5)$) rectangle ($(v2) + (.25,-.5)$);
        \draw [line width = .7mm, blue,rounded corners=5mm,line cap=round,opacity=.4] ($(v1) + (-.4,1)$) rectangle ($(v3) + (.25,-.5)$);
        \draw [line width = .7mm, blue,rounded corners=3mm,line cap=round,opacity=.4] ($(v4) + (-.25,.5)$) rectangle ($(v5) + (.25,-.5)$);
        \draw [line width = .7mm, blue,rounded corners=7mm,line cap=round,opacity=.4] ($(v1) + (-.55,1.5)$) rectangle ($(v5) + (.4,-.8)$);

        \draw [white, fill=white ] (2,-2.1) rectangle ( 6.5,-2.9);

        \node[vertex] at (3.15, -1.5){};
        \node[vertex] at (3.5, -1){};
        \node[vertex] at (5.3, -1.5){};
        \node[vertex] at (4.2, -.5){};

        \draw [edge,gray] (v1) to (v2) to (v3) to (v4) to (v5);
        \node [vertex] at (v1){};
        \node [vertex] at (v2){};
        \node [vertex] at (v3){};
        \node [vertex] at (v4){};
        \node [vertex] at (v5){};

        \node at (6.75, -1.5) {$\rightarrow$};

        \coordinate (v1) at (8,-2.5){};
        \coordinate (v2) at (9,-2.5){};
        \coordinate (v3) at (10,-2.5){};
        \coordinate (v4) at (11,-2.5){};
        \coordinate (v5) at (12,-2.5){};

        \node [vertex] at (v1){};
        \node [vertex] at (v2){};
        \node [vertex] at (v3){};
        \node [vertex] at (v4){};
        \node [vertex] at (v5){};

        \coordinate (v6) at (8.5,-1.7){};
        \coordinate (v7) at (9,-.8){};
        \coordinate (v8) at (11.5,-1.2){};
        \coordinate (v9) at (10,0){};
    
        \node [vertex] at (v6){};
        \node [vertex] at (v7){};
        \node [vertex] at (v8){};
        \node [vertex] at (v9){};

        \draw[edge] (v9)--(v8);
        \draw[edge] (v9)--(v7);
        \draw[edge] (v7)--(v6);
        \draw[edge] (v7)--(v3);
        \draw[edge] (v6)--(v1);
        \draw[edge] (v6)--(v2);
        \draw[edge] (v8)--(v4);
        \draw[edge] (v8)--(v5);
    \end{tikzpicture}
    \caption{Turning a tubing of $\ell_5$ into a binary tree. Every tube is identified with one vertex of the binary tree. Observe that we skipped the innermost and outermost tubes in the figure on the left.}
    \label{fig:ladders_trees}
\end{figure}
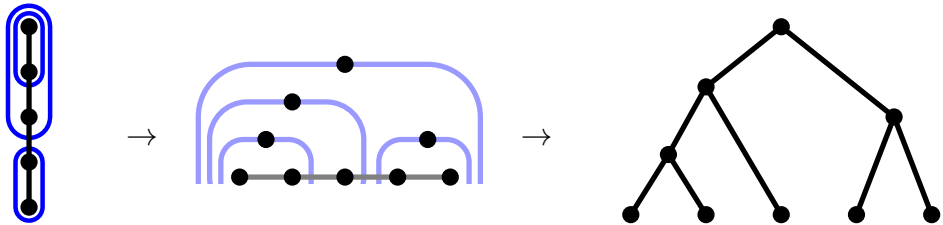

\FloatBarrier

\subsection{Corollas or chains} \label{sec corollas}

Corollas are trees where all non-root vertices are leaves.
\[
s_0 = 1, \qquad  s_1  =\rtLine 1, \qquad s_2 = \rtLine 2, \qquad s_3 = \rtV, \qquad s_4 = \rtW, \ldots
\]

The tubings of $s_n$ are easy to describe.  The only way to partition a corolla into two pieces, both of which are connected, is to have one piece be the corolla of size one smaller and the other piece be a single leaf.  A tubing then is a choice of leaf at this outermost level, a choice of one of the remaining leaves at the next level, and so on, see \cref{figure corollas}.  That is, tubings of $s_n$ are in bijection with permutations of the leaves.  Consequently, using \cref{lem recursive tubing count}, the number of tubings is $N(s_{n+1}) = n\cdot N(s_n)$. This proves the following lemma.

\begin{lemma}\label{lem chains tubings}
$s_{n+1}$ has $n!$ tubings, $N(s_{n+1})=n!$
\end{lemma}

\begin{figure}[htb]
    \centering
    \input{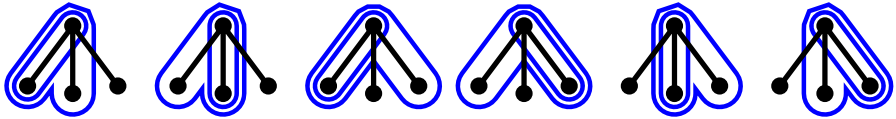}
    \caption{The six different tubings of $s_4$.}
    \label{figure corollas}
\end{figure}

The analogue of \cref{prop tubings ladder} for the case of corollas follows directly from the above characterization of the tubings. 
The exact form for the contribution of the corollas to $G(x,L)$ in \cref{thm main thm} (in the unweighted case) is
\begin{align}
    \phi_L(s_n) = n!c_0^{n-1}\left(\sum_{i=1}^{n}c_{n-i} \frac{L^i}{i!}\right),
\end{align}
and in particular the contribution of $c_n$ to the anomalous dimension is $n!c_0^{n-1}c_{n-1}L$.

As Feynman graphs, the corollas correspond to a chain of one loop bubbles inside one outer insertion.  The sum of them amounts to the \emph{chain approximation}. Unlike the ladder approximation, the chain approximation does not arise from a Dyson--Schwinger equation and so does not have all the properties one would want, physically, see \cite{balduf_dyson_2023}. Nonetheless, it can serve as an interesting mathematical example.

Ladders (\cref{lem ladders tubings}) and corollas (chains) (\cref{lem chains tubings}) lead to the following bound.  Another approach to these bounds is via the connection to associahedra in \cref{sec associahedra} and then using vertex counting results on polytopes, but that would bring us too far afield. 

\begin{lemma}\label{tubings_bound}
The number of tubings $N(t)$ of a rooted tree $t$ with $n$ vertices is bounded between the Catalan number $C_{n-1}$ and the factorial $(n-1)!$
\end{lemma}
The idea here is that the ladders and corollas are the extreme cases for the number of tubings a rooted tree can have.

\begin{proof}
We use induction. The claim is correct for all trees which contain 3 or fewer vertices, by explicit enumeration as shown above. Let $t$ be a tree which contains $k<n$ vertices, and assume that $C_{k-1}\leq N(t) \leq (k-1)!$ holds. We want to show that the same holds for a tree which has $n$ vertices. 
Using \cref{lem recursive tubing count}, the number of tubings of a tree $t$ is $ N(t) = \sum_{e\in E(t)} N(t_1) N(t_2)$, where $t_1,t_2$ are the two trees that arise from removing edge $e$. The trees $t_1$ and $t_2$ contain $1\leq n_1 <n$ and $n-n_1$ vertices respectively.

First, consider the upper bound. By assumption, $N(t_1) \leq (n_1-1)!$ and $N(t_2)\leq (n-n_1-1)!$. For a fixed $n$, the product $(|V(t_1)|-1)!(|V(t_2)|-1)!$ is a concave function, it reaches its maximum when either $n_1=1$ or $n_1=n-1$, and in both cases $\max\left((|V(t_1)|-1)!(|V(t_2)|-1)!\right) \leq 1\cdot (n-2)!$. The tree $t$ contains $(n-1)$ edges, hence, the sum in \cref{lem recursive tubing count} has $(n-1)$ terms. The maximum value it can attain is therefore
\begin{align*}
    N(t) &= \sum_{e\in E(t)} N(t_1) N(t_2) \leq (n-1) \cdot 1\cdot (n-2)!=(n-1)!
\end{align*}
as claimed. The upper bound corresponds to the corollas (\cref{lem chains tubings}), they are \enquote{as dense as possible}, removing any of the edges splits off only a single vertex.

\medskip

For the lower bound, first recall that by \cref{lem recursive tubing count} the number of tubings is not affected by the choice of root of the tree, that is, two rooted trees with the same underlying unrooted tree have the same number of tubings.

The next thing we need to consider is a vector recording how many ways a tree can be broken into pieces of particular sizes.  Specifically, given a tree $t$ with $n$ vertices (which may be unrooted or may be rooted or plane rooted, but neither the root nor the plane structure are needed), let $(n_1, n_2, n_3, \ldots, n_{\lfloor n/2\rfloor})$ be defined so that $n_i$ is the number of edges $e$ of $t$ with the property that removing $e$ leaves a tree of size $i$ and a tree of size $n-i$.

Let $t_1$ and $t_2$ be trees with $n$ vertices and $(n_1^{(1)}, n_2^{(1)}, \ldots, n_{\lfloor n/2 \rfloor}^{(1)})$ and $(n_1^{(2)}, n_2^{(2)}, \ldots, n_{\lfloor n/2 \rfloor}^{(2)})$ respectively the vectors defined above.  Say a tree $t_1$ is \emph{more evenly broken} than $t_2$ if for each $1\leq i \leq \lfloor n/2\rfloor$, $n_1^{(1)} + \cdots + n_i^{(1)} \leq n_1^{(2)} + \cdots + n_i^{(2)}$ and at least one of the inequalities is strict.

The claim is that no tree on $n$ vertices is more evenly broken than $\ell_n$.  To prove this claim let $t$ be any tree on $n$ vertices which does not have the same underlying unrooted tree as $\ell_n$, that is, $t$ is not a path in the graph theory sense.  

Suppose for a contradiction that $t$ is more evenly broken than $\ell_n$ and that the length of $t$'s longest path is longest among all trees on $n$ vertices which are more evenly broken than $\ell_n$.  Let $P$ be a longest path in $t$.  Since $t$ is not a path, there is at least one vertex on $P$ which has a neighbour not on $P$. Let $v$ be such a vertex that is as close a possible to an end of $P$.  Note that $v$ cannot be an end of $P$ since then $P$ would not be longest, and consequently $v$ cannot be a leaf.  Now we will consider three important connected subgraphs of of $t$.  Let $P_1$ be the part of $P$ from $v$ to the nearest end of $P$.  $P_1$ has at least 2 vertices, is a path, and  $v$ is the only vertex of $P_1$ with neighbours which are not in $P_1$.  Let $t_2$ be the component of $t-E(P)$ which includes $v$.  Note that $t_2$ contains at least two vertices and is a tree.  $t_2$ can be alternately characterized as the maximal subtree of $t$ containing $v$ but not containing any edges of $P$.  Finally, let $t_1$ be the subgraph of $t$ given by removing the edges of $P_1$ and the edges of $t_2$ and removing all isolated vertices this creates.  Note that $t_1$ is a tree and has at least as many vertices as $P_1$ by choice of $v$.   Also, note that $P_1$, $t_1$ and $t_2$ all include $v$ and identifying their three copies of $v$ we reobtain $t$. An example of this decomposition is shown in \cref{figure proof N}.

\begin{figure}[htb]
    \begin{center}
        \input{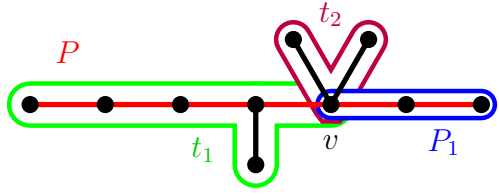}
    \end{center}
    \caption{Decomposition of a tree into $t_1, t_2$ and $P_1$. The longest path $P$ is indicated in red.}
    \label{figure proof N}
\end{figure}

Now using $P_1$, $t_1$ and $t_2$ build another tree $t'$ by identifying the copies of $v$ in $P_1$ and $t_1$ and identifying the copy of $v$ in $t_2$ with the other end of $P_1$.  Then $t'$ is more evenly broken than $t$.  To see this, removing any edge that is not in $P_1$ results in components of the same size in $t$ and in $t'$. Consider then the edge which breaks $P_1$ into a path of size $i_1$ including $v$ and a path of size $i_2$ not including $v$.  In $t$ cutting this edge gives a component of size $i_2$ and a component of size $|t_1| + i_1 + |t_2| -2$, while in $t'$ cutting this edge gives a component of size $|t_2|+i_2-1$ and a component of size $|t_1| + i_1-1$.  Now, $i_2 < |P_1| \leq |t_1|$ and $i_1$, $i_2$ and $|t_2|$ are each at least 1, so the smallest of the four components whose sizes are listed above is the one of size $i_2$, so this cut is less even in $t$ than in $t'$.  Putting all this together $t'$ is more evenly broken than $t$.

\begin{figure}[htb]
	\begin{center}
		\input{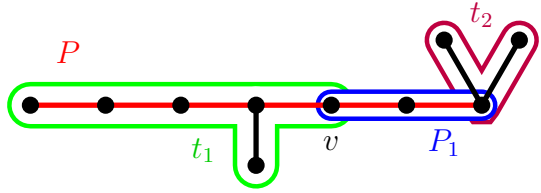}
	\end{center}
	\caption{The tree $t'$ constructed from recombining the pieces $t_1, t_2$ and $P_1$ from \cref{figure proof N} as described in the proof.}
	\label{figure proof N 2}
\end{figure}

Now by assumption $t$ is more evenly broken than $\ell_n$ and we just proved $t'$ is more evenly broken than $t$, so $t'$ is more evenly broken than $\ell_n$, but $t'$ has a longer longest path than $t$, contradicting the choice of $t$.  Therefore no tree on $n$ vertices is more evenly broken than $\ell_n$.

With this claim proven we can now prove that no tree on $n$ vertices has fewer tubings than $\ell_n$ by induction.  The result holds for $n=1$, suppose it holds for all trees on $<n$ vertices and consider $t$ on $n$ vertices and let $(n_1, n_2, \ldots, n_{\lfloor n/2 \rfloor})$ be the vector counting how it breaks.  $t$ is no more evenly broken than $\ell_n$ so $n_1+\cdots + n_i \geq 2i$ for $1\leq i < \lfloor n/2 \rfloor$ and $n_1+\cdots + n_{\lfloor n/2 \rfloor} = |E(t)| = n-1$.  By \cref{lem recursive tubing count} we can count tubings by counting tubings of the two connected subgraphs for each way of cutting one edge of $t$.  Using the inductive bound for those subgraphs, which are themselves rooted trees, we see that the number of tubings of $t$ is at least
\[
\sum_{i=1}^{\lfloor n/2\rfloor} n_i C_{i-1}C_{n-i-1}
\]
Now we need a final claim: we claim that for $i\leq n/2$,  $C_{i-1}C_{n-i-1}\leq C_{i-2}C_{n-i}$.  Let us notice why this claim will suffice to prove the lower bound and then we will prove the claim.  If the claim holds then since $t$ is no more evenly broken than $\ell_n$ we have, using a telescoping rewriting,
\begin{align*}
|\Tub(t)| & \geq \sum_{i=1}^{\lfloor n/2\rfloor} n_i C_{i-1}C_{n-i-1} \\
& = (C_{\lfloor n/2\rfloor-1}C_{n-\lfloor n/2\rfloor -1})\sum_{i=1}^{\lfloor n/2\rfloor} n_i  + \sum_{j=2}^{\lfloor n/2 \rfloor} (C_{j-2}C_{n-j} - C_{j-1}C_{n-j-1}) \sum_{i=1}^{j-1}n_i \\
& \geq  (C_{\lfloor n/2\rfloor-1}C_{n-\lfloor n/2\rfloor -1})\sum_{i=1}^{\lfloor n/2\rfloor} n_i^{(\ell_n)} + \sum_{j=2}^{\lfloor n/2 \rfloor} (C_{j-2}C_{n-j} - C_{j-1}C_{n-j-1}) \sum_{i=1}^{j-1}n_i^{(\ell_n)} \\
& = \sum_{i=1}^{\lfloor n/2\rfloor} n_i^{(\ell_n)} C_{i-1}C_{n-i-1} \\
& = |\Tub(\ell_n)| = C_{n-1}
\end{align*}
where $n_i^{(\ell_n)} = 2$ for $i<n/2$ and for $n$ even $n_{n/2}=1$, proving the lower bound.

It remains now to prove the final claim.  This is a small algebraic exercise in binomial coefficients:
\[
    \frac{C_{i-1}C_{n-i-1}}{C_{i-2}C_{n-i}} = \frac{(2i-3)(n-i+1)}{i(2n-2i-1)}
\]
We wish to prove this rational function is $\leq 1$ for $1\leq i\leq n/2$.  The denominator is positive in this range so it suffices to show that the numerator minus the denominator is less than 0.  The numerator minus the denominator is $(2i-3)(n-i+1) - i(2n-2i-1) = 6i-3n-3$ which is less than $0$ provided $i\leq n/2 + 1/2$ and proving the claim.
\end{proof}

Note that this bound is very weak for high numbers of vertices since $C_n \sim \frac{4^n}{(n+1)\sqrt n}$ grows slower than factorially.

\begin{example}\label{ex yukawa bound}
    For the Yukawa theory in $D=4$ dimensions, the bound on the number of tubings, \cref{tubings_bound}, allows us to deduce a bound on the coefficients of the anomalous dimension itself.  This is because the Mellin transform, computed in \cref{eg yukawa mellin transform}, has the particularly simple form $c_j = (-1)^{j+1}$.
    A tubing $\tau$ of a tree on $n$ vertices consists of $2n-1$ tubes (\cref{lem recursive tubing}). The $b$-statistic (\cref{def
    b statistic}) is the number of tubes of which a given vertex is root. Consequently, the $n$
    vertices are, together, root of $2n-1$ tubes, and the Mellin monomials and factors for the root evaluate to an overall prefactor
\begin{align*}
    \prod_{v \in V(t)} c_{b(v,\tau)-1} = \prod_{v\in V(t)} (-1)^{b(v,\tau)} = (-1)^{2n-1} = -1.
\end{align*}
By \cref{thm main thm}, still in the unweighted case, the anomalous dimension    is 
\begin{align*}
    \gamma(x) = (-1) \cdot \sum_t \frac{x^{w(t)}}{|\operatorname{Aut}(t)|} \left( \prod_{v\in V(t)} (1+sw(v))_{\od v} \right) \sum_{\tau \in \Tub(t)} 1 .
\end{align*}
For a tree $t$ with $n$ vertices, the  sum on the right is bounded by \cref{tubings_bound}: $C_{n-1} \leq \sum_{\tau \in \Tub(t)}1 \leq (n-1)!$.
\end{example}

\FloatBarrier

\section{Tubing Feynman rules}
\label{sec tubing feynman rules}

In the present section, we show that tubings solve Dyson--Schwinger equations as stated in \cref{thm main thm}. We work in the algebraic setting introduced in \cref{sec algebraic set up}.

\subsection{Algebraic results and single-equation case}\label{sec single equation}

Let $D$ be a set of decorations and suppose for each $d \in D$ we are given a formal Laurent series
\begin{equation}
    F_d(\rho) = \sum_{n \ge 0} c_{n,d} \rho^{n-1}.
\end{equation}
We then define a family of Hochschild 1-cocycles $\{\Lambda^{(d)}\}_{d \in D}$ on $K[L]$ by
\begin{equation}\label{eq Lambda d}
    \Lambda^{(d)} p(L) = p(\partial_\rho) (e^{L\rho} - 1) F_d(\rho)|_{\rho=0}.
\end{equation}
(That these are indeed 1-cocycles follows from \cref{rem polynomial cocycles} or by a tedious but straightforward computation.)
By \cref{thm ck universal property dec} there is a unique map $\phi\colon \CK(D) \to K[L]$
satisfying $\phi B_+^{(d)} = \Lambda^{(d)}\phi$. We aim to give a formula for this map in terms of
tubings.

\begin{remark} \label{rem dse with cocycles}
    In the case of $D = \ZZ_{\ge 1}$, the operators we have defined are exactly those that appear in \cref{dse_differential_general}, which can be written as
    \begin{equation}
        G(x, L) = 1 + \sum_{k \ge 1} x^k \Lambda^{(k)}(G(x, L)^{1 + sk}).
    \end{equation}
    Having written it in this form, it is clear that applying the map $\phi$ to the solution of the
    \textit{combinatorial} DSE \cref{dse algebraic general} gives the solution to
    \cref{dse_differential_general}. Physically $\phi$ is a tree version of the \textit{renormalized
    Feynman rules} as discussed in \cref{sec mellin transform}.
\end{remark}

\begin{theorem} \label{phi formula dec}
    The map in question is given on trees by
    \begin{equation} \label{eq phi formula dec}
        \phi(t) = \sum_{\tau \in \Tub(t)} c(\tau) \sum_{i=1}^{b(\tau)} c_{b(\tau) - i, d(\rootv(t))}
        \frac{L^i}{i!}
    \end{equation}
    where $c(\tau)$ is the Mellin monomial (\cref{def Mellin monomial}) and $d(\rootv(t))$ is the decoration of the root of $t$.
\end{theorem}
The proof of \cref{phi formula dec} will be given at the \hyperlink{phi formula dec proof} {end of the current section}, it relies on some lemmas to be established in the following. 

For the moment, let $\psi$ denote the algebra map defined by the right side of \cref{eq phi formula
dec}. Let $\sigma$ be the linear term of $\psi$. Explicitly, for a tree $t$ we have
\begin{equation} \label{eq sigma dec}
    \sigma(t) =
    [L^1]\psi(t)=\sum_{\tau \in \Tub(t)} c_{b(\tau) - 1, d(\rootv(t))}c(\tau)
\end{equation}
and $\sigma$ vanishes on disconnected forests (including the empty forest) since $\psi$ is multiplicative by definition and has zero constant term on individual trees. By \cref{thm character exponential}, $\sigma$ is an infinitesimal character \cref{infinitesimal character}.

\begin{lemma}
    For any tree $t$ and $k \ge 1$,
    \[
        \sigma^{*k}(t) = \sum_{\genfrac{}{}{0pt}{2}{\tau \in \Tub(t)} { b(\tau) \ge k }} c_{b(\tau) - k, d(\rootv(t))}c(\tau).
    \]
\end{lemma}

Note that for $k > 1$, $\sigma^{*k}$ is \textit{not} an infinitesimal character and does not vanish on all disconnected forests, though it trivially must vanish on the empty forest. However, it will
be sufficient for our purposes to understand its value on trees.

\begin{proof}
    By induction on $k$. Since the $b$-statistic (\cref{def b statistic}) of every tubing $\tau$ satisfies $b(\tau) \ge 1$, the base case
    is exactly \cref{eq sigma dec}. Then by \cref{eqn ck coproduct} we have
    \[
        \sigma^{*k+1}(t) = (\sigma * \sigma^{*k})(t) = \sum_f \sigma(f) \sigma^{*k}(t \setminus
        f)
    \]
    as $f$ ranges over subforests. However, since $\sigma$ vanishes on disconnected forests only
    terms where $f$ is a tree contribute, and since $\sigma^{*k}$ vanishes on the empty forest we
    may further restrict to \textit{proper} subtrees $t'\subset t$. Thus, using the induction hypothesis we have
    \begin{align*}
        \sigma^{*k+1}(t)
        &= \sum_{t'} \sigma(t') \sigma^{*k}(t \setminus t') \\
        &= \sum_{t'} \sum_{\tau' \in \Tub(t')} \sum_{\genfrac{}{}{0pt}{2}{\tau'' \in \Tub(t \setminus t') }{ b(\tau'') \ge
        k}} c_{b(\tau') -
        1, d(\rootv(t'))} c(\tau') c_{b(\tau'') - k, d(\rootv(t))} c(\tau'').
    \end{align*}
    Now by the recursive construction of tubings (\cref{rem tau prime and tau double prime}), the
    pair $(\tau', \tau'')$ uniquely determines a tubing $\tau$ of $t$. In this tubing $\tau$,  the root  has one more tube than it has in $\tau''$, so
    \begin{equation} \label{eqn:bstat recurrence}
        b(\tau) = b(\tau'') + 1.
    \end{equation}
    The Mellin monomials (\cref{def Mellin monomial}) satisfy
    \begin{equation} \label{eqn:melmon recurrence}
        c(\tau) = c_{b(\tau') - 1, d(\rootv(t'))} c(\tau') c(\tau'')
    \end{equation}
    as we get a factor for each non-root vertex of $t'$ and $t''$ as well as for the root of $t'$ which does not contribute to $c(\tau')$.
    Hence we can rewrite the above triple sum as
    \[
        \sigma^{*k+1}(t) = \sum_{\genfrac{}{}{0pt}{2}{\tau \in \Tub(t) }{b(\tau) \ge k+1}} c_{b(\tau) - k - 1, d(\rootv(t))}
        c(\tau)
    \]
    as wanted.
\end{proof}

\begin{lemma} \label{lem convolution exponential dec}
    $\psi = \exp_*(L\sigma)$.
\end{lemma}

By \cref{thm character exponential}, this is equivalent to $\psi$ being a Hopf algebra morphism.

\begin{proof}
    Both sides are algebra morphisms, so we only need to show they agree on trees. We compute
    \begin{align*}
        \psi(t)
        &= \sum_{\tau \in \Tub(t)} c(\tau) \sum_{k=1}^{b(\tau)} c_{b(\tau)-k, d(\rootv(t))}
        \frac{L^k}{k!} \\
        &= \sum_{k \ge 1} \frac{L^k}{k!} \sum_{\genfrac{}{}{0pt}{2}{\tau \in \Tub(t) }{ b(\tau) \ge k}} c_{b(\tau) - k,
        d(\rootv(t))} c(\tau) \\
        &= \sum_{k \ge 1} \frac{L^k}{k!} \sigma^{*k}(t) \\
        &= \exp_*(L\sigma)(t). \qedhere
    \end{align*}
\end{proof}

The following elementary fact about 1-cocycles will be useful to us.

\begin{lemma} 
\label{lem cocycle convolution}
    Let $H$ be a bialgebra, $\Lambda$ be a Hochschild 1-cocycle on $H$, and $a, b\colon H \to A$ for some
    commutative algebra $A$. Then
    \[
        (a * b)\Lambda = b(1) a \Lambda + a * b\Lambda.
    \]
\end{lemma}

\begin{proof}
    A simple computation:
    \begin{align*}
        (a * b)\Lambda
        &= m_A (a \otimes b) \Delta_H \Lambda \\
        &= m_A (a \otimes b) (\Lambda \otimes 1 + (\mathrm{id} \otimes \Lambda)\Delta_H) \\
        &= m_A (a\Lambda \otimes b(1) + (a \otimes b\Lambda)\Delta_H) \\
        &= b(1) a \Lambda + a * b\Lambda. \qedhere
    \end{align*}
\end{proof}

In particular, if $b(1) = 0$ (such as when $b$ is an infinitesimal character), \cref{lem cocycle convolution} simply becomes
$(a * b)\Lambda = a * b\Lambda$.

\begin{lemma} \label{lem sigma bplus dec}
    For each $d \in D$,
    \[
        \sigma B_+^{(d)} = \sum_{i \ge 0} c_{i,d} \sigma^{*i}.
    \]
\end{lemma}

\begin{proof}
    For $i \ge 1$ and $d \in D$ define an infinitesimal character $\sigma_{i, d}$ by
    \[
        \sigma_{i, d}(t) = \sum_{\genfrac{}{}{0pt}{2}{\tau \in \Tub(t) }{ b(\tau) = i}} c(\tau)
    \]
    when $t$ is a tree with root decorated by $d$ and otherwise zero. Clearly, we then have
    \[
        \sigma = \sum_{i \ge 1} \sum_{d \in D} c_{i-1, d} \sigma_{i, d}.
    \]
    By construction we have $\sigma_{i, d'} B_+^{(d)} = 0$ for $d \ne d'$. Thus the result follows
    if we can show that $\sigma_{i, d} B_+^{(d)} = \sigma^{*i-1}$. We do this by induction on $i$.
    For the base case, note that the only tubing $\tau$ satisfying $b(\tau) = 1$ is the unique
    tubing of the one-vertex tree, which also has $c(\tau) = 1$. Thus $\sigma_{1,d}$ sends the
    one-vertex tree with decoration $d$ to 1 and all other trees to 0, so $\sigma_{1,d} B_+^{(d)}$
    sends the empty forest to 1 and all other forests to 0, i.e. $\sigma_{1,d}B_+^{(d)} =
    \varepsilon = \sigma^{*0}$.

    Now consider $\sigma_{i+1, d}$. This vanishes on one-vertex trees. For $t$ a tree with more than
    one vertex and root decorated by $d$ we have
    \begin{align*}
        \sigma_{i+1, d}(t)
        &= \sum_{\genfrac{}{}{0pt}{2}{\tau \in \Tub(t) }{ b(\tau) = i+1}} c(\tau) \\
        &= \sum_{t'} \left(\sum_{\tau' \in \Tub(t')} c_{b(\tau') - 1,d} c(\tau')\right)
        \left(\sum_{\genfrac{}{}{0pt}{2}{\tau'' \in \Tub(t \setminus t') }{ b(\tau'') = i}} c(\tau'')\right) & \text{by \cref{eqn:bstat recurrence} and \cref{eqn:melmon recurrence}}\\
        &= \sum_{t'} \sigma(t') \sigma_{i, d}(t \setminus t') \\
        &= (\sigma * \sigma_{i, d})(t).
    \end{align*}
    (where as usual $t'$ ranges over rooted subtrees of $t$ in both relevant sum) and thus while $\sigma_{i+1, d} \ne \sigma * \sigma_{i, d}$ the two do agree on the image of
    $B_+^{(d)}$.

    Thus
    \begin{align*}
        \sigma_{i+1, d} B_+^{(d)}
        &= (\sigma * \sigma_{i, d}) B_+^{(d)} \\
        &= \sigma * \sigma_{i, d} B_+^{(d)} && \text{by \cref{lem cocycle convolution}} \\
        &= \sigma * \sigma^{*i-1} && \text{by the induction hypothesis} \\
        &= \sigma^{*i}
    \end{align*}
    as wanted.
\end{proof}

Finally, we are ready to prove our formula.

\begin{proof}[Proof of \cref{phi formula dec}] \hypertarget{phi formula dec proof}
    To show $\phi = \psi$, by uniqueness we need only show that $\psi$ satisfies the required
    formula $\psi B_+^{(d)} = \Lambda^{(d)} \psi$. Note that both of these have zero constant term, so it is sufficient to show that they agree after differentiating with respect to $L$. By
    \cref{rem polynomial cocycles}, we can rewrite $\Lambda^{(d)}$ in the integral form
    \[
        \Lambda^{(d)} f(L) = \int_0^L \mathrm{d}u\, \sum_{n \ge 0} c_{i, d} \partial_u^i f(u)
    \]
    and thus
    \[
        \partial_L \Lambda^{(d)} = \sum_{i \ge 0} c_{i, d} \partial_L^i.
    \]
    Now using our lemmas we can compute
    \begin{align*}
        \partial_L \psi B_+^{(d)}
        &= (\psi * \sigma) B_+^{(d)} && \text{by \cref{lem convolution exponential dec}} \\
        &= \psi * \sigma B_+^{(d)} && \text{by \cref{lem cocycle convolution}} \\
        &= \sum_{i \ge 0} c_{i, d} \psi * \sigma^{*i} && \text{by \cref{lem sigma bplus dec}} \\
        &= \sum_{i \ge 0} c_{i, d} \partial_L^i \psi && \text{by \cref{lem convolution exponential
        dec}} \\
        &= \partial_L \Lambda^{(d)} \psi.
    \end{align*}
    The result follows.
\end{proof}

We are now able to prove our main result for the single-equation case.

\begin{proof}[Proof of \cref{thm main thm}]
    As suggested in \cref{rem dse with cocycles}, we apply $\phi$ to the series from \cref{thm
    multi primitive comb dse} that solves the combinatorial DSE. The result follows immediately.
\end{proof}

\subsection{Systems}\label{sec systems}

\Cref{phi formula dec} can also be applied to \textit{systems} of Dyson--Schwinger equations \cite{foissy_classification_2010,kissler_systems_2019}. Let
$A$ be a finite set which indexes the equations in our system. Our set of decorations will be $A
\times \ZZ_{\ge 1}$. Given parameters $\mathbf s = (s_a)_{a \in A}$ we have the combinatorial system
\begin{equation} \label{eq combinatorial system}
    \begin{aligned}
        T_a(x) &= 1 + \sum_{k \ge 1} x^k B_+^{(a, k)}\big(T_a(x)Q(x)^k \big) \\
        Q(x) &= \prod_{a \in A} T_a(x)^{s_a}
    \end{aligned}
\end{equation}
in the algebra $\CK(A \times \ZZ_{\ge 1})$. Our first goal will be to prove an analogue of \cref{thm multi primitive comb dse} for systems.

For a vertex $v$ of an $(A \times \ZZ_{\ge
1})$-decorated tree $t$, write $\alpha(v)$ for the $A$ component of its decoration (which we will
call the \textit{type}) and $w(v)$ for the $\ZZ_{\ge 1}$ (which we continue to call the
\textit{weight}). Write $\od[a]v$ for the number of children of type $a$. Then define
\begin{equation}
    \xi(v, \mathbf s) = (1 + s_{\alpha(v)} w(v))_{\od[\alpha(v)]{v}}
    \prod_{a \ne \alpha(v)} (s_a w(v))_{\od[a]{v}}.
\end{equation}
Note that in the case of a single equation (only one type) this simplifies to $\xi(v, s) = (1 +
sw(v))_{\od v}$, the same factor contributed by the vertex to the solution given in \cref{thm multi primitive comb dse}.

\begin{theorem} \label{thm cdse system sols}
    The solution to the combinatorial system \cref{eq combinatorial system} is given by
    \begin{equation} \label{eq cdse system formula}
        T_a(x) = 1 + \sum_{\alpha(\rootv(t)) = a} \left(\prod_{v \in V(t)} \xi(v, \mathbf s)
        \right) \frac{t x^{w(t)}}{|\operatorname{Aut}(t)|}.
    \end{equation}
\end{theorem}

\begin{proof}
    For each $a \in A$ let $\mathcal T_a$ denote the set of $(A \times \ZZ_{\ge 1})$-decorated trees
    $t$ such that $\alpha(\rootv(t)) = a$, and let $\mathcal F_a$ denote the set of forests such
    that all components lie in $\mathcal T_a$. For a forest $f$, write $\kappa(f)$ for the
    number of components and
    \[
        \xi(f, \mathbf s) = \prod_{v \in V(f)} \xi(v, \mathbf s).
    \]

    Let $T_a(x)$ be given by \cref{eq cdse system formula} and let $T_a^+(x) = T_a(x) - 1$. In terms
    of the notation just defined we have
    \[
        T_a^+(x) = \sum_{t \in \mathcal T_a} \xi(t, \mathbf s) \frac{t x^{w(t)}}{|\Aut(t)|}.
    \]
    Now a forest $f$ with $r$ components can be represented as a product $t_1 \cdots t_r$ of trees.
    These trees can be permuted arbitrarily, so the total number of distinct ways to represent it
    this way is $r!$ divided by the number of permutations that rearrange isomorphic trees. This
    quantity is equivalently the number of orbits of $\Aut(f)$ on the set of orderings of the components. By elementary
    group theory (or the compositional formula for exponential generating functions) we get
    \[
        \frac{T_a^+(x)^r}{r!} = \sum_{\genfrac{}{}{0pt}{2}{f \in \mathcal F_a }{ ~ \kappa(f) = r}} \xi(f, \mathbf s)
        \frac{fx^{w(f)}}{|\Aut(f)|}.
    \]
    (Here we are using the fact that $\xi$ is multiplicative and $w$ is additive over connected components.) It follows that for any scalar $p$,
    \begin{align*}
        T_a(x)^p
        &= \sum_{r \ge 0} \binom{p}{r} T_a^+(x)^r \\
        &= \sum_{r \ge 0} (p)_r \sum_{\genfrac{}{}{0pt}{2}{f \in \mathcal F_a }{  ~ \kappa(f) = r}} \xi(f, \mathbf s)
        \frac{fx^{w(f)}}{|\Aut(f)|} \\
        &= \sum_{f \in \mathcal F_a} (p)_{\kappa(f)} \xi(f, \mathbf s) \frac{f x^{w(f)}}{|\Aut(f)|}.
    \end{align*}

    For a general $(A \times \ZZ_{\ge 1})$-decorated forest, write $\kappa_a(f)$ for the number of
    components that are in $\mathcal T_a$. Then for any exponents $(p_a)_{a \in A}$ we can take the
    product to get
    \[
        \prod_{a \in A} T_a(x)^{p_a} = \sum_f \left(\prod_{a \in A} (p_a)_{\kappa_a(f)} \right)
        \xi(f, \mathbf s) \frac{f x^{w(f)}}{|\Aut(f)|}.
    \]
    Here the appearance of $|\Aut(f)|$ can be explained by the same group-theoretic reasoning as before, but it also follows from the observation that since automorphisms must preserve the decoration of the root, $\Aut(f)$ will factor as a direct product over $a \in A$ of the autormorphism groups of subforests in $\mathcal F_a$.

    Now any $t \in \mathcal T_a$ can be uniquely written as $t = B_+^{(a, k)}f$ for some forest $f$
    and some $k \in \ZZ_{\ge 1}$. Moreover we have $w(t) = w(f) + k$, $\Aut(t) \cong \Aut(f)$, and
    for any type $a'$ we have
    $\od[a']{\rootv(t)} = \kappa_{a'}(f)$. Hence
    \[
        \xi(t, \mathbf s) = (1 + s_ak)_{\kappa_a(f)} \left(\prod_{a' \ne a}
        (s_{a'}k)_{\kappa_{a'}(f)}\right) \xi(f, \mathbf s).
    \]
    Putting all of this together,
    \begin{align*}
        T_a(x)
        &= 1 + \sum_{t \in \mathcal T_a} \xi(t, \mathbf s) \frac{t x^{w(t)}}{|\Aut(t)|} \\
        &= 1 + \sum_{k \ge 1} \sum_f (1 + s_ak)_{\kappa_a(f)} \left(\prod_{a' \ne a}
        (s_{a'}k)_{\kappa_{a'}(f)}\right) \xi(f, \mathbf s) \frac{(B_+^{(a, k)}f) x^{w(f) +
        k}}{|\Aut(f)|} \\
        &= 1 + \sum_{k \ge 1} x^k B_+^{(a, k)} \left(T_a(x)^{1 + s_ak} \prod_{a' \ne a}
        T_{a'}(x)^{s_{a'}k}\right) \\
        &= 1 + \sum_{k \ge 1} x^k B_+^{(a, k)}(T_a(x) Q(x)^k)
    \end{align*}
    and hence these series solve \cref{eq combinatorial system} as desired.
\end{proof}

The analytic system corresponding to \cref{eq combinatorial system} is
\begin{equation} \label{eq analytic system}
    \begin{aligned}
        G_a(x, L) &= 1 + \sum_{k \ge 1} x^k G(x, \partial_\rho) Q(x, \partial_\rho)^k (e^{L\rho} -
        1)F_{a,k}(\rho)\bigg|_{\rho = 0} \\
        Q(x, L) &= \prod_{a \in A} G_a(x, L)^{s_a}
    \end{aligned}
\end{equation}
where
\begin{equation}
    F_{a, k}(\rho) = \sum_{n \ge 0} c_{n,a,k}\rho^{n-1}.
\end{equation}
(Note that for aesthetic reasons we write $c_{n,a,k}$ with three separate subscripts where to perfectly match the notation of the previous subsection we would have $c_{n,(a,k)}$.)

Applying \cref{phi formula dec} to the combinatorial solution from \cref{thm cdse system sols} immediately gives the following.

\begin{theorem}
    The solution to the analytic system \cref{eq analytic system} is given by
    \begin{equation}
        G_a(x, L) = 1 + \sum_{\alpha(\rootv(t)) = a} \left(\prod_{v \in V(t)} \xi(v, \mathbf
        s)\right) \left(\sum_{\tau \in \Tub(t)} c(\tau) \sum_{i=1}^{b(\tau)} c_{b(\tau) - i,
        a, w(\rootv(t))} \frac{L^i}{i!} \right)
        \frac{x^{w(t)}}{|\operatorname{Aut}(t)|}.
    \end{equation}
\end{theorem}

\section{Chord diagrams}
\label{sec chord}

One of us with other coauthors has used expansions indexed by rooted connected chord diagrams to solve the Dyson--Schwinger equation \cref{dse_differential_general} first in the case $k=1$ and $s=-2$ \cite{marie_chord_2013}, then in the case where $s$ is any negative integer and no restrictions on $k$ \cite{hihn_generalized_2019}.  Physical and combinatorial properties and consequences of these expansions were further studied in \cite{courtiel_terminal_2017, courtiel_connected_2019, courtiel_nexttok_2020}.  

Chord diagrams have long been studied across mathematics and the sciences \cite{touchard_probleme_1952}, including Vassiliev invariants in knot theory \cite{zagier_vassiliev_2001, bollobas_linearized_2000}, graph sampling \cite{acan_enumerativeprobabilistic_2013}, analysis of computer structures \cite{flajolet_sequence_1980}, and bioinformatics \cite{andersen_topological_2013, hofacker_combinatorics_1998, bafna_genome_1996}.  From an enumerative perspective rooted chord diagrams have been studied in their own right see \cite{stein_class_1978, nijenhuis_enumeration_1979, nabergall_combinatorics_2023, nabergall_enumerative_2022} and many others, and are also as a special case of what are known as arc diagrams \cite{burrill_generating_2014, pilaud_hopf_2019}.

These chord diagram expansions were nice in that they were the first way to give combinatorial solutions to the Dyson--Schwinger equation where each indexing combinatorial object contributed simple monomials similar to the Mellin monomials of \cref{def Mellin monomial} and a small polynomial in $L$.  This combinatorial control on the solution to the Dyson--Schwinger equation was particularly suitable to studying the leading log, next to leading log, etc. contributions \cite{courtiel_terminal_2017, courtiel_nexttok_2020}.  The chord diagram expansion also singled out certain parameters on rooted connected chord diagrams, the \emph{terminal chords}, see below, which had not had substantial study in the past\footnote{An equivalent concept has appeared in the combinatorics of associated Hermite polynomials \cite{drake_combinatorics_2009} and certain infinite limits of Gaussian $\beta$-ensembles \cite{benaych-georges_matrix_2022}.} but are interesting in their distribution \cite{courtiel_terminal_2017}, and correspond to reasonable parameters on other combinatorial objects such as vertices in bridgeless combinatorial maps \cite{courtiel_connected_2019}, and lead to interesting enumerative questions and results \cite{nabergall_combinatorics_2023,nabergall_enumerative_2022}.  These expansions rejuvenated the pure combinatorial study of chord diagrams.

However, the chord diagram expansion solutions to Dyson--Schwinger equations also had a significant downside: namely they were proved by grinding certain matching recurrences so they did not give direct insight into questions such as which chord diagrams are contributed by which Feynman graph, how to generalize to other values of $s$ or to systems or beyond, or why really it was chord diagrams that appeared.  The tubing expansions of the present paper resolve these downsides while maintaining the upsides of combinatorially indexed solutions to Dyson--Schwinger equations with each object having a simple contribution to the sum, and being enumeratively interesting combinatorial objects.  To see these benefits clearly, we will lay out the definitions for the chord diagram expansions, and prove that the tubing expansions and chord diagram expansions agree on their common domain.

\subsection{Chord diagram set up} \label{sec chord diagram set up}

\begin{definition}\label{rooted chord diagram}\label{root chord}
A \emph{rooted chord diagram} of \emph{size} $n$ is a set of ordered pairs of the form $(a_i,b_i)$, called \emph{chords}, such that $\{a_1,b_1, a_2, b_2, \dots, a_n,b_n\} = \{1,2, \dots, 2n\}$ and $a_i < b_i$ for each $1 \leq i \leq n$. 
The chord with $a_1=1$ is the \emph{root chord}.
\end{definition}

If we consider the set $\{1,2, \dots, 2n\}$ up to cyclic permutation, then the equivalence classes are unrooted chord diagrams.
In the graphical sense, an unrooted chord diagram is a 1-regular graph on a cyclically ordered vertex set. Here, chords are the edges of the graph. A rooted chord diagram is a 1-regular graph on a linearly ordered vertex set and the ordered pair $(a_i,b_i)$ is a chord where $a_i$ is its first vertex, or \emph{source}, and $b_i$ is its second vertex, or \emph{sink}. The root chord is the first chord. We will only consider rooted chord diagrams, so from now on, \emph{chord diagram} indicates rooted chord diagrams.

 \begin{definition}\label{decomposable}\label{indecomposable}
 A chord diagram is \emph{decomposable} if its vertices can be partitioned into two sets, $S,T$, such that $s < t$ for all $s \in S$ and $t \in T$, and that there are no chords $(a_i,b_i)$ with $a_i \in S$ and $b_i \in T$.  A chord diagram is \emph{indecomposable} if it is not decomposable. 
 \end{definition}

Intuitively, a chord diagram is decomposable if it is a concatenation of smaller nonempty chord diagrams.

 \begin{definition}\label{non-crossing}
We say a pair of chords $(a_i,b_i)$ and $(a_j,b_j)$ \emph{cross} if $a_i<a_j<b_i<b_j$ or $a_j<a_i<b_j<b_i$.  
A chord diagram is \emph{non-crossing} if it contains no pair of crossing chords.
 \end{definition}

This definition has the suggested meaning in the graphical representation as arcs above the linearly ordered endpoints:  two chords cross precisely when the arcs cross.

 \begin{definition}\label{intersection graph}\label{connected chord diagram}
The \emph{intersection graph} of a chord diagram is a graph such that for each chord $c_i$ in a chord diagram, there is a vertex $v_i$ in the intersection graph. Two vertices $v_1$ and $v_2$ are adjacent if the chords $c_1$ and $c_2$ cross in the chord diagram. A chord diagram is \emph{connected} if its intersection graph is connected. A \emph{connected component} of a chord diagram is a subset of chords $S$ such that the chord diagram $S$ is connected.
\end{definition}

\begin{figure}[htb]
    \centering
    \begin{tikzpicture}
        \node[vertex, gray, label = below:$a_1$](a1) at (0,0){};
        \node[vertex, gray, label = below:$a_2$](a2) at ($(a1) + (.8,0)$){};
        \node[vertex, gray,label = below:$a_3$](a3) at ($(a2) + (.8,0)$){};
        \node[vertex, gray,label = below:$b_3$](a4) at ($(a3) + (.8,0)$){};
        \node[vertex, gray,label = below:$b_1$](a5) at ($(a4) + (.8,0)$){};
        \node[vertex, gray,label = below:$a_4$](a6) at ($(a5) + (.8,0)$){};
        \node[vertex, gray,label = below:$a_5$](a7) at ($(a6) + (.8,0)$){};
        \node[vertex, gray,label = below:$b_2$](a8) at ($(a7) + (.8,0)$){};
        \node[vertex, gray,label = below:$b_4$](a9) at ($(a8) + (.8,0)$){};
        \node[vertex, gray,label = below:$b_5$](a10) at ($(a9) + (.8,0)$){};
        
        \draw[chord] (a1) to node[pos=.3,above]{$c_1$} (a5);
        \draw[chord] (a2) to node[pos=.5, above]{$c_2$} (a8);
        \draw[chord] (a3) to node[pos=.5, above]{$c_3$} (a4);
        \draw[chord] (a6) to node[pos=.6, above]{$c_4$} (a10);
        \draw[chord] (a7) to node[pos=.9, above]{$c_5$} (a9);

        \node[vertex, label = below:$v_1$] (v1) at (9,.5){};
        \node[vertex, label = below:$v_2$] (v2) at (10.5,1){};
        \node[vertex, label = below:$v_3$] (v3) at (10,0){};
        \node[vertex, label = below:$v_4$] (v4) at (12,1){};
        \node[vertex, label = below:$v_5$] (v5) at (12,0){};
        \draw[edge] (v1) -- (v2) --(v4);
        \draw[edge] (v2) -- (v5);
        
    \end{tikzpicture}
    \caption{An indecomposable (\cref{indecomposable}), but still non-connected (\cref{connected chord diagram}), chord diagram, and its intersection graph (\cref{intersection graph}).}
    \label{figure chord example}
\end{figure}
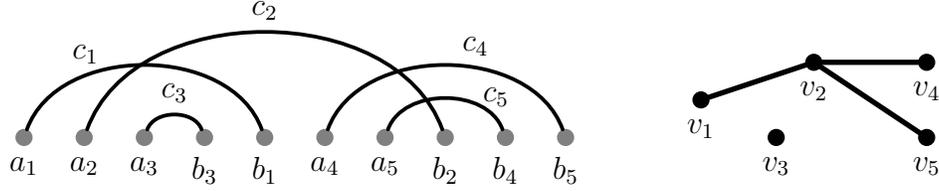

Note that connectivity in chord diagrams is not defined in the graph theoretical sense on the chord diagram as a 1-regular graph, but rather on the intersection graph, or on the arcs of the chord diagram, see \cref{figure chord example}. 

\begin{example}\label{eg Yukawa chord}
In the Yukawa graphs of \cref{eg yukawa 1}, the fermion line gives a linear order to the vertices, so labelling these vertices $\{1, \ldots, 2n\}$, and hence the meson lines themselves define a chord diagram according to \cref{rooted chord diagram}. Drawing the meson lines as arcs as in \cref{fig yukawa}, we obtain the usual linear graphical representation. Since these Yukawa graphs are formed by inserting the one loop bubble into itself, there is always a chord $(1,2n)$ and the chords are non-crossing.  Consequently, a Yukawa graph is an indecomposable (\cref{indecomposable}) non-crossing (\cref{non-crossing}) chord diagram, and all indecomposable non-crossing chord diagrams correspond to Yukawa graphs of this type.  
\end{example}

While the earlier definitions were standard, the next definitions are special to the needs of the chord diagram expansion solutions of Dyson--Schwinger equations.

\begin{definition}\label{intersection order}
For a connected chord diagram $C$, the {\em intersection order} on $C$ is defined recursively as follows: starting with 1, label the root chord with the next available label, then remove the root chord and label the resulting sequence of nested connected components recursively in the order of their first points.  
\end{definition}

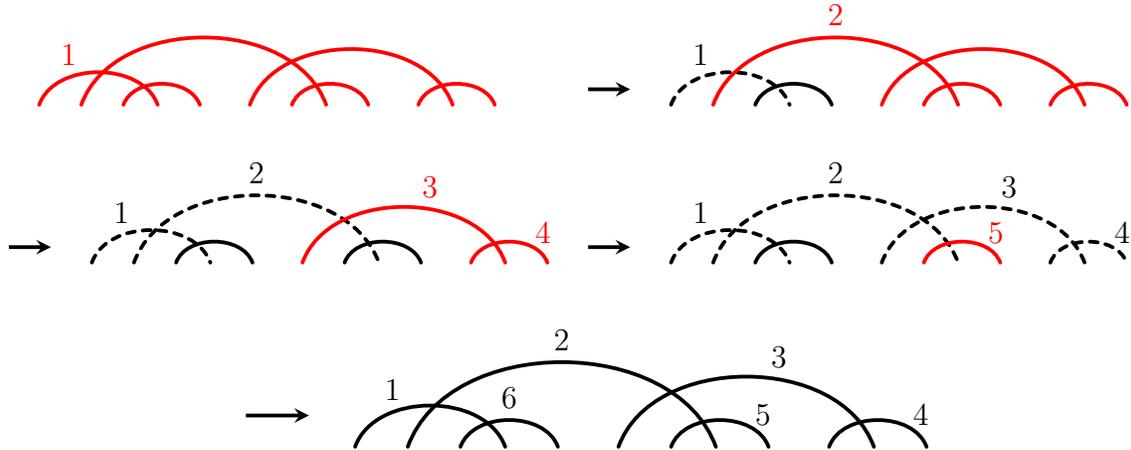
\begin{figure}[htb]
    \centering
    \begin{tikzpicture}[scale=.7]
        \node (a1) at (0,0){};
        \node (a2) at ($(a1) + (.8,0)$){};
        \node (a3) at ($(a2) + (.8,0)$){};
        \node (a4) at ($(a3) + (.8,0)$){};
        \node (a5) at ($(a4) + (.8,0)$){};
        \node (a6) at ($(a5) + (.8,0)$){};
        \node (a7) at ($(a6) + (.8,0)$){};
        \node (a8) at ($(a7) + (.8,0)$){};
        \node (a9) at ($(a8) + (.8,0)$){};
        \node (a10) at ($(a9) + (.8,0)$){};
        \node (a11) at ($(a10) + (.8,0)$){};
        \node (a12) at ($(a11) + (.8,0)$){};
        
        \draw[chord,red ] (a1) to node[pos=.3,above]{$1$} (a4);
        \draw[chord,red] (a2) to  (a8);
        \draw[chord,red] (a3) to  (a5);
        \draw[chord,red] (a6) to   (a11);
        \draw[chord,red] (a7) to   (a9);
        \draw[chord,red] (a10) to  (a12);
        
        \draw[line width=.5mm, -stealth] (10.5,.5)-- +(.8,0);
        
         \node (a1) at (12,0){};
        \node (a2) at ($(a1) + (.8,0)$){};
        \node (a3) at ($(a2) + (.8,0)$){};
        \node (a4) at ($(a3) + (.8,0)$){};
        \node (a5) at ($(a4) + (.8,0)$){};
        \node (a6) at ($(a5) + (.8,0)$){};
        \node (a7) at ($(a6) + (.8,0)$){};
        \node (a8) at ($(a7) + (.8,0)$){};
        \node (a9) at ($(a8) + (.8,0)$){};
        \node (a10) at ($(a9) + (.8,0)$){};
        \node (a11) at ($(a10) + (.8,0)$){};
        \node (a12) at ($(a11) + (.8,0)$){};
        
        \draw[chord,dashed] (a1) to node[pos=.3,above]{$1$} (a4);
        \draw[chord,red] (a2) to node[pos=.5, above]{$2$} (a8);
        \draw[chord] (a3) to  (a5);
        \draw[chord,red] (a6) to   (a11);
        \draw[chord,red] (a7) to  (a9);
        \draw[chord,red] (a10) to  (a12);

        \draw[line width=.5mm, -stealth] (-.5,-2.5)-- +(.8,0);
         \node (a1) at (1,-3){};
        \node (a2) at ($(a1) + (.8,0)$){};
        \node (a3) at ($(a2) + (.8,0)$){};
        \node (a4) at ($(a3) + (.8,0)$){};
        \node (a5) at ($(a4) + (.8,0)$){};
        \node (a6) at ($(a5) + (.8,0)$){};
        \node (a7) at ($(a6) + (.8,0)$){};
        \node (a8) at ($(a7) + (.8,0)$){};
        \node (a9) at ($(a8) + (.8,0)$){};
        \node (a10) at ($(a9) + (.8,0)$){};
        \node (a11) at ($(a10) + (.8,0)$){};
        \node (a12) at ($(a11) + (.8,0)$){};
        
         \draw[chord,dashed ] (a1) to node[pos=.3,above]{$1$} (a4);
        \draw[chord,dashed] (a2) to node[pos=.5, above]{$2$} (a8);
        \draw[chord] (a3) to   (a5);
        \draw[chord,red] (a6) to node[pos=.6, above]{$3$} (a11);
        \draw[chord] (a7) to  (a9);
        \draw[chord,red] (a10) to node[pos=.9, above]{$4$} (a12);

        \draw[line width=.5mm, -stealth] (10.5,-2.5)-- +(.8,0);
         \node (a1) at (12,-3){};
        \node (a2) at ($(a1) + (.8,0)$){};
        \node (a3) at ($(a2) + (.8,0)$){};
        \node (a4) at ($(a3) + (.8,0)$){};
        \node (a5) at ($(a4) + (.8,0)$){};
        \node (a6) at ($(a5) + (.8,0)$){};
        \node (a7) at ($(a6) + (.8,0)$){};
        \node (a8) at ($(a7) + (.8,0)$){};
        \node (a9) at ($(a8) + (.8,0)$){};
        \node (a10) at ($(a9) + (.8,0)$){};
        \node (a11) at ($(a10) + (.8,0)$){};
        \node (a12) at ($(a11) + (.8,0)$){};
        
       \draw[chord,dashed ] (a1) to node[pos=.3,above]{$1$} (a4);
       \draw[chord,dashed] (a2) to node[pos=.5, above]{$2$} (a8);
       \draw[chord] (a3) to (a5);
       \draw[chord,dashed] (a6) to node[pos=.6, above]{$3$} (a11);
       \draw[chord,red] (a7) to node[pos=.9, above]{$5$} (a9);
       \draw[chord,dashed] (a10) to node[pos=.9, above]{$4$} (a12);

       \draw[line width=.5mm, -stealth] (4,-5.7)-- +(1.2,0);
       \node (a1) at (6,-6.5){};
       \node (a2) at ($(a1) + (1,0)$){};
       \node (a3) at ($(a2) + (1,0)$){};
       \node (a4) at ($(a3) + (1,0)$){};
       \node (a5) at ($(a4) + (1,0)$){};
       \node (a6) at ($(a5) + (1,0)$){};
       \node (a7) at ($(a6) + (1,0)$){};
       \node (a8) at ($(a7) + (1,0)$){};
       \node (a9) at ($(a8) + (1,0)$){};
       \node (a10) at ($(a9) + (1,0)$){};
       \node (a11) at ($(a10) + (1,0)$){};
       \node (a12) at ($(a11) + (1,0)$){};
       
       \draw[chord] (a1) to node[pos=.3,above]{$1$} (a4);
       \draw[chord] (a2) to node[pos=.5, above]{$2$} (a8);
       \draw[chord] (a3) to node[pos=.5, above]{$6$} (a5);
       \draw[chord] (a6) to node[pos=.6, above]{$3$} (a11);
       \draw[chord] (a7) to node[pos=.9, above]{$5$} (a9);
       \draw[chord] (a10) to node[pos=.9, above]{$4$} (a12);

    \end{tikzpicture}
    \caption{Finding the intersection order (\cref{intersection order}) for a connected chord diagram. At every stage, the connected component to be labeled next is marked in red. The chords that have been labeled before, and are no longer relevant for determining connectivity of the remainder, are dashed. Solid black chords will be labeled after all red cords are labeled.}
    \label{figure chord intersection order}
\end{figure}

As an example let us consider how to label the chords of the diagram in \cref{figure chord intersection order} in intersection order.  The root chord is labelled 1.  Removing the root chord leaves two connected components.  The first or outer component, call it $C_1$ (marked red in the top right drawing of \cref{figure chord intersection order}), has four chords and the second or inner one, call it $C_2$ has a single chord. Following the definition of the intersection order, we will next label $C_1$ in the intersection order.  The root chord of $C_1$, then, gets label 2.  Next we recursively look at the components resulting from removing the root of $C_1$.  Again two components result.  The first of these (marked red in the third drawing of \cref{figure chord intersection order}) gets labelled next, getting labels 3 and 4 for its two chords according to the intersection order.  Now we are ready to label the second component that resulted from removing the root of $C_1$ (marked red in the fourth drawing).  This component has only one chord which gets label 5.  Finally, we are done labelling $C_1$ and so we move to labelling $C_2$, giving its solitary chord the next available label of 6.  Further examples can be found in \cite{marie_chord_2013, courtiel_terminal_2017}.

\begin{definition}\label{terminal chord}\label{1-terminal}
A chord $c = (a, b)$ in a chord diagram $C$ is {\em terminal} if there is no chord $c' = (a', b')$ such that $a < a' < b < b'$, that is, there is no chord $c'$ crossing $c$ on the right. A diagram $C$ is {\em 1-terminal} if it has exactly one terminal chord. 
\end{definition}

 The diagram in \cref{figure chord intersection order} has three terminal cords, labelled 4,5, and 6 in intersection order.  Note that the intersection order does not agree with the left to right order by sources or by sinks.
Observe that the chord with rightmost endpoint in each connected component of a chord diagram is necessarily terminal. In particular,  1-terminal chord diagrams are connected. Furthermore, we require the following key fact characterizing 1-terminality.

\begin{lemma}[{\cite[Corollary 4.10]{nabergall_combinatorics_2023}}]
\label{1-terminal characterization}
Let $C$ be a connected chord diagram. Then $C$ is 1-terminal if and only if $C - c_{1}$ is 1-terminal, where $c_{1}$ is the root chord.
\end{lemma}

Our goal is to establish a bijection between chord diagrams and tubings.
Tubings are constructed recursively, so we will work with a corresponding recursive construction for chord diagrams. The constructions are aligned by via insertion places for each type of object. 

\begin{definition}\label{chord diagram insertion place}
Let $C$ be a chord diagram of size $n$. Then the \emph{chord diagram insertion places (CDIPs) } for $C$ are $y_i=\{i,i+1\}$ for $i \in \{1,2, \dots, 2n-1\}$.  We order the chord diagram insertion places by their index.
\end{definition}

The chord diagram insertion places can be visualized as the space between ends of chords in a chord diagram, see \cref{figure CDIP}. 

\begin{figure}[htb]
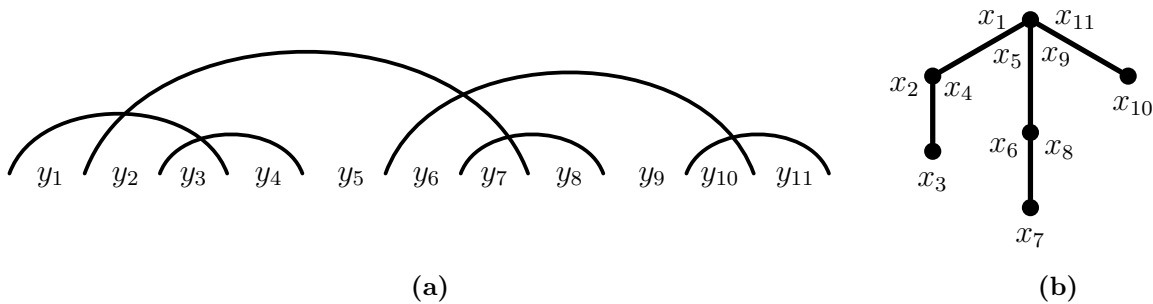

    \centering
    \begin{subfigure}[b]{.7\linewidth}
        \input{figure_CDIP.tikzpicture}
        \vspace{.8cm}
        \caption{}
        \label{figure CDIP}
    \end{subfigure}
    \begin{subfigure}[b]{.29 \linewidth}
        \input{figure_RTIP.tikzpicture}
        \caption{}
        \label{figure RTIP}
    \end{subfigure}
    
    \caption{ {\bfseries(a)} The chord diagram of \cref{figure chord intersection order} with CDIPs (\cref{chord diagram insertion place}) $y_i$  indicated.  {\bfseries (b)} A rooted tree with RTIPs (\cref{rooted tree insertion place}) $x_i$  indicated. }
    
\end{figure}

\begin{definition}\label{rooted tree insertion place}
Given a plane rooted tree $t$ a \emph{rooted tree insertion place (RTIP)} is a place where a new leaf could be added. That is, at each vertex there is an insertion place before any of the children of the vertex, between any two consecutive children, and after the last child, and all insertion places are of this form.

We place an order on the RTIPs as follows.  Let $r$ be the root of $t$ and let $t_1,t_2, \dots, t_k$ be the list of subtrees (whose roots are the children of $r$ in that order). Define the order on the insertion places, $x_i$, of $t$ in this order:
\begin{itemize}
\item the insertion place before $t_1$.

\item all insertion places of $t_1$ in this same order, applied recursively on $t_1$.
\item the insertion place between $t_1$ and $t_2$.
\item all insertion places of $t_2$ in this same order recursively on $t_2$.

$\vdots$

\item the insertion place after $t_k$

\end{itemize}
\end{definition}

To think of this intuitively, this order on insertion places can be visualized by drawing the tree and traversing around it counterclockwise starting to the left of the root, see \cref{figure RTIP}.

\subsection{Bijection between tubings and chord diagrams}\label{sec bijection tubings chord diagrams}

The next order of business is defining a bijection between tubings and connected chord diagrams.  This will be the bijection that gives us the connection between tubing expansions of Dyson--Schwinger equations and chord diagram expansions of Dyson--Schwinger equations.  Some of the special classes of tubings defined previously will correspond to interesting classes of chord diagrams under this bijection.

We first introduce a map from RTIPs (\cref{rooted tree insertion place}) to CDIPs (\cref{chord diagram insertion place}).
\begin{definition}\label{map RTIP to CDIP}
Given a plane rooted tree $t$ with $n$ vertices and a chord diagram $C$ with $n$ chords, let $x_i$ be the $i$th RTIP of $t$ and let $y_i$ be the $i$th CDIP of $C$ under the orders given in \cref{rooted tree insertion place} and \cref{chord diagram insertion place}.  Then define the map $f$ via $f(x_i) = y_i$.
\end{definition}

\begin{lemma}\label{bijection RTIP to CDIP}
The number of RTIPs on a plane rooted tree $t$ with $n$ vertices is the same as the number of CDIPs on a chord diagram $C$ with $n$ chords and so $f$ is a bijection between the RTIPs of $t$ and the CDIPs of $C$.
\end{lemma}
\begin{proof}
Let $P_n$ be the number of RTIPs on a plane rooted tree with $n$ vertices.
Let $Q_n$ be the number of CDIPs on a chord diagram with $n$ chords.

We first claim that $P_n=2n-1$ by induction. Base case $n=1$. There is only one insertion place on a single vertex. So $P_1=1$.
Now, suppose for a plane rooted tree with $n$ vertices that $P_n=2n-1$.
Suppose a new vertex, $v_j$, is added at an insertion place, $x$, on an existing vertex $v_i$. Note that every plane rooted tree on $n+1$ vertices can be constructed by adding a new vertex in this way to a plane rooted tree on $n$ vertices. Then the new vertex $v_j$ has one insertion place; $v_i$ loses the $x$ insertion place (replaced by the edge $v_0v_1$) and thereby gains two new insertion places on either side of the edge $v_0v_1$. Therefore, $P_{n+1} = P_n -1+2+1=P_n+2$ for $n \geq 1$. Thus $P_{n+1}=2n-1+2=2(n+1)-1$.

We next claim that $Q_n=2n-1$. This follows directly from the definition of chord diagram insertion places. There are exactly $2n-1$ pairs $\{i,i+1\}$ for $i \in \{1,2, \dots , 2n-1\}$. 

Thus, $P_n=Q_n$ and hence $f$ is a bijection.
\end{proof}

Now we are ready to define the main bijection.

\begin{definition}
\label{map tubed RTs to connected CDs}
Let $\tau$ be a tubing of a plane rooted tree $t$ with $n$ vertices. Then, define the chord diagram $\theta(\tau)$ as follows.

Base case: if $t=\bullet$ then $\tau$ is the unique tubing of $\bullet$ and $\theta(\tau)$ is the unique chord diagram with one chord.

Recursive step: Let $t''$, $t'$, $\tau''$, and $\tau'$ be as in \cref{rem tau prime and tau double prime}. Let $C''$ and $C'$ be $\theta(\tau'')$ and $\theta(\tau')$ respectively. Let $x$ be the insertion place of $t''$ where $t'$ is inserted into $t''$. Let $y$ be the CDIP of $C''$ given by $f(x)$. 
Then define $C= \theta(\tau)$ as follows:
Insert $C'$ into $y$ such that the resulting chord diagram is a disconnected indecomposable chord diagram. Let $q'$ be the root chord of $C'$. Take the left end of $ q'$ and move it before all chords of $C''$. The resulting chord diagram is $C$. Graphically, the procedure is shown in \cref{figure bijection construction}.
\end{definition}

\begin{figure}[htb]
    \centering
    \input{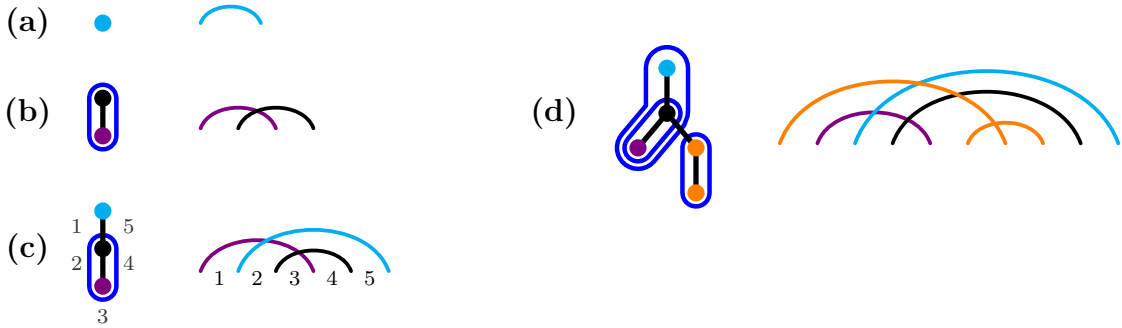}
    \caption{Step by step construction of the chord diagram corresponding to a tubing under $\theta$ (\cref{map tubed RTs to connected CDs}. Corresponding objects have the same color. {\bfseries (a)} a single vertex (with a tube which is not shown) corresponds to a single chord. {\bfseries (b)} A second vertex (purple) has been attached below the upper one. The new chord crosses the old one. {\bfseries (c)} The tube of b) is attached below a single vertex. The resulting chord diagram has five   insertion places (\cref{chord diagram insertion place}), indicated by small numbers.  {\bfseries (d)} another copy of b) is attached at position 4 to the tree c).  }
    \label{figure bijection construction}
\end{figure}

\begin{remark} \label{remark:map tubed RTs to CDs}
The description provided in \cref{map tubed RTs to connected CDs} is a more intuitive way to describe the map $\theta$. The following is a more formal definition of the recursive step of the map.  With $t'', t', \tau'', \tau', C''$, and $C'$ as above, define $C=\theta(\tau)$ as follows:

Let $\{1, \dots, 2n\}$ be the vertices of $C''$ with CDIP $y=\{i,i+1\}$. 
Similarly let $\{1, \dots, 2m\}$ be the vertices of $C'$ with root chord $(1,k)$. The vertices of $C$ are $\{1, \dots, (2n+2m)\}$.
The chords of $C$ are defined from the chords of $C''$ and $C'$ as follows:

\begin{enumerate}[label=(\roman*)]
\item For each chord $\{a_j,b_j\}$ from $C'$:

\begin{itemize}
    \item the root chord $\{1,k\}$ becomes $\{1,k+i\}$
    \item all other chords $\{a_j,b_j\}$ become $\{a_j+i,b_j+i\}$.
\end{itemize}

\item For each chord $\{a_j,b_j\}$ from $C''$:
\begin{itemize}
    \item if $a_j \leq i$ and $b_j \leq i$, then $\{a_j,b_j\}$ becomes $\{a_j+1,b_j+1\}$ 
    \item if $a_j \leq i$ and $b_j > i$, then $\{a_j,b_j\}$ becomes $\{a_j+1,b_j+2m\}$ 
    \item if $a_j > i$ and $b_j \leq i$, then $\{a_j,b_j\}$ becomes $\{a_j+2m,b_j+2m\}$ 
\end{itemize}

(or to say it another way, if $a_j \text{ (or } b_j) \leq i$, then $a_j \text{ (or } b_j)$ becomes  $a_j +1 \text{ (or } b_j +1)$, and if $a_j \text{ (or } b_j) > i$, then $a_j \text{ (or } b_j)$ becomes  $a_j +2m \text{ (or } b_j + 2m)$).
\end{enumerate}
\end{remark}

Write $\tau' \oplus_{x} \tau''$ for the tubing obtained by inserting $\tau'$ into $\tau''$ at insertion place $x$ and write $C' \oplus_{y} C''$ for chord diagram obtained by inserting $C'$ into $C''$ at insertion place $y$ in the above manner.

\begin{lemma}\label{lem:IP shift}
    Let $C''$ be a chord diagram with insertion places $y''_1,\dots,y''_m$. Let $C'$ be a chord diagram with $k$ CDIPs that is inserted into $C''$ at $y''_j$ to form chord diagram $C$. Then the insertion places $y''_1,\dots,y''_{j-1}$ in $C''$ become $y_2,\dots, y_{j}$ in $C$. The insertion place $y''_i$ gets split up by the insertion of $C'$ and together with the insertion places of $C'$ becomes $y_{j+1},\dots y_{j+k+1}$ in $C$. Finally, the insertion places $y''_{j+1},\dots,y''_m$ of $C''$ become $y_{j+k+1},\dots, y_{m+k+1}$ in $C$. 
\end{lemma}

\begin{proof}
    The addition of $C'$ which contains $k$ insertion places adds $k+1$ insertion places to the chord diagram because the insertion cuts $y''_j$ into two IPs: the one to the left of the inserted chords and the one to the right. The indices of the insertion places change because the root chord of $C'$ is moved all the way to the left of $C''$, thereby increasing the index of all insertion places of $C''$ by at least 1. All of the insertion places after $y''_j$ are increased by $k+1$ because of the addition of the $k$ insertion places from $C'$. An example is shown in \cref{figure cdip insertion}.
\end{proof}

\begin{figure}[htb]
	\centering
	\input{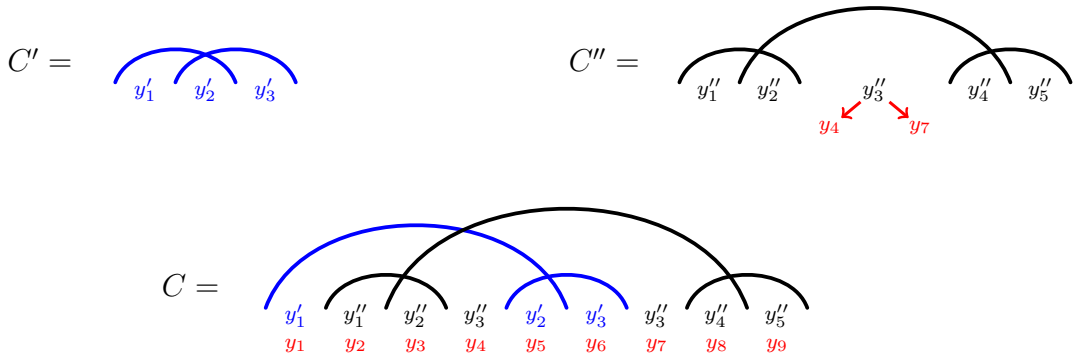}
	\caption{In this example, $C'$ has $k=3$ insertion places. $C'$ is inserted into $C''$ at $y''_3$. In $C$, the insertion places relative to $C'$ and $C''$ are labelled in blue and black, respectively, and the insertion places of $C$ are labelled in red. We can see that $y''_3$ is split into $y_4$ and $y_7$, together with $y_5,y_6$, which come from $C'$.}
	\label{figure cdip insertion}
\end{figure}

Note that the insertion of $C$ into $C''$ similarly splits the insertion place $y'_1$ of $C'$, but in this lemma, aim to see how the insertion places of $C''$ change in $C$.

\begin{theorem} \label{bijection tubed RTs to connected CDs}
$\theta$ is a bijection from the set of tubed rooted plane trees with $n$ vertices to the set of rooted connected chord diagrams with $n$ chords, and its inverse map is called $\mu$.
\end{theorem}
\begin{proof}
To prove this, we construct the inverse map, $\mu : C \mapsto \tau$.

Let $C$ be a rooted connected chord diagram.
If $C$ is just one chord, then define 
$\mu(C)$ to be the unique tubing of $\bullet$. 
Otherwise, let $r$ be the root chord. The removal of the root chord results in either a connected chord diagram, or a disconnected, indecomposable chord diagram. Note that this indecomposable diagram will be a sequence of nested connected components, as in \cref{intersection order}. Recall from \cref{intersection graph} that a chord diagram is disconnected if there exist chords that do not intersect with some other chord(s).
In the former case, we let $r$ be $C'$ and $C \backslash r$ be $C''$.
In the latter case, we take $r$ together with all but the outermost component of $C \backslash r$ and call it $C'$, and let the rest of the chord diagram, namely the outer component, be $C''=C \backslash C'$.
In both cases,  $C''$ and $C'$ are each connected rooted chord diagrams.

The last thing we need is the insertion place where $C'$ is inserted into $C''$ to give the chord diagram $C$.  To specify this, let $\eta$ be the smallest index in $C$ greater than 1 of a chord end in $C'$. In $C$, the insertion place to the left of this chord is $y_{\eta-1}$.
%
%
%
So, looking at insertion places of $C''$, we see that $C=C'\oplus_{\eta-2, \eta-1} C''$ and so $y_{\eta-2}$ is the insertion place in $C''$ we are looking for. This is because when $C'$ is inserted into $C''$, the root chord of $C'$ is shifted all the way to the left of $C''$, as described in \cref{map tubed RTs to connected CDs}; and by \cref{lem:IP shift}, the CDIPs to the left of $C'$ in $C$ are increased by 1 after insertion, thus in the inverse, we subtract 1.
Recall from \cref{map RTIP to CDIP} and \cref{bijection RTIP to CDIP} that there exists a map $f^{-1}$ that takes a CDIP and gives an RTIP. 
Recursively let $\tau''=\mu(C'')$ and $\tau'=\mu(C')$  and $x=f^{-1}(y)$ as an RTIP of $t''$.
Now define $t$ to be the insertion of $t'$ into $t''$ at $x$ and define $\tau$ to be the tubing of $t$ given by the tubes of $\tau''$ and $\tau'$ along with one tube containing all of $t$. Then $t$ and $\tau$ give $\mu(C)$.

Directly from these definitions, we see that $\mu$ and $\theta$ are mutual inverses, and the number of vertices in $t$ is equal to the number of chords in $C$. 
\end{proof}

\cref{figure bijection example corollas} shows the bijection $\theta$ for the tubings of a 4-vertex corolla tree, \cref{figure bijection example ladders} for the 4-vertex ladder tree, and \cref{figure bijection example} shows all other 4-vertex trees. 

\begin{figure}[htb]
    \centering
    \input{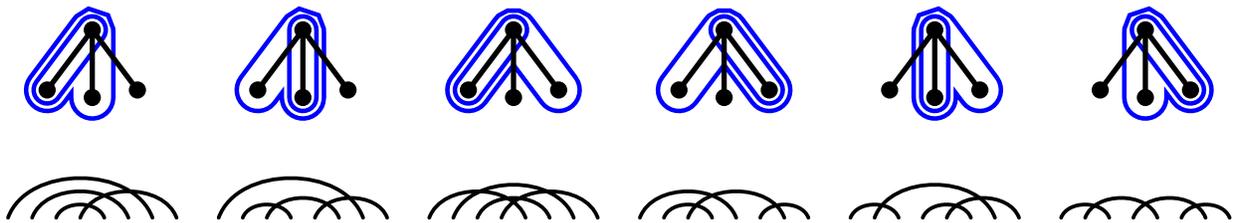}
    \caption{Chord diagrams corresponding to the tubings of the corolla tree $s_4$ of \cref{figure corollas} under the bijection $\theta$ (\cref{map tubed RTs to connected CDs}).  }
    \label{figure bijection example corollas}
\end{figure}

The bijection $\theta$ behaves nicely on the special classes of tubings considered earlier in \cref{sec examples}, including ladders and leaf tubings, as we now demonstrate. We first require several additional notions picking out the associated classes of chord diagrams (see for example \cite{nabergall_combinatorics_2023} for full formal definitions and further background on these notions). 

A {\em subdiagram} $D'$ of a chord diagram $D$ is a subset of its chords. We identify a subdiagram with the chord diagram obtained by standardizing its points to the interval $\{1, 2, \ldots, 2|D'|\}$. A chord diagram $D$ is {\em permutation} if all of its sinks lie to the right of all of its sources. Labelling the sinks of $D$ by the order of their attached sources gives a permutation that determines $D$. Permutation diagrams can also be characterized as those diagrams do not have $\{(1, 2), (3, 4)\}$ as a subdiagram. For a permutation $\sigma$, a permutation diagram $D$ is {\em $\sigma$-avoiding} if $D$ does not contain the permutation diagram determined by $\sigma$ as a subdiagram. Equivalently, if $\pi$ is the permutation determining a permutation chord diagram $D$, then $D$ is $\sigma$-avoiding if $\pi$ does not contain $\sigma$ as a pattern. Recall that a permutation $\pi$ contains $\sigma$ as a pattern if there exists an injection from the elements of $\sigma$ to the elements of $\pi$ that preserves both the left-to-right order and ground set order. 

\begin{proposition}\label{prop interesting chord classes}
The map $\theta$ restricts to 
\begin{enumerate}
\item a bijection between tubed ladders and connected 213-avoiding permutation diagrams, and
\item a bijection between leaf tubings and 1-terminal chord diagrams. 
\end{enumerate}
\end{proposition}
\begin{proof}
Note that a diagram is permutation and 213-avoiding if and only if all of its subdiagrams are 213-avoiding and are permutation diagrams. Furthermore, for connected chord diagrams $C'$ and $C''$, observe that $C' \oplus_{x} C''$ is permutation if and only if $C'$ and $C''$ are permutation and $x = |C''|$. Now suppose $C'$ and $C''$ are 213-avoiding permutation diagrams and write $\sigma'$ and $\sigma''$ for the permutations determining $C'$ and $C'$, respectively. Then the permutation $\sigma$ determining $C' \oplus_{n} C''$ can be written as $\pi'\pi''$, where $\pi''$ is obtained from $\sigma''$ by increasing every letter by 1 and $\pi'$ is obtained from $\sigma'$ by increasing every letter except 1 by $n = |C''|$. In particular, other than 1 every letter of $\pi'$ is larger than any letter of $\pi''$, so if 213 is a pattern of $\pi'\pi''$ then any instance of the pattern must lie in the first $|\pi'|$ letters, contradicting the assumption that $\pi'$ is 213-avoiding. It follows that $C' \oplus_{n} C''$ is 213-avoiding. Finally, note that the rooted tree insertion place for ladders is always $n$. Then combining all of the above facts and applying the construction of $\theta$, we conclude that $\tau$ is a tubed ladder if and only if $\theta(\tau)$ is a connected 213-avoiding permutation diagram, as desired for the first point. 

We now turn to the second point. By construction, a tubing $\tau$ is a leaf tubing if and only if $\theta(\tau)$ is constructed from the empty diagram by iteratively inserting root chords in such a way that the diagram is connected at each step. By  \cref{1-terminal characterization} this is equivalent to $\theta(\tau)$ being 1-terminal.  
\end{proof}

\begin{figure}[htb]
    \centering
    \input{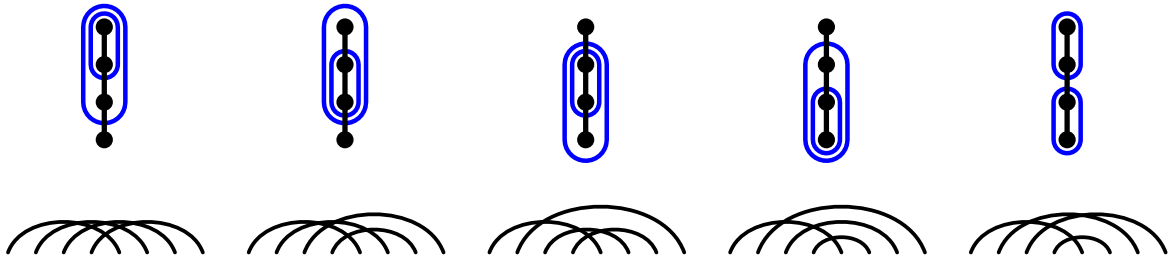}
    \caption{Chord diagrams corresponding to the tubings of the ladder tree $\ell_4$ of \cref{figure ladders} under the bijection $\theta$ (\cref{map tubed RTs to connected CDs}). Observe that these chord diagrams are different from the ones constructed in \cref{rem ladders chord diagrams}}
    \label{figure bijection example ladders}
\end{figure}

Connected 213-avoiding permutation diagrams are well known to be counted by the Catalan numbers, providing an alternative proof of  \cref{lem ladders tubings}. Similarly, since 1-terminal chord diagrams have been shown to be counted by the double factorials $(2n-3)!! = 1\cdot 3\cdots (2n-5)(2n-3)$ (see \cite{courtiel_terminal_2017, nabergall_combinatorics_2023}), the above result implies that double factorials count leaf tubings (compare the discussion in the beginning of \cref{{sec examples}}) as well. 

As we showed in  \cref{lem chains tubings}, the number of tubings of corollas with $n$ vertices is given by $(n-1)!$. Following the above result it is natural to ask whether this class of tubings is also mapped to a nice class of chord diagrams under $\theta$. As immediately implied by the definition, the number of permutation diagrams with $n$ chords is given by $n!$. Additionally restricting to those permutation diagrams that are 1-terminal gives a natural class of chord diagrams counted by $(n-1)!$. Together with  \cref{prop interesting chord classes} and the fact that all tubings of corollas are leaf tubings, this suggests that tubed corollas would most naturally correspond to 1-terminal permutation diagrams. But the map $\theta$ clearly does not restrict to a bijection between tubed corollas and 1-terminal permutation diagrams -- e.g. see the fourth pair in the first column of  \cref{figure bijection example}. 

We now briefly present an alternative bijection, other than $\theta$ (\cref{map tubed RTs to connected CDs}), between leaf tubings and 1-terminal chord diagrams that does restrict to the desired map on tubed corollas. The bijection is essentially described recursively in \cite{nabergall_combinatorics_2023,nabergall_enumerative_2022}, but here we give a non-recursive definition via \emph{decreasing trees}, ordered trees with a decreasing labelling. The map given in the proof of \cref{leaf_label_lemma} then completes the description. For a 1-terminal chord diagram $T$, let $c_{1}, \ldots, c_{n}$ be the chords of $T$ ordered by their sources. Furthermore, define the \emph{sink group} of a chord $c \in T$ to be the sequence of chords (in order of their sinks) attached to a sink in the maximal (possibly empty) interval of sinks immediately to the right of the source of $c$. Then define $\kappa(T)$ to be the tree with vertices labelled $1, \ldots, n$ such that the children of the vertex labelled $i$ are the vertices labelled $i_{1}, \ldots, i_{k}$, where $c_{i_{j}}$ is the $j$th chord in the sink group of $c_{i}$. It is straightforward to see that $\kappa$ is a bijection between 1-terminal chord diagrams with $n$ chords and decreasing trees with $n$ vertices. Furthermore, by 1-terminality and the fact that the terminal chord (the rightmost chord) of a 1-terminal permutation diagram is the only chord with a nonempty sink group it readily follows that $\kappa$ restricts to a bijection between 1-terminal permutation diagrams and decreasing corollas, as desired. Observe that this bijection distentangles the underlying tree of the tubing, identifying it with a simple feature of 1-terminal diagrams, sink groups, in a way that $\theta$ does not appear to.

\subsection{The tubing expansion matches the chord diagram expansion}
\label{sec tubing expansion matches}

We have shown that the tubing expansion of \cref{thm main thm} is a solution of the Dyson--Schwinger equation \cref{dse_differential_general}, and that the map $\theta$ is a bijection between tubings and chord diagrams (\cref{bijection RTIP to CDIP}). However, this is not enough yet to conclude that the chord diagram expansion solutions of Dyson--Schwinger equations of \cite{marie_chord_2013, hihn_generalized_2019} are the result of applying $\theta$ to the tubing expansion of \cref{thm main thm} since they could be two different expansions indexed by the same objects; that is, two expansions could associate different terms of the formal power series expansion to the same chord diagram and so collect together terms of the formal power series solution differently. In order to show that the two expansions do line up it remains to check that the parameters used in the two expansions line up diagram-by-diagram.  Showing this is the goal of the present section.

\medskip

The solutions to Dyson--Schwinger equations of \cite{marie_chord_2013, hihn_generalized_2019} are indexed by rooted connected chord diagrams with weighted chords.  Specifically, a \emph{weighted rooted chord diagram} is a rooted chord diagram $C$ (\cref{rooted chord diagram}) along with a map $w$ from the chords of $C$ to the positive integers.  For a chord $a$, $w(a)$ is called the weight of $a$.  The weight of the chord diagram, written $\|C\|$ is the sum of the weights of the chords.  We will also conflate chords and their indices in intersection order (\cref{intersection order}) and so write $w(i)$ for the weight of the $i$th chord in intersection order.  Also for a chord diagram $C$ with terminal chords indexed by $t_1<t_2<\cdots < t_j$ and sequences $\{c_{m,n}\}$ define
\[
c(C) = \prod_{i=2}^{j} c_{(t_i-t_{i-1}) , w(t_i)} \prod_{a \text{ not terminal}} c_{0, w(a)}.
\]

We need to define one more parameter of a chord before we can state the chord diagram expansion result and that is the definition of the branch-left or $\nu$ parameter.  This definition requires an auxiliary binary tree, but fortunately, we do not need the leaf labelling of this binary tree, only its shape.  These binary trees are the kind which is given by either a root alone or a root and two binary trees, called the left subtree and the right subtree.  Note that in this case the subtrees may not be empty, so this is a type of full binary tree.  Given two such binary trees $T_1$ and $T_2$ define the insertion operation $T_1 \circ_k T_2$ with $k$ at most the number of vertices of $T_2$ as follows (see Definition 3.8 of \cite{marie_chord_2013} and Definition 4.1 of \cite{hihn_generalized_2019}, and the example in \cref{figure branch left insertion}). 
\begin{itemize}
    \item Label the vertices of $T_2$ following a pre-order traversal\footnote{The notion of tree traversal comes from theoretical computer science. Mathematically a traversal of a tree is a total order on its vertices, but following the computer science language, it is usually spoken of as a procedure wherein the vertices are met sequentially following this total order.  Specifically, for a plane rooted tree, the \emph{pre-order traversal} meets the root first, then meets all the vertices of the subtree rooted at the first child of the root, recursively following a pre-order traversal on this subtree, then all the vertices of the subtree rooted at the second child of the root following a pre-order traversal and so on.} of $T_2$.  
    \item Put a new vertex in the middle of the edge leading from vertex $k$ of $T_2$ to the parent of that vertex making the subtree rooted at vertex $k$ the right subtree of this new vertex.  If $k=1$ so the vertex is the root of $T_2$, then put a new vertex above the root of $T_2$ with $T_2$ as the right subtree.
    \item Place $T_1$ as the left subtree of the new vertex.
\end{itemize}

\begin{figure}[htbp]
    \centering
    \input{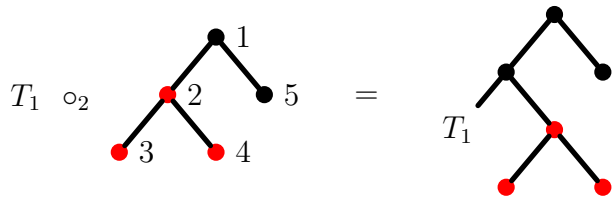}
    \caption{Insertion operation  $T_1 \circ_k T_2$ to build the auxiliary binary tree.}
    \label{figure branch left insertion}
\end{figure}

In a similar way to the construction of $\mu$ in \cref{bijection tubed RTs to connected CDs}, this decomposition allows us to recursively define a map from rooted connected chord diagrams to binary trees.  Specifically, map the rooted connected chord diagram on one chord to the binary tree on one vertex.  Given a rooted connected chord diagram $C$ with more than one chord, decompose it into the outermost connected component after removing the root, $C_2$ and the rest of the diagram, $C_1$, and let $k$ be the insertion place in $C_2$ where $C_1$ less its root was inserted.  Recursively apply this map to $C_1$ and $C_2$ to get $T_1$ and $T_2$ respectively and then define the result of this map on $C$ to be $T_1 \circ_k T_2$. See \cite{marie_chord_2013} Definition 3.9 or \cite{hihn_generalized_2019} Definition 4.3 for details.  Note that this map is not a bijection.  It can be made injective by labelling the leaves in a way that allows us to recreate the chord diagram, however, then the map is not surjective onto all leaf-labelled binary trees.  A characterization of which leaf-labelled binary trees occur is given in \cite{marie_chord_2013} Theorem 4.8 but it is quite unsatisfactory.  Fortunately, for the present purposes, we do not need the leaf labelling nor any form of bijectivity, but it is important to note that under the map to binary trees every chord of the chord diagram corresponds to a leaf of the binary tree and every leaf comes from a chord of the chord diagram.

Now, given a rooted connected chord diagram $C$ and a chord $a$ of $C$, define $\nu(a)$ to be the number of edges in the path that begins at the vertex of the binary tree associated to $a$ and goes up and to the left so long as that is possible.

The chord diagram expansion result is that when $s$ is a negative integer, the solution to the Dyson--Schwinger equation \cref{dse_differential_general} with $F_k(\rho) = \sum_{i\geq 0} c_{i,k}\rho^{i-1}$ can be written as (Theorem 7.4 of \cite{hihn_generalized_2019})
\begin{equation}\label{eq chord soln}
G(x,L) = 1+\sum_{C} \left(\prod_{a \text{ chord of } C}(-1)^{\nu(a)}\binom{w(a)(-s)+\nu(a)-2}{\nu(a)}\right) c(C)x^{\|C\|}\sum_{k = 1}^{b(C)}\frac{L^k}{k!}c_{t_1-k,w(t_1)}
\end{equation}
where the sum runs over weighted rooted connected chord diagrams with terminal chords indexed $t_1<t_2<\cdots t_j$ in intersection order and $b(C)=t_1$ is the index of the first terminal chord. The signs in \cref{eq chord soln} are due to different conventions between the present paper and \cite{hihn_generalized_2019}.  Using the negative binomial identity, we can rewrite the  factor in parentheses as
\begin{equation}\label{eq better binomial factors}
\prod_{a \text{ chord of } C}\binom{1+w(a)s}{\nu(a)}.
\end{equation}

\medskip

Our claim now is that this expansion agrees with the tubing expansion (\cref{thm main thm}) under the action of $\theta$ (\cref{map tubed RTs to connected CDs}).  We will need quite a few preliminary lemmas towards the proof of this claim.

\begin{lemma}\label{lem root first term}
Let $\tau$ be a tubing of a rooted plane tree $t$ and let $C= \theta(\tau)$ be the associated chord diagram under the bijection $\theta$ of \cref{map tubed RTs to connected CDs}. Then the first terminal chord of $C$ corresponds to the root vertex of $t$.
\end{lemma}
Note that $\theta$ induces a bijection between the vertices of $t$ and the chords of $C$.

\begin{proof}
By Proposition 22 of \cite{courtiel_terminal_2017} the first terminal chord is the chord with the rightmost endpoint.

The proof of the lemma is inductive.  When $t$ has only one vertex then $C$ has only one chord.  That one chord is terminal and corresponds to the root of $t$.

Suppose $t$ has two or more vertices.  Then $\tau$ has two tubes immediately inside the outermost tube and these two tubes partition the vertices of $t$ into two trees.  Let $t'$ and $t''$ be these two trees with $t''$ containing the root of $t$ and let $\tau'$ and $\tau''$ be the induced tubings on $t'$ and $t''$ respectively as discussed in \cref{rem tau prime and tau double prime}.  
Inductively, the first terminal chord of $\theta(\tau'')$ corresponds to the root of $t''$.  In inserting $\theta(\tau')$ into $\theta(\tau'')$ so as to form $\theta(\tau)$, the chords of $\theta(\tau')$ have all their endpoints either strictly inside $\theta(\tau'')$ or to the left of $\theta(\tau'')$.  In particular the chord with the rightmost endpoint of $\theta(\tau)$ comes from the chord with the rightmost endpoint of $\theta(\tau'')$, hence corresponds to the root of $t''$, and thus to the root of $t$.
\end{proof}

\begin{lemma}\label{lem root term}
With $\tau, t, C$ as in \cref{lem root first term}, 
the set of terminal chords of $C$ corresponds to the set containing the root of $t$ along with all vertices of $t$ that are the root of at least one tube containing two or more vertices.  
\end{lemma}

Note that the root of $t$ is always the root of the outermost tube of $\tau$, so the only case where it needs to be separately included as in the statement of the lemma is the case where $t$ has exactly one vertex.  The chord of the one chord diagram is terminal and so the vertex of unique tubing of the one vertex tree must also be terminal.  The one chord diagram is the only rooted connected chord diagram whose root chord is terminal.

\begin{proof}
The proof of the lemma is again inductive.  When $t$ has only one vertex then $C$ has only one chord.  That one chord is terminal and is included among the set of vertices described in the statement.

Suppose $t$ has 2 or more vertices.  Let $t'$, $t''$, $\tau'$, $\tau''$ be as in the proof of the previous lemma and \cref{rem tau prime and tau double prime}.

The only new crossings created by the insertion of $\theta(\tau')$ into $\theta(\tau'')$ in order to build $\theta(\tau)$ are those between the root of $\theta(\tau')$ (which becomes the root of $\theta(\tau)$ and one or more chords of $\theta(\tau'')$.  Since the root chord is the first chord, these new crossings can not cause any chords of $\theta(\tau'')$ to stop being terminal.  The new crossings stop the root of $\theta(\tau')$ from being terminal, if it was before, which only happens when $t'$ is the one vertex tree.  The insertion of $\theta(\tau')$ into $\theta(\tau'')$ in order to build $\theta(\tau)$ does not remove any crossings, so it does not cause any previously non-terminal chords to become terminal.

Now consider the insertion of $t'$ into $t''$ in order to build $t$.  Any vertex that is in a tube of size $2$ or more in $\tau'$ or $\tau''$ remains so in $\tau$ since the tubes of $\tau$ are exactly the tubes of $\tau''$, the tubes of $\tau'$ and the outermost tube.  If $t'$ had only one vertex, then that vertex is in its own tube and in the outermost tube in $\tau$, so it is not the root of a tube containing 2 or more vertices in $\tau$.  This exactly mirrors the situation with the terminal chords, so using the statement inductively on $\tau''$ and $\tau'$ we get that the vertices of $t$ corresponding to terminal chords are exactly those which are roots of tubes containing 2 or more vertices.
\end{proof}

\Cref{figure bijection example} shows a larger number of 4-vertex tubings with emphasis on the terminal chords as given by \cref{lem root first term,lem root term}.

\begin{figure}[htb]
    \centering
    \input{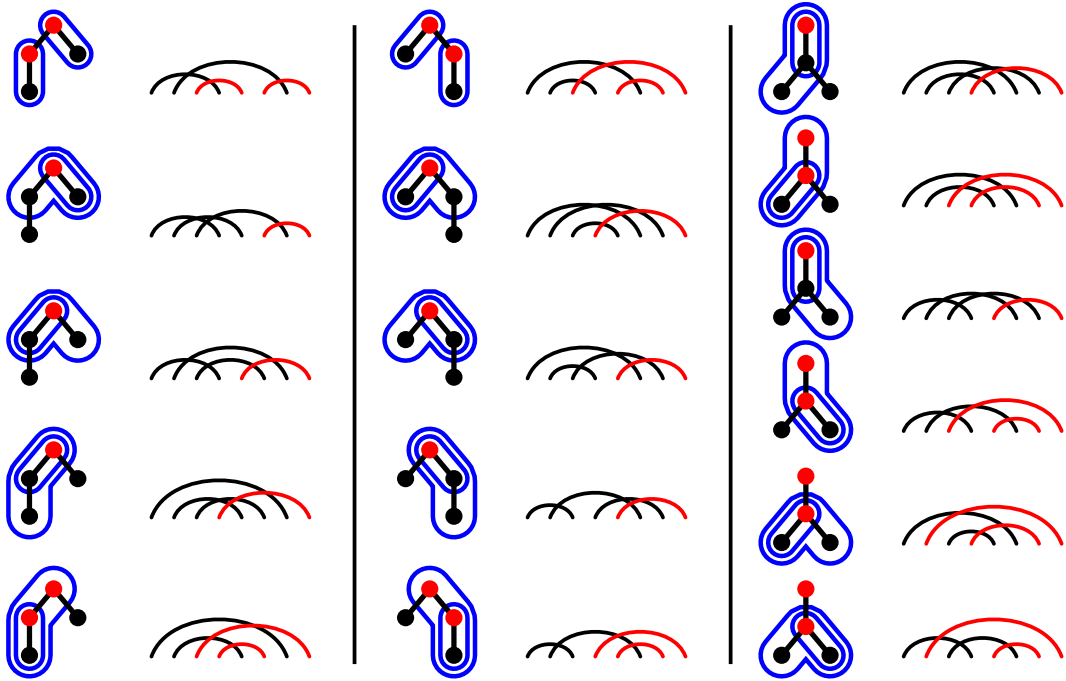}   
    \caption{The bijection between tubings and chord diagrams for all tubings on 4 vertices except for ladders and corollas.   Terminal chords / vertices are drawn in red, they satisfy \cref{lem root term}.    }
    \label{figure bijection example}
\end{figure}

\begin{lemma}\label{lem root b}
With $\tau, t, C$ as in \cref{lem root first term}, let $t_1$ be the index of the first terminal chord of $C$ in the intersection order (\cref{intersection order}),
then the $b$ statistic (\cref{def b statistic}) is $b(\tau) = t_1$ 
\end{lemma}

\begin{proof}
When $t$ has only one vertex then $C$ has only one chord.  That one chord is terminal and has index $1$. The root vertex is in exactly one tube so $b(\tau) = t_1$.

Suppose $t$ has 2 or more vertices.   Let $t'$, $t''$, $\tau'$, $\tau''$ be as in the proof of the previous lemma and \cref{rem tau prime and tau double prime}.

The index of the first terminal chord in $\theta(\tau'')$ is one less than the index of the first terminal chord in $\theta(\tau)$ since in the intersection order for $\theta(\tau)$, the root comes first, then all the chords of $\theta(\tau'')$, then other chords of $\theta(\tau')$.  Similarly, $b(\tau'') -1 = b(\tau)$ since the outermost tube contains the root in $\tau$ and otherwise the tubes containing the root are the same as in $\tau''$.

The result follows by induction.
\end{proof}

\begin{lemma}
With $\tau, t, C$ as in \cref{lem root first term}, 
let $t_1 < t_2 < \cdots < t_k$ be the indices of the terminal chords of $C$ in the intersection order, and let $v_1, v_2, \ldots, v_k$ be the vertices corresponding to the terminal chords with $v_i$ corresponding to the terminal chord with index $t_i$.
Then $b(v_i, \tau)-1 = t_i-t_{i-1}$ for $1 < i \leq k$.
\end{lemma}

Note that for the use we are making of the terminal chords, \cref{eq chord soln}, the indices $t_i$ do not matter, only $t_1$ and the $t_i-t_{i-1}$ matter.  Conveniently, these values are easy to read off the tubing, while the indices themselves are messier to see at the level of the tubing.

\begin{proof}
Proceed one final time in the same way.  When $t$ has only one vertex then $C$ has only one chord.  Then $k=1$ so there is nothing to prove.

Suppose $t$ has 2 or more vertices.   Let $t'$, $t''$, $\tau'$, $\tau''$ be as in the proof of the previous lemma and the discussion after  \cref{lem recursive tubing}.

Recall that the intersection order on $\theta(\tau)$ is, first the root (which comes from the root of $\theta(\tau')$, then all the chords of $\theta(\tau'')$ in intersection order, then the remaining chords of $\theta(\tau')$ in intersection order.   These facts together tell us that the differences $t_i-t_{i-1}$ are all the differences of indices of terminal chords from $\theta(\tau'')$ and from $\theta(\tau')$ along with the difference between the index of the chord of $\theta(\tau)$ which comes from the first terminal chord of $\theta(\tau')$ and the index of the chord of $\theta(\tau)$ which comes from the last terminal chord of $\theta(\tau'')$.  Let $t_j$ and $t_{j-1}$ be respectively the indices of these two chords in $\theta(\tau)$. 

Note that the last chord of a diagram is always terminal, so along with the structure of the intersection order, we obtain $t_{j-1} = |\theta(\tau'')| + 1$ and, using  \cref{lem root b} $t_j = |\theta(\tau'')| + b(\tau')$.  Therefore $t_{j} - t_{j-1} = b(\tau')-1$.
 
Now returning to the tubings, $v_1$ is the root of $t$ by  \cref{lem root first term}, so for $1<i\leq k$, $v_i$ is not the root of $t$ and so $b(v_i, \tau^\ell) = b(v_i, \tau)$ where $\ell$ is ${}'$ or ${}''$ according to whether $v_i$ is in $t'$ or $t''$.  Therefore, inductively $b(v_i,\tau)-1 = t_i-t_{i-1}$ for $i\neq j$, $1<i\leq k$, and $b(v_j, \tau) - 1 = b(v_j, \tau') -1 = b(\tau')-1 = t_{j}-t_{j-1}$ using the explicit calculation above and thus completing the induction.
\end{proof}

\begin{lemma}
Let $C$ be a rooted connected chord diagram, let $t$ be the underlying tree of the tubing $\mu(C)$, and let $T$ be the binary tree obtained from $C$ as discussed at the beginning of this subsection. Then, $t$ can be obtained from $T$ by contracting all paths in $T$ that start at a leaf and move up and to the left as long as that is possible.   

In particular, if $a$ is a chord of $C$ and corresponds to $v\in V(t)$ then $\nu(a) = \od{v}$. 

Furthermore, if $C$ is decorated or weighted and the decorations or weights are assigned to the leaves of $T$ and the vertices of $t$ correspondingly, then contracting the upper-left paths as described above and associating the weight of the unique leaf in the path to the resulting vertex, the weights obtained for $t$ through $T$ agree with the weights for $t$ obtained directly from $C$ via $\mu$.
\end{lemma}

\begin{proof}
The key is to show that vertices of $T$ correspond to the insertion places of $t$ after the upper-left paths are contracted.  To that end suppose for the moment that $T$ is any binary rooted tree and $t$ the result of contracting the upper-left paths of $T$. 

When contracting the upper-left paths, each edge that went to a right child gets contracted and each edge that goes to a left child remains an edge in the new tree. Each upper-left path becomes a vertex with children attached by the edges which led to left children of vertices on the path.  If the upper-left path has $k$ edges then there are exactly $k$ such left children and so the contracted vertex itself has $k$ children.  This proves that the second statement of the lemma is implied by the first statement of the lemma.  Note that these $k$ children inherit an order from their order along the upper-left path.  Every upper-left path has exactly one leaf at its lower end.  The leaves that are left children have no upper-left edge and fit into this picture by having upper-left paths of length $0$.  Including the leaves that are left children in this way every leaf of $T$ corresponds to an upper-left path and every vertex of $T$ is in exactly one upper-left path, so the vertices of $t$ are in bijection with the leaves of $T$ and $t$ is a plane rooted tree.  

Insertion in $T$ acts by picking a vertex $v$ and then inserting above $v$ with $v$ as the new right child and with the root of the inserted tree as the new left child $v$.  The edge from the new vertex to $v$ is then an edge that will be contracted in an upper-left path and so $v$ and the new vertex belong to the same upper-left path while the inserted tree will be a child of the vertex in $t$ resulting from the contraction of this upper-left path.  This is the usual insertion into $t$.  Specifically, if $v$ is at neither end of its upper-left path, then this is insertion into $t$ into the insertion place between the child of $v'$ given by $v$'s parent's left child and the child of $v'$ given by $v$'s left child, where $v'$ is the vertex in $t$ obtained by the contraction of the upper-left path containing $v$.  If $v$ was a leaf, then $v$ has no left child so the insertion place in $t$ is the insertion place after every child of $v'$.  If $v$'s parent was not in the upper-left path or $v$ was the root, then the insertion place in $t$ is the insertion place before every child of $v'$.

Furthermore, the indexing of the insertion places in $T$ is according to a pre-order traversal.  Contracting the upper-left paths, each vertex of the path becomes an insertion place of $t$ as described above, but the order these places are met in the traversal does not change, so the indexing of the insertion places agrees with the indexing of the insertion places in $t$ \cref{rooted tree insertion place}.

Finally, now return to a rooted connected chord diagram $C$ with $T$ and $t$ obtained from it as in the statement of the lemma.  The above discussion demonstrated that insertion in $T$ and insertion in $t$ correspond and that the indexing of the insertion places corresponds upon contracting upper-left paths of $T$.  Additionally, the one vertex $T$ contracts trivially to the one vertex $t$.  Inductively, then, the constructions of $T$ and $t$ from $C$ correspond after contracting the upper-left paths of $T$ implying the first statement of the lemma.  As observed earlier in the proof, the first statement of the lemma implies the second statement of the lemma.  The identification between chords of $C$, leaves of $T$ and vertices of $t$ carries through the induction, proving the final statement of the lemma.
\end{proof}

\begin{remark}
One of the arguably unsatisfactory things about \cite{marie_chord_2013} and \cite{hihn_generalized_2019} was the need to move to these binary trees in order to define the $\nu$ statistic and for other parts of the constructions.   One question that was never able to be answered in the past was what set of leaf-labelled binary trees is generated.  The result above finally gives a kind of answer.  The leaf labelling existed in order to make explicit the bijection between the leaves and the chords, and it was necessary to get an injective map from chord diagrams to some sort of binary rooted trees, as without the leaf labelling the map is highly non-injective.

The results of the present paper give us a much more natural and elegant way to record enough extra information in a tree format, namely the tubings.  By \cref{bijection tubed RTs to connected CDs}, $\theta$  is a bijection between rooted connected chord diagrams and tubed rooted trees, so we finally have a bijection with a combinatorially nice image that carries information like $\nu$.  This is what we were looking for but failed to find in the original construction.
\end{remark}

\begin{theorem}\label{thm expansions agree}
The chord diagram expansion solutions \cref{eq chord soln} to the Dyson--Schwinger equation \cref{dse_differential_general} when $s$ is a negative integer agree with the tubing solutions of \cref{thm main thm} after applying the bijection $\theta$ (\cref{map tubed RTs to connected CDs}).
\end{theorem}

\begin{proof}
Combining the previous lemmas we see that the expansions line up term for term.  For the binomial factors, first use the argument discussed after \eqref{eq chord soln} to obtain binomial factors as in \eqref{eq better binomial factors}, which line up under $\theta$ with the alternate form of the tubing expansion for $G(x,L)$ given in the displayed equation immediately after the statement of \cref{thm main thm}.
\end{proof}

\begin{remark}\label{rem graph by graph}

    As mentioned at the beginning of \cref{sec chord}, the other main deficiency of the original chord diagram expansions is that despite starting with a Feynman graph expansion, which is combinatorially indexed, albeit with an intricate contribution from each graph, and moving to the combinatorially indexed chord diagram expansion, it was not clear which chord diagrams corresponded to the contribution of a given Feynman graph, or even if each Feynman graph's contribution could be written as a sum of individual chord diagram contributions.  The nature of the original proof made this potential connection entirely opaque.

    The original impetus for what became the work of this paper was the desire to answer the question of which chord diagrams give the contribution of a given Feynman graph.  Now we have done this, as we explain below.
    
    For concreteness, let us start with the Yukawa example (\cref{eg Yukawa chord,eg yukawa 1}).  In that case from  \cref{eg yukawa insertion trees}  and \cref{bijection tubed RTs to connected CDs} we see that tubings of the insertion trees of Yukawa graph are in bijection with rooted connected chord diagrams.  \Cref{thm expansions agree} and its proof tell us that the chord diagram and tubing expansions of the solution to the Dyson--Schwinger equation line up term by term under $\theta$ so we have a precise identification of the chord diagrams corresponding to the Yukawa graph by tubing the insertion tree of the Yukawa graph and applying $\theta$.  Finally,  the contribution of the Yukawa graph to the Feynman graph expansion of the Dyson--Schwinger equation agrees with the contribution of these tubings to the tubing expansion, because from \cite{broadhurst_combinatoric_2000} we know that the contribution of the Yukawa graph can be obtained through iterated linearized coproduct via the infinitesimal character $\sigma$ and its convolution powers $\sigma^{*k}$. In terms of renormalization group theory, this is \cref{gammak_rge}: The full Green function can be recovered from the anomalous dimension $\gamma(x)$, and the latter is the linear coefficient in $L$ of the involved Feynman graphs.
    
    In fact, this last argument regarding the convolution powers of $\sigma$ giving the Feynman rules is the argument given more mathematically in \cref{sec tubing feynman rules} and hence holds generally, not just for the Yukawa example.  That is, the Hopf algebraic structure of renormalization tells us that any Feynman graph will contribute to the Green function via linearized coproduct and $\sigma^{*k}$ in precisely the manner which we capture by the tubings of its insertion trees and the proof of this is the argument of \cref{sec tubing feynman rules} and applies to all the Dyson--Schwinger equations of the types we study here. Effectively, the tubings are nothing but a graphical representation of the iterated linearized coproduct applied to the relevant Feynman graphs, represented by their insertion trees.
    
    From there we can apply $\theta$ to get the contribution in terms of chord diagrams. 
    Thus we can finally answer the question of which chord diagrams correspond to the contribution to the Green function of a given Feynman graph, namely, tube the insertion tree of the Feynman graph in all possible ways and then apply $\theta$.
\end{remark}

\section{Conclusion}
\label{sec conclusion}

We discussed the relation between binary tubings, chord diagrams, and Dyson--Schwinger equations where graphs are inserted into only one edge. Concretely:
\begin{enumerate}
    \item Our main result \cref{thm main thm} expresses the series coefficients of the solution $G(x,L)$ of the DSE in terms of binary tubings. In particular, the anomalous dimension can be computed as a weighted sum of tubings.
    \item The tubings represent a 1:1 map between individual Feynman graphs, or their insertion trees, and the summand in the Green function which corresponds to this very graph.
    \item Our tubing expansions solve systems of Dyson--Schwinger equations, not just single equations.  This had not been achieved with the earlier chord diagram expansions.
    \item Restricting to ladder trees, we reproduce the known solutions of linear DSEs (\cref{sec ladders}).
    \item Ladders and corollas represent extreme cases with respect to the number of tubings (\cref{tubings_bound}), they imply a bound for the coefficients of the anomalous dimension of the Yukawa model (\cref{ex yukawa bound}).
    \item The solution of the DSEs in question is known to be encoded by chord diagrams. Our proofs  (\cref{sec tubing feynman rules}) are based entirely on tubed rooted trees, and do not refer to chord diagrams.
    \item We derive an explicit bijection between tubings and the corresponding chord diagrams (\cref{map tubed RTs to connected CDs}). Furthermore, we identify the counterparts of terminal chords, the branch-left parameter and the intersection order,  in terms of tubed trees (\cref{sec tubing expansion matches}).
    \item We find that under this bijection, interesting classes of rooted trees correspond to interesting classes of chord diagrams (\cref{prop interesting chord classes}).
    \item A slightly different notion of tubings has existed in the literature. In \cref{sec associahedra} we call the existing definition \emph{A-tubings} since they furnish the edges and vertices of graph associahedra; our tubings are in bijection with maximal A-tubings of line graphs (\cref{prop L bijection}). 
\end{enumerate}

From a physics perspective, the perhaps most striking point of our results is that they operate on objects which are well-known in Hopf algebra renormalization theory, namely the insertion trees associated to Feynman graphs, and the coproduct, which amounts to cutting these trees. As discussed in \cref{rem linearized coprod,rem graph by graph}, a tubing can be understood as one particular series of iterated coproducts applied to the tree, and each such series gives a contribution to the $L$-linear summand of the Green function, namely the anomalous dimension \cref{G_log_expansion}. The Connes--Kreimer Hopf algebra of rooted trees, which has proven valuable in the abstract understanding and classification of Dyson--Schwinger equations, now also serves as a combinatorial formalism for their solution. This way, one can avoid entirely the use of chord diagrams, which -- apart from looking graphically similar to the particular example of Yukawa Feynman graphs -- are alien to the Hopf algebraic theory of renormalization.  Furthermore, the tubing formulation allowed us to easily generalize beyond Dyson--Schwinger equations which had been solved by chord diagrams to all real $s$ and to systems.

From a combinatorial perspective, our results give a new notion of binary tubings of rooted trees, connect this notion to rooted chord diagrams, to previous notions of tubings which we called A-tubings, and to the Connes--Kreimer Hopf algebra.  Furthermore, we considered interesting statistics on these objects such as the $b$-statistic, which are nice combinatorially, but would likely not have occurred to anyone to study without the physics asking for them.

The following questions are left for future work:
\begin{enumerate}
    \item Consider the tubings of trees on $n$ vertices. In a certain sense, they are ``grouped'' on two levels: Firstly, one can count all those tubings into one set which arise from the same underlying rooted tree. Secondly, by \cref{lem recursive tubing count}, the number of tubings is independent of which vertex is the root, and thereby one could gather all those sets which arise from the same un-rooted tree upon choosing different roots. Is this structure represented in any way on the chord diagram side of the bijection? Similarly, when a chord diagram is drawn in circular form, we can call all those diagrams ``related'' which arise by choosing a different root chord. What does this operation do on the tubing side of the bijection?
    \item Can we obtain useful information on A-tubings by the bijection to binary tubings, as remarked in \cref{rem count A-tubings}?
    \item The bulk of \cref{sec bijection tubings chord diagrams} consisted of the construction of a recursively defined bijection $\theta$ from tubings to connected chord diagrams, but at the end of the section we noted the existence of an alternate bijection which has a non-recursive definition in the case of leaf tubings and is related to a map described more generally, albeit again recursively, in \cite{nabergall_combinatorics_2023, nabergall_enumerative_2022}.  We speculate that it is possible to adapt the recursive map $\alpha$ on connected chord diagrams constructed in \cite{nabergall_combinatorics_2023,nabergall_enumerative_2022} into a natural non-recursive bijection between tubings and connected chord diagrams extending $\kappa$, but leave this for later work.  
    \item Similar techniques to what are used here to solve Dyson--Schwinger equations by tubing expansions can be used to tackle multiple insertion places. This generalization appears in the fifth author's PhD thesis \cite[Chapters 3 and 4]{olson-harris_applications_2024} and will be further expanded upon in a future paper by some of the authors.
    \item We would like to understand the precise relationship between our tubing perspective and the $S\star Y$ operation that is central to \cite{broadhurst_combinatoric_2000} and will pursue this in future work.
\end{enumerate}

\appendix

\section{Tubings and graph associahedra}\label{sec associahedra}

This appendix is a comparison with other notions of tubings and related objects which exist. We want to stress that the definitions and results of our paper are motivated by a concrete application in quantum field theory, we do not aim to re-derive existing theorems. Instead, the present section serves to clarify the relations to similar constructions. 

As noted earlier, our notion of tubing is a special case of the notion of tubings or pipings of posets \cite{galashin_associahedra_2023}, which is used by Galashin to define \emph{poset associahedra} \cite{galashin_associahedra_2023, nguyen_poset_2023, nguyen_poset_2023a, sack_realization_2023}.

On the other hand, the most standard notion of tubing in the literature is the definition of \cite{carr_coxeter_2006}. Those tubings are defined for arbitrary finite graphs; they are described algebraically in \cite{forcey_algebraic_2019}, and an extension to colored tubings has been defined in \cite{devadoss_colorful_2023}.  
Those tubings furnish the vertices of \emph{graph associahedra}\footnote{We use the term \emph{associahedra} only for the classical case, based in this context on ladder graphs $\ell_n$, while the analogous objects based on more general graphs are called \emph{graph associahedra}.} \cite{stasheff_homotopy_1963,carr_coxeter_2006} (see \cref{figure assoiciahedron 1}) therefore we call them \emph{A-tubings} in order to distinguish them from our own notion of binary tubings (\cref{def tubing}). 

The connection between binary tubings (or poset tubings) and A-tubings has been used in \cite{nguyen_poset_2023}, but we think it is helpful to provide a more detailed yet explicit and elementary explanation of this connection.

\begin{figure}[htb]
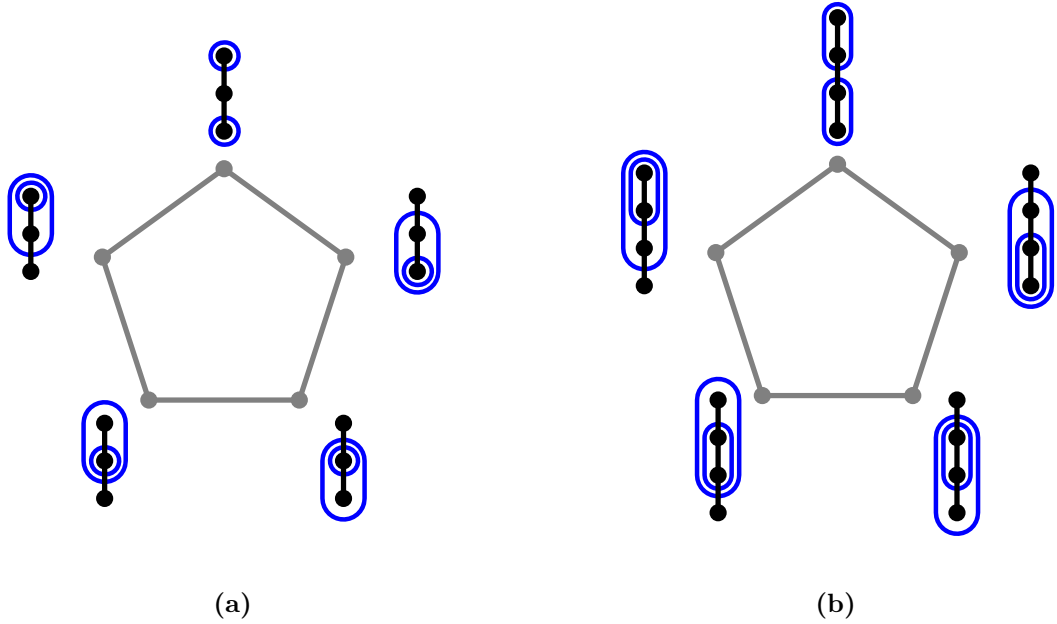

	\centering
	\begin{subfigure}[b]{.6\linewidth}
		\centering
		\input{figure_associahedron_1.tikzpicture}
		\vspace{.7cm}
		\caption{}
		\label{figure assoiciahedron 1}
	\end{subfigure}
	\begin{subfigure}[b]{.35\linewidth}
		\centering
		\input{figure_associahedron_2.tikzpicture}
		\caption{}
		\label{figure assoiciahedron 2}
	\end{subfigure}
	
	\caption{\textbf{(a)} The 2-dimensional associahedron (grey pentagon). The vertices represent the five different A-tubings of $\ell_3$, each of which consists of two tubes. The edges of the associahedron correspond to A-tubings where one of the tubes is removed (not shown in the picture) \textbf{(b)} The same associahedron, but decorated with the binary tubings of $\ell_4$, corresponding to the A-tubings of (a) via the map $L$ (\cref{def tubing L map}). As usual, we do not draw the innermost and outermost tubes. }
	
\end{figure}

\begin{definition}\label{def A-tubing}
	An \emph{A-tubing} is a set of tubes (\cref{def tube})  of a connected finite graph such that any two tubes are either nested, or they are disjoint and their union does not qualify as a tube. An $A$-tubing does not include a tube which contains all of the vertices of the graph\footnote{Note that the outermost tube, that is, the one which contains all vertices, is unique.  In the literature, there are different conventions regarding whether this tube is included in an A-tubing. Following   \cite{carr_coxeter_2006,devadoss_colorful_2023}, we choose to leave it out, although its presence would only require a trivial modification to the map $L$ defined below. Note also that, since no two tubes in an $A$-tubing have a union which could be a tube, the $A$-tubing does not contain tubes which together contain all vertices of the graph. That is, there is necessarily at least one vertex not contained in any tube.}.
	A \emph{maximal A-tubing} of a graph is an A-tubing where no more tubes can be added without violating the conditions of an A-tubing. 
\end{definition}
Recall that a binary tubing according to \cref{def tubing} is always maximal in the sense that no more tubes can be added. We have not used non-maximal binary tubings, though these also appear in the poset associahedron. Further, in \cref{def A-tubing}, a maximal A-tubing excludes the outermost tube which contains all vertices of a graph, but this is a mere convention difference compared to the binary tubings. When drawing these tubings, we will draw \emph{all} tubes of an A-tubing, and not leave out the innermost and outermost ones as we do for binary tubings.

The pivotal difference of an A-tubing to binary tubings is the last condition of \cref{def A-tubing}: In a binary tubing every tube of size $>1$ contains two smaller tubes whose union is the original tube.  In contrast, an A-tubing can never contain two tubes whose union is a tube, not even a tube which is not included in that particular tubing.

There is also a notion of nesting given in \cite[Section 3.2.2]{ward_massey_2022} and \cite{jonson_tubings_2018}.  Maximal nestings can be easily seen to be the same as our binary tubings using the fact that a connected subgraph of a tree is necessarily an induced subgraph.  Non-maximal nestings would correspond to exactly the partial binary tubings mentioned previously, that is to objects resulting from removing some finite number of tubes from a binary tubing.  Nestings in this sense are also known to give polytopes \cite{laplante-anfossi_diagonal_2022}.  Furthermore, the nestings and the A-tubings (and hence their polytopes) are connected via the \emph{line graph} \cite{ward_massey_2022, jonson_tubings_2018}.

For the reader who is interested in this polytope connection, we sketch this bijection via the line graph elementarily and in our notation, without requiring the reader to be familiar with the algebraic setting of many of the above papers, but we emphasize that this is not a new result.

First review the line graph construction from graph theory.  The line graph of a graph $G$ is the graph $L(G)$ whose vertices are the edges of $G$ and where two vertices are adjacent if they are incident to a common vertex in $G$. Note that edges incident to a $k$-valent vertex $v$ in $G$ become the vertices of a complete graph $K_k$ in $L(G)$ and so for an arbitrary simple graph $G$, the line graph $L(G)$ consists of complete graphs $K_n$ on $n \geq 1$ vertices, such that each vertex of the line graph is part of not more than two distinct complete graphs \cite{krausz_demonstration_1932}. In particular, the line graph $L(t)$ of a tree $t$ is not necessarily a tree itself, though the line graph of a ladder is a ladder. If a tree $t$ has $n$ vertices, then, by Eulers formula, $L(t)$ has $n-1$ vertices.

\begin{definition}\label{def tubing L map}
	Let $\tau$ be a binary tubing of a tree $t$. Then let $L(\tau)$ be an A-tubing of the line graph $L(t)$, constructed by the following algorithm: 
	\begin{enumerate}
		\item Leave out all those tubes of $\tau$ which contain only a single vertex.
		\item If $t$ is not a single vertex, leave out the tube of $\tau$ which contains all vertices of $t$.
		\item Every remaining tube of $\tau$ is a tube of the line graph $L(t)$, containing those vertices corresponding to the edges contained in the original tube. 
	\end{enumerate}
\end{definition}

\begin{remark}\label{rem L inverse}
	The inverse $L^{-1}$ of the map from \cref{def tubing L map} is given by the following algorithm: Let $L(t)$ be the line graph of a tree $t$ and  $\alpha$ be a maximal A-tubing of $L(t)$. 
	\begin{enumerate}
		\item For every tube $a\in \alpha$, draw a tube $b$ in $t$ which contains all edges, and their adjacent vertices, corresponding to the vertices contained in $a$.
		\item Add one tube around the whole tree $t$.
		\item If $t$ is not a single vertex, add one tube around every individual vertex in $t$.
	\end{enumerate}
\end{remark}

\begin{proposition}\label{prop L bijection}
	For each tree $t$, The map $L$ from \cref{def tubing L map} is a bijection between the binary tubings of a tree $t$ (\cref{def tubing}) and the maximal A-tubings (\cref{def A-tubing}) of the corresponding line graph $L(t)$. 
\end{proposition}
\begin{proof}
	One can check explicitly from their definitions that the two algorithms \cref{def tubing L map,rem L inverse} result in valid tubings according to \cref{def A-tubing,def tubing} and that they are each other's inverses.
	The proof is almost verbatim identical to \cite[Lemma 3.17]{ward_massey_2022} (see also \cite{jonson_tubings_2018}), therefore we skip the details here.
\end{proof}

\begin{example}\label{ex A tubings ladders}
	The ladders $\ell_n$ from \cref{sec ladders} provide a straightforward example of the map $L$ (\cref{def tubing L map}): The line graph of a ladder $\ell_n$ is the ladder $L(\ell_n) = \ell_{n-1}$.  By \cref{lem ladders tubings}, the number of binary tubings of $\ell_n$ is $N(\ell_n)=C_{n-1}$. At the same time, it is known \cite{carr_coxeter_2006} that $\ell_{n-1}$ has $C_{n-1}$ different A-tubings. 
	Indeed, by \cref{prop L bijection}, the binary tubings of $\ell_n$ correspond to the A-tubings of $\ell_n$.
	
	The A-tubings of $\ell_{n-1}$ furnish the vertices of the classical $(n-2)$-dimensional associahedron, shown in \cref{figure assoiciahedron 1} for the case $n=4$. Using the map $L$, the same associahedron can be decorated by binary tubings of $\ell_n$,   shown in \cref{figure assoiciahedron 2}.
\end{example}

\begin{remark}\label{rem count A-tubings}
	The central outcome of the present paper is that binary tubings (\cref{def tubing}) represent a solution to a Dyson--Schwinger equation. In particular, for the Yukawa Mellin transform \cref{eg yukawa mellin transform}, the solution is, up to overall signs, the generating function of binary tubings of a certain class of rooted trees. Using the bijection $L$, the solution can equivalently be viewed as a sum of maximal A-tubings (\cref{def A-tubing}) of line graphs of certain trees. Knowing the solution of the Yukawa DSE implies knowing the generating function of the maximal A-tubings of line graphs in the corresponding cases. 
	
	For the linear DSE (\cref{ex dse trees ladder}), the solution is given by the binary tubings of ladder trees, and we confirm that the Catalan numbers count the maximal A-tubings of ladders, see \cref{ex A tubings ladders}.
	
	Likewise, by \cref{lem chains tubings}, there are $n!$ binary tubings of a corolla tree on $(n+1)$ vertices. Using the bijection $L$, this means that there are $n!$ maximal A-tubings of a complete graph $K_n$.
	
	For the Yukawa DSE with $s=-2$, the Green function is the sum of binary tubings of all plane rooted trees (\cref{eg yukawa insertion trees}). The resulting function is known \cite{broadhurst_exact_2001}, and by the bijection $L$ it is the generating function of maximal A-tubings of all line graphs corresponding to plane rooted trees.  
	
	These examples illustrate that it is possible to obtain information about A-tubings of line graphs by the bijection to binary tubings. We did not investigate this possibility further. 
	
\end{remark}

\cref{figure associahedron large} shows the graph associahedron associated to a 5-vertex tree other than $\ell_5$.   Note that one consequence of \cref{prop L bijection} is that binary tubings of trees always define the vertices of a polyhedron, and reversing the $L$ map we see the other faces always correspond to sets of tubes which are partial tubings in the sense that they are the result of removing tubes from binary tubings.

\begin{figure}[htb]
	\centering
	\input{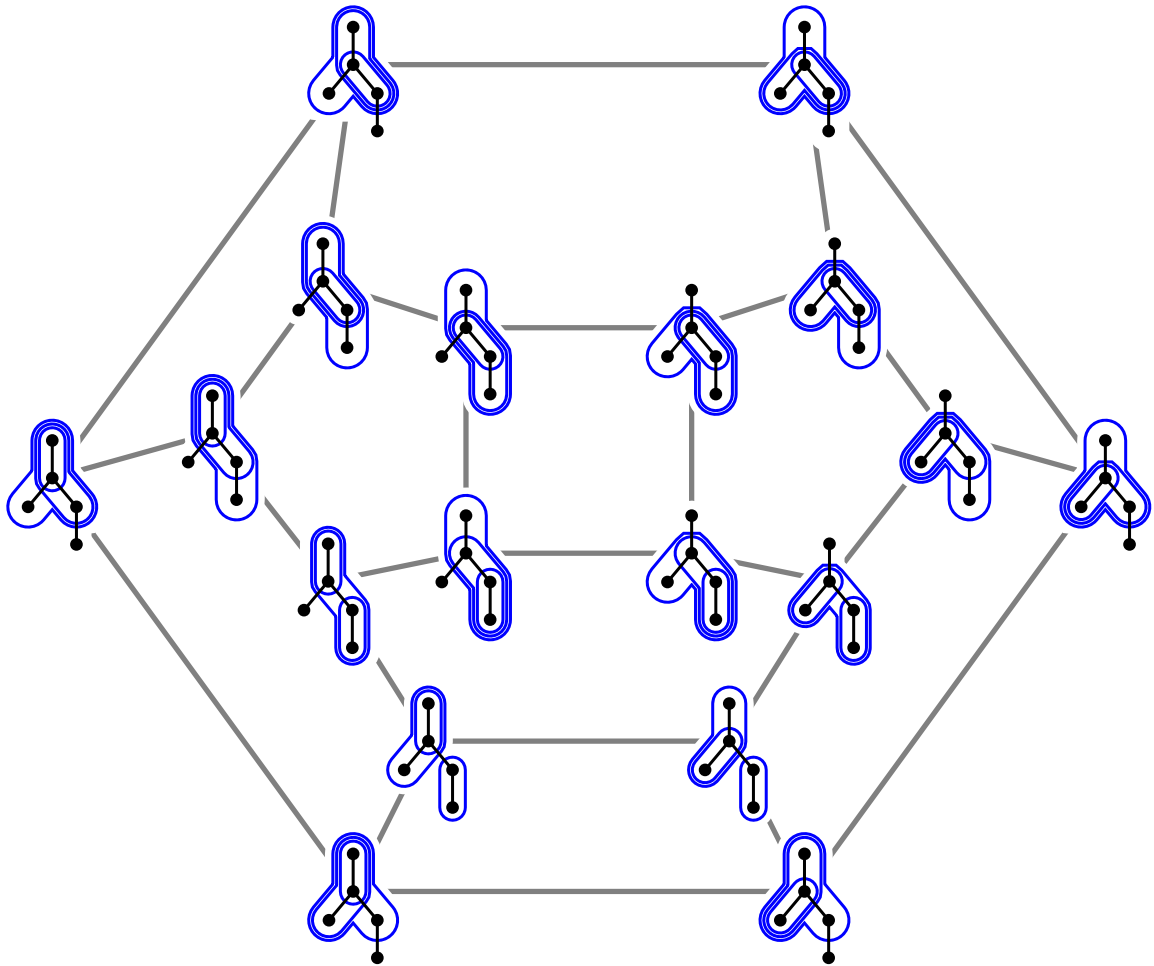}
	\caption{Projection of a 3-dimensional graph associahedron (gray), decorated with the 18 different binary tubings of the underlying 5-vertex tree.}
	\label{figure associahedron large}
\end{figure}

\FloatBarrier

\printbibliography

\end{document}